\documentclass[12pt]{amsart}
\usepackage{enumerate}
\usepackage{amssymb,amscd,amsthm}
\usepackage{color}
\usepackage{bbm} 
\usepackage[all]{xy}
\usepackage{tikz}

\usepackage{amsmath}
\usepackage{mathtools}
\usepackage{cleveref}

\newtheorem{thm}{Theorem}[section]
\newtheorem{theorem}[thm]{Theorem}
\newtheorem{lemma}[thm]{Lemma}
\newtheorem{prop}[thm]{Proposition}
\newtheorem{cor}[thm]{Corollary}
\newtheorem{corollary}[thm]{Corollary}

\theoremstyle{definition}
\newtheorem{definition}[thm]{Definition}

\newtheorem{defn}[thm]{Definition}
\newtheorem{example}[thm]{Example}
\newtheorem{remark}[thm]{Remark}
\newtheorem{question}[thm]{Question}
\newtheorem*{ThmA}{Theorem A}
\newtheorem*{ThmB}{Theorem B}
\newtheorem*{ConjC}{Conjecture C}
\newtheorem*{ConjD}{Conjecture D}

\newcommand\CC{\mathbb{C}}

\newcommand\RR{\mathbb{R}}
\newcommand\ZZ{\mathbb{Z}}
\newcommand\R{\mathbb{R}}
\newcommand\F{\mathbb{F}}
\newcommand\HH{\mathbb{H}}

\newcommand\minus{\smallsetminus}
\newcommand\nothing{\varnothing}

\DeclareMathOperator{\homeo}{Homeo}
\DeclareMathOperator{\Id}{Id}
\DeclareMathOperator{\interior}{int}
\DeclareMathOperator{\map}{Map}
\DeclareMathOperator{\rk}{rk}

\def\ol#1{{\overline {#1}}}

\newcounter{separated}

\newcommand{\commentjd}[1]
{}

\newcommand{\commentph}[1]
{}

\newcommand{\commentem}[1]
{}

\newcommand{\commenth}[1]
{}

\begin{document}

\title
{Fibers of maps to totally nonnegative spaces}
\author{James F.~Davis}
\address{Indiana University\\Bloomington, IN}
\email{jfdavis@iu.edu}
\author{Patricia Hersh}
\address{University of Oregon\\Eugene, OR}
\email{plhersh@uoregon.edu}
\author{Ezra Miller}
\address{Duke University\\Durham, NC}
\email{ezra@math.duke.edu}

\begin{abstract}
This paper undertakes a study of the structure of  the fibers of the Chevalley exponentiation  maps 
$f_{(i_1,\dots ,i_d)}$. 
The fibers  
of these maps $f_{(i_1,\dots ,i_d)}$  encode the nonnegative real relations amongst exponentiated Chevalley generators. 
Our main theorems show that the fibers admit  cell stratifications, 
that  these cell 
stratifications 
have the same face posets as 
 interior dual block complexes of subword 
complexes, and that these posets 
are contractible.
 
We conjecture 
 that each  such 
fiber is a regular CW complex homeomorphic to the interior dual 
block complex of a subword complex.  
This conjecture  is shown to have  as a corollary  
a new proof of the Fomin--Shapiro Conjecture by way of 
general topological results  
regarding  approximating maps by homeomorphisms.

 \end{abstract}
 
\thanks{PH was supported by NSF grants  
DMS-1200730,  DMS-1500987, and DMS-1953931.  
EM was supported by NSF grants  DMS-1001437 and DMS-1702395.  JFD was supported by NSF grants 
DMS-121099, DMS-1615056, and the Simons Foundation Collaboration Grant 713226.}
\maketitle

\section{Introduction}\label{intro-section}

Let $U$ be the unipotent radical of a Borel subgroup in a reductive, connected
algebraic group defined and split over $\RR $.  The totally nonnegative real part of  
~$U$ is stratified into Bruhat cells. 
Fomin and Shapiro \cite{FS} 
conjectured in the case that the associated Weyl group $W$  is finite 
that this stratification restricted to the link of the identity element in the totally nonnegative real part  of  ~$U$  
is a regular CW decomposition of a topological closed ball, namely a decomposition into pieces  (called cells) 
each homeomorphic to an open ball in such a way that the closure of each cell is homeomorphic to a closed ball.  
They proved that this stratified space has Bruhat order as the 
partial order of closure relations on its cells, and they obtained homological results 
supporting their conjecture.  

This conjecture from \cite{FS} 
was proven  by Hersh  \cite{He}  
in a way  that heavily relied upon  a 
realization of   these spaces due to Lusztig   as images of maps 
$f_{(i_1,\dots ,i_d)}$ that were  studied extensively by Lusztig in \cite{Lu}. 
 Indeed these maps 
$f_{(i_1,\dots ,i_d)}$,  which are indexed by  
words for Coxeter group elements   
are  interesting and fundamental  in their own right. 
A main goal of  the present paper,  which may be regarded
as a sequel to \cite{He}, 
 is to better understand the structure of the 
  fibers of  these maps $f_{(i_1,\dots ,i_d)}$. 

Much of the interest in these
maps $f_{(i_1,\dots ,i_d)}$,   and in some  closely related  
 maps as well, 
comes from a desire to understand  their fibers (see e.g.\  \cite{BZ}, \cite{BFZ},  \cite{BFZ2}, \cite{KW} and \cite{PSW}). 
 One motivation for studying the fibers of $f_{(i_1,\dots ,i_d)}$   is  
the  fact that these fibers describe exactly  the nonnegative real 
relations amongst exponentiated Chevalley generators.  Lusztig became interested in  
 $f_{(i_1,\dots ,i_d)}$ 
due to a connection  he observed to his theory of canonical bases: when  applying a braid move that converts  
a  reduced word  $(i_1,\dots ,i_d)$ for the longest element to another reduced word $(j_1,\dots ,j_d)$, this results in a change of basis that converts each  positive  real 
point in  the fiber $f^{-1}_{(i_1,\dots ,i_d)}(p)$ to a positive  real  point  in  the fiber 
$f_{(j_1,\dots ,j_d)}^{-1}(p)$; remarkably, this change of basis  is  
a tropicalized version of the change of basis needed for  
converting from  
  Lusztig's  canonical basis expressed  in terms of the reduced word $(i_1,\dots ,i_d)$   to Lusztig's canonical basis 
  working in terms of  $(j_1,\dots ,j_d)$  instead 
 (see e.g.\  \cite{BZ}, \cite{Lu-canon} for more on this result of Lusztig).  
The study of the maps $f_{(i_1,\dots ,i_d)}$  
also served as one of the inspirations  for the theory of cluster algebras, by virtue of  the fact that  this change of basis operation  is closely  
related to the notion of  mutation in cluster algebras.

We first define $f_{(i_1,\dots ,i_d)}$ in the type A  case, since this is particularly concrete and conveys the general idea fairly well.   In this case,  $G = GL_n(\R)$ and $U = U^n$ is the subgroup of upper triangular matrices with ones along the diagonal.   For $1 \leq j \leq n-1$, let $x_j : \R \to U$ be the group homomorphism $x_j(t) = I_n + tE_{j,j+1}$ where $E_{j,j+1}$ is the $n \times n$-matrix with all zeroes, except for a single 1  in the $(j,j+1)$-position. 
For $Q = (i_1,\dots ,i_d)$, define the map
\begin{align*}
 f_Q = f_{(i_1,\dots ,i_d)} & : \R^d_{\geq 0} \to U^n \\
 (t_1, \dots, t_d) & \mapsto x_{i_1}(t_1) \cdots x_{i_d}(t_d).
\end{align*}

\begin{example}
In $GL_3(\R )$ we have  
$$f_{(1,2,1)}(t_1,t_2,t_3)  = \left( \begin{array}{ccc}
1 & t_1 + t_3 & t_1t_2\\
0 & 1 & t_2\\
0 & 0 & 1
\end{array} \right) .$$
\end{example}

Let $U^n_{\geq 0}$ be the totally nonnegative part of $U^n$, in other words the set of  matrices in $U^n$ all of whose minors  
are nonnegative.   Then
$$f_Q : \R^d_{\geq 0} \to U^n_{\geq 0},$$
since $x_j(t) \in U^n_{\geq 0}$ for $t \geq 0$ and $U^n_{\geq 0}$ is closed under multiplication.
\commenth{Next sentence just added (though there is an example later saying almost this exact same thing -- but it seems helpful for readers here):}  In this case, one may readily verify that  the braid relations $$x_i(t_1)x_{i+1}(t_2)x_i(t_3) = x_{i+1}\left(\frac{t_2t_3}{t_1+t_2}\right)x_i(t_1+t_3)x_{i+1}\left(\frac{t_1t_2}{t_1+t_3}\right)$$ hold  for any 
$t_1,t_2,t_3>0$ and any $i=1,2,\dots ,n-1$ as do   the commutation relations $x_i(t_1)x_j(t_2) = x_j(t_2)x_i(t_1)$ for any  $t_1,t_2\ge 0$ and any $|j-i|>1$; there are also  
relations $x_i(t_1)x_i(t_2) = x_i(t_1+t_2)$ called  ``modified nil-relations'' for each  $i\in \{ 1,2,\dots ,n-1\} $ and each  $t_1,t_2\ge 0$.  

All of this  
  generalizes to the context of the unipotent radical $U$ of a Borel subgroup in a reductive, connected algebraic group $G$ defined and split over $\R$.
  We briefly sketch this theory; for more details see Section \ref{positivity-background}.    
The Weyl group of $G$ is  a finitely generated 
 Coxeter group $W$ with a finite set $\Sigma = \{s_1, \dots, s_n\} $ of simple reflections as generators.    We  assume throughout this paper that 
 $W$  is a finite group. 
We say that $Q = (i_1,\dots ,i_d)$ is a word for $w \in W$ if  $w = s_{i_1} \dots s_{i_d}$, and we call  such $Q$ a reduced word for $w$ if $d$ is minimal.  

Associated to each  simple reflection   $s_i\in W$  is  a group homomorphism $x_i : \R \to U$ with  $x_i(t)$ an exponentiated Chevalley generator (see e.g.\ \cite{Lu}).   Given a word  $Q = (i_1,\dots ,i_d)$, where $1 \leq i_j \leq n$ for $j=1,\dots ,d$, Lusztig considers the map 
\begin{align*}
 f_Q & : \R^d_{\geq 0} \to U \\
 (t_1, \dots, t_d) & \mapsto x_{i_1}(t_1) \cdots x_{i_d}(t_d),
\end{align*}
and defines the totally nonnegative part $U_{\geq 0}$ to be the subset of $U$ generated (as a monoid) by $x_j(t)$ for $t \geq 0$.  He also  notes that by results of \cite{Lo} and \cite{Wh} in type A this is equivalent to the usual definition for the totally nonnegative part of $U_{\geq 0}$ in that case.  

Lusztig proves many interesting things about the map  $f_Q$ when $Q$ is a reduced word; let  us now describe three of these things.   If $Q$ and $Q'$ are two different reduced words for the same element $w \in W$, then the images of $f_Q$ and $f_{Q'}$ coincide.   If $Q$ is  a reduced word, then $f_Q : \R^d_{>0} \to U$ is an embedding.   If $Q$ is a reduced word for the longest element $w_0$ of $W$, then the image of $f_Q:\RR_{\ge 0}^d\rightarrow U$ is $U_{\geq 0}$.    
 
\subsection{Main results}\label{main-result-section}

We now discuss the main results of our paper, namely the construction of a cell stratification for each fiber  of 
$f_Q$ (see Theorem A in Section ~\ref{A-section},  proven in  Theorem  ~\ref{p:cells}) 
and a combinatorial analysis of the partially ordered set  (poset)  of closure relations on the  strata (see Theorem B in Section ~\ref{B-section},  proven in  Propositions~\ref{p:fiberNerve} and~\ref{c:dualBlock}). 
In discussing our  results, 
we emphasize  the parallels that exist  between our results on fibers of $f_Q$ 
and previous results of Lusztig \cite{Lu}, Fomin-Shapiro \cite{FS}, Hersh \cite{He}, and Galashin-Karp-Lam \cite{GKL19} on the images of the maps $f_Q$.

First recall that a  {\em decomposition} of a topological space $X$ is a collection $\{X_\alpha \}_{\alpha \in I}$ of  subsets  whose union is $X$ such that any pair of sets are disjoint or coincide.     A {\em stratification} of $X$ is a decomposition in which $X_\alpha \cap \overline{X_\beta} \not = \emptyset$ implies $X_\alpha \subseteq\overline{X_\beta}$.  A {\em $d$-cell} is a topological space homeomorphic to the interior of the $d$-ball.  We include the case of a $0$-cell which is a point.  
 A {\em cell} is a $d$-cell for some $d$.  
  A {\em cell decomposition (respectively, cell stratification)} is a decomposition (respectively, stratification) where each $X_\alpha$ is a cell.

\commenth{We need to be careful  that most of the time  we are  intersecting $f_{(i_1,\dots ,i_d)}^{-1}(p)$ with $\RR_{\ge 0}^d$.
 Hopefully I have resolved this by now.}

There is a  decomposition of the nonnegative part $U^n_{\geq 0}$ of the unipotent  matrix 
group $U^n$ that can be described in three different ways.  First, one can specify a stratum by specifying a  subset of all generalized minors (as defined in \cite{FZ99})  and requiring that the generalized minors in the chosen  subset are positive and that  all other generalized minors are zero. 
 In type A, the generalized minors are exactly the  usual  matrix minors.   
The second description of strata (which only makes  sense in finite type)   is as follows.
 Let $Q = (i_1, \dots, i_d)$ be a reduced word for the longest element  $w_0 \in W$.   Then one can specify  a subset $S \subseteq \{1, \dots, d\}$ and define the stratum indexed by $S$ as 
$$
\{f_Q(t_1, \dots, t_d) \mid t_j > 0 \text{ for } j \in S \text{ and } t_j = 0 \text{ for } j \not \in S \} \subset U_{\geq 0}.
$$
Two different subsets give rise to the same stratum if and only if the subwords of $Q$ they index have the same Demazure product  
(see Definition ~\ref{d:demazure} and Lemma ~\ref{0-Hecke-to-Demazure}) as each other.
The third definition proceeds by representing an element $w \in W$ by a reduced word $R = (j_1, \dots, j_e)$ for $w$  and defining the
stratum   $U(w) = f_R(\R^e_{> 0})$ for $w\in W$, a notion which turns out not to  depend on the choice of reduced word for $w$.
This definition has the advantage that each $U(w)$ is nonempty and  that  $w \not = w'$ implies 
 $U(w)$ and $U(w')$ are disjoint. 
The fact that  all three definitions give the same decomposition follows from results   in \cite{FZ},  \cite{FZ99}, \cite{FS}  and \cite{Lu}.  See Lemma 2.10 in \cite{FS} for a summary of how the first definition relates to the other two and see \cite{Lu} for the  equivalence of the other two.

\subsection{Cell stratification and other topological results}\label{A-section}

One of the main accomplishments of Lusztig's paper \cite{Lu} is that the above decomposition $\{U(w)\}$ gives a cell stratification of $U_{\geq 0}$.   Likewise
one of the  main results of our work is:

\begin{ThmA}
For any word $Q$ of length $d$, any $w\in W$, and any $p\in U(w)$, 
$$
\{  f^{-1}_Q(p) \cap \R^S_{>0 }\}_{S \subseteq \{1,\dots, d\}}
$$ 
gives a cell stratification of  the nonnegative real part of  $f^{-1}_Q(p)$, namely of 
 $f^{-1}_Q(p) \cap \RR_{\ge 0}^d$, 
letting 
$\R^S_{>0}$  denote the subset of $\R^d_{\ge 0}$ whose strictly positive parameters are  exactly those indexed by $S$.  
\end{ThmA}

Theorem A is proven in Theorem~\ref{p:cells}.  This in turn is a consequence of the more technical Theorem~\ref{thm:CW} (in combination   
with Lemma ~\ref{decomp-is-strat} which proves that the cell decomposition given in  
Theorem ~\ref{thm:CW} is a cell stratification).   
Theorem ~\ref{thm:CW}  pulls together nearly all of our topological  
 results concerning the fibers, in particular, most of the results of  Section~\ref{s:CW}, making it one of the main 
 achievements in our paper.  
 
 Other structural  results of Section ~\ref{s:CW} that may be of independent interest  include:

\begin{enumerate} 
\item
 a generalization (see Theorem ~\ref{generalized-Lusztig-homeo}) of Lusztig's result  from \cite{Lu}  that the map 
 $f_{(i_1,\dots ,i_d)}$ given by any reduced word $(i_1,\dots ,i_d)$   is a homeomorphism from  $\RR_{>0}^d$ to $Y_{\delta (i_1,\dots ,i_d)}^o$, and 
\item
a recursive result (see Lemma ~\ref{full-theorem}) showing that certain fibers are homeomorphic to other fibers given by strictly shorter words.
\end{enumerate}

\subsection{Determination of poset of closure relations and other combinatorial results}\label{B-section}

Fomin and Shapiro showed in \cite{FS} that the poset of  closure relations on  nonempty cells  given by the cell stratification $\{U(w)\}_{w \in W}$ of $U_{\geq 0}$  is the  Bruhat order on the Weyl group $W$.   
 Inspired in part by the results of Bj\"orner and Wachs we describe next, they also conjectured that this cell stratification is a regular CW decomposition in [\cite{FS}, Conjecture 1.10],  a conjecture we henceforth call the Fomin-Shapiro Conjecture.  
%
Bj\"orner and Wachs \cite{BW} showed for $W$ a finite group that this poset is thin and shellable.  This combines with  general results of Bj\"orner \cite{Bj} to imply that the Bruhat order is the partial order of closure relations on the cells of a regular CW complex $K$ homeomorphic to a closed ball.    

\begin{example}
The strata in Figure 1 are given by the subwords $Q$ of $(1,2,1,2,1,2)$ having Demazure product $s_1s_2$  (in the sense of 
Definition ~\ref{d:demazure} and Lemma ~\ref{0-Hecke-to-Demazure}), with the dimension of the 
stratum  given by such $Q$ equalling $|Q|-l(s_1s_2) = |Q|-2$ where $|Q|$ denotes the number of letters in $Q$.  There are six  strata of dimension 0 given by the six subwords of 
$(1,2,1,2,1,2)$ that are reduced words for $s_1s_2$, eight strata of dimension 1 given by the non-reduced subwords of the form $(1,1,2)$ and $(1,2,2)$ 
and three  strata of dimension 2 given by the  non-reduced subwords of the form $(1,1,2,2), (1,1,1,2)$ and $(1,2,2,2)$.
\end{example}

 \begin{figure} 
\begin{tikzpicture}[scale=0.5]
%

\path [fill=lightgray] (2,8) -- (4,4) -- (8,7) -- (2,8);
\path [fill=lightgray] (4,4) -- (8,7) -- (14,8) -- (10,5) -- (4,4);
\path [fill=lightgray] (10,5) -- (4,4) -- (8,2) -- (10,5);

\draw (2,8)--(4,4)--(8,2)--(10,5)--(14,8);
\draw (2,8)--(8,7)--(14,8);
\draw (4,4)--(8,7);
\draw (4,4)--(10,5);

\node [below] at (8,2){\scriptsize $(0,0,0,0,a,b$)}; 
\node [below left] at (4,4){\scriptsize $(a,0,0,0,0,b$)}; 
\node [above left] at (2.1,7.8){\scriptsize $(a,b,0,0,0,0)$}; 
\node [above] at (8,7.2){\scriptsize $(a,0,0,b,0,0)$}; 
\node [above right] at (14,7.8){\scriptsize $(0,0,a,b,0,0)$}; 
\node [below right] at (10,5.1){\scriptsize $(0,0,a,0,0,b)$}; 

\end{tikzpicture}
%
%
%
\caption{Cell stratification in type A  for $f^{-1}_{(1,2,1,2,1,2)}(M)\cap \RR_{\ge 0}^6$  for any fixed $M = x_1(a)x_2(b) \in U(s_1s_2)$, i.e., for any fixed  choice of  $a,b>0$}
\end{figure}

 In  analogy with  the work of  Bj\"orner and Wachs determining  the homotopy type  
 of each open interval in Bruhat 
 order, namely in  the poset of closure relations for the image of $f_Q$,  we will likewise determine the homotopy type  
 of the face poset   for each of  the fibers of $f_Q$.  
Doing so 
involves first determining the partial order of closure relations on the strata for each fiber, a  result that parallels work of 
Fomin and Shapiro regarding  the face poset for the Bruhat  stratification of the  image of $f_Q$.
  The subword complexes of Knutson and Miller and their interior dual block complexes
 play an important role in this combinatorial  aspect  of understanding the fibers of $f_Q$, bringing to clear view 
 ideas first hinted at in work of \cite{AH}.   

\begin{ThmB}
For any word $Q = (i_1, \dots, i_d)$, 
 the face poset for the cell stratification  for  $f_{(i_1,\dots ,i_d)}^{-1}(p) \cap \RR_{\ge 0}^{d}$   induced by the  natural 
 cell stratification  of $\RR_{\ge 0}^d $ (given by specifying for each stratum which coordinates are positive and which are 0)
 is isomorphic to  the face poset  of the  interior dual block complex of the 
subword complex $\Delta (Q,w)$ for $p\in U(w)$.  
Furthermore,  the interior dual block complex of any nonempty subword complex $\Delta (Q,w)$  is 
 contractible.
\end{ThmB}

Theorem B is proven in Propositions~\ref{p:fiberNerve} and~\ref{c:dualBlock}.
This bridge to subword complexes allows us to use gallery-connectedness and other properties of subword complexes to prove in 
Lemma ~\ref{intersect-with-simplex}  for
any  $p\in U(\delta (i_1,\dots ,i_d))$  that $$f_{(i_1,\dots ,i_d)}^{-1}(p)\cap \RR_{\ge 0}^d = f_{(i_1,\dots ,i_d)}^{-1}(p)\cap \Delta_K^{d-1}$$ 
for some $K>0$ where $\Delta_K^{d-1}$ denotes the $(d-1)$-simplex consisting of the points $(t_1,\dots ,t_d)\in \RR_{\ge 0}^d$ whose 
coordinates sum to $K$.  

\begin{remark}\label{non-pure-example}
One might hope that each  fiber $f_{(i_1,\dots ,i_d)}^{-1}(p)\cap \RR_{\ge 0}^d$ 
is a closed ball, or at least is  ``pure'', namely that all of its maximal strata (in its stratification induced by the natural cell stratification of $\RR_{\ge 0}^d$)
have the same dimension as each other.   However, there are counterexamples 
to both of  these statements.  

For example, consider
$(i_1,\dots ,i_d) = (1,3,2,1,3,2)$  for  any choice of  $p\in U(w)$ 
for  $w = s_1 s_3 s_2$.  
Theorems A and B  together imply that  the maximal cells are not all of the same dimension as each other.  See Figure 2 
 for the subword complex (in light grey) and its interior dual block complex (in dark grey), the latter of which has exactly  the cells in our stratification for 
$f_{(i_1,\dots ,i_d)}^{-1}(p)\cap \RR_{\ge 0}^d$.  The five vertices in the interior dual block complex in Figure 3  are labeled with the five subwords 
of $(1,3,2,1,3,2)$ that are reduced words for $s_1s_3s_2$.  
\end{remark}

 \begin{figure} \label{interior-dual-figure}
\begin{tikzpicture}[scale=0.8]

\path [fill=lightgray] (0,0) -- (0,4) -- (4,4) -- (6,2) -- (4,0) -- (0,0);
\path [fill=gray] (1,2) -- (2,3) -- (3,2) -- (2,1) -- (1,2); 

%
\draw (0,0)--(0,4)--(4,4)--(0,0); 
\draw (0,0)--(4,0)--(6,2)--(4,4);
\draw (4,0)--(0,4);
\draw (4,0)--(4,4);

\draw (1,2)--(2,1)--(3,2)--(5,2);
\draw (1,2)--(2,3)--(3,2);



\node [left] at (1,2){\scriptsize $(-,-,-,1,3,2)$}; 
\node [below] at (2,1){\scriptsize $(-,3,-,1,-,2)$}; 

\node [above] at (2,3){\scriptsize $(1,-,-,-,3,2)$}; 

\node [above] at (3.4,2){\scriptsize $(1,3,-,-,-,2)$}; 

\node [right] at (5,2){\scriptsize $(1,3,2,-,-,-)$}; 

\draw [fill] (1,2) circle [radius=0.05]; 
\draw [fill] (2,1) circle [radius=0.05]; 
\draw [fill] (3,2) circle [radius=0.05]; 
\draw [fill] (2,3) circle [radius=0.05]; 
\draw [fill] (5,2) circle [radius=0.05];

\end{tikzpicture}
%
%
%
\caption{The subword complex $\Delta (Q,w)$ and its interior dual block complex  for $Q=(1,3,2,1,3,2)$ and $w=s_3s_1s_2$}
\end{figure}

\subsection{An illustrative example}\label{example-and-where-proven}
Our topological results described in Section ~\ref{A-section} and proven in Section ~\ref{s:CW} may  be combined to deduce 
that the  portion of the fiber $f^{-1}_{(1,2,1,2,1)}(M)\cap \RR_{\ge 0}^5$ depicted (up to homeomorphism)  by the interior of the  region  in Figure 3 is homeomorphic to 
$$\bigg\{ (t_1,t_2) \in \RR^2 \mid 0 < t_1 < 5 \hspace{.07in}{\rm and} \hspace{.07in} 0 < t_2 <  \frac{7}{12-t_1} \bigg\} $$  
which is homeomorphic to $(0,1)^2$  and that each stratum depicted in Figure 3 is homeomorphic to 
$(0,1)^{s}$ for some $s\ge 0$.   

 \begin{figure} 
\begin{tikzpicture}[scale=0.6]

\path [fill=lightgray] (0,0) -- (4,0) -- (4,2.5) -- (2.5,4) -- (0,4) -- (0,0);

%
\draw (0,0)--(0,4)--(2.5,4); 
\draw (0,0)--(4,0)--(4,2.5)--(2.5,4);



\node [left] at (0,0){\scriptsize $(0,0,5,1,7)$}; 
\node [left] at (0,4){\scriptsize $(0,\frac{7}{12},12,\frac{5}{12},0$)}; 

\node [right] at (4,0){\scriptsize $(5,0,0,1,7)$}; 

\node [right] at (4,2.5){\scriptsize $(5,1,0,0,7)$}; 

\node [right] at (2.5,4.2){\scriptsize $(5,1,7,0,0)$};

\draw [fill] (0,0) circle [radius=0.05]; 
\draw [fill] (0,4) circle [radius=0.05]; 
\draw [fill] (4,0) circle [radius=0.05]; 
\draw [fill] (4,2.5) circle [radius=0.05]; 
\draw [fill] (2.5,4) circle [radius=0.05]; 

\end{tikzpicture}
%
%
%
\caption{The fiber $f^{-1}_{(1,2,1,2,1)}(M)\cap \RR_{\ge 0}^5$ up to homeomorphism for 
$M = x_1(5)x_2(1)x_1(7) \in  U(s_1s_2s_1)$ in type A}
\end{figure}

Specifically, Corollary ~\ref{closed-interval}  shows that certain parameters $t_i$ amongst  
$\{ t_1,t_2\}$ each  take a closed interval $[0,t_i^{\max }(k_1,\dots ,k_{i_1})]$ of possible values within  the entire fiber once values $k_1,\dots ,k_{i-1}$  for the  parameters $t_1,t_2,\dots ,t_{i-1}$  have already  been specified, Lemma ~\ref{redundant-non-max-which-delta} characterizes these maximal values $t_i^{\max }(k_1,\dots ,k_{i-1})$,  Lemma ~\ref{defined-forced-value} shows that each of the other parameters $t_j$  takes a unique value that is a function of the values  $k_1,\dots ,k_{j-1}$ chosen for the parameters  $t_1,\dots ,t_{j-1}$, Lemma ~\ref{other-char} describes which parameters take  a range of values and which  have  a uniquely forced value,  
Lemma ~\ref{0-iff-max-to-left} fits these and other results  together to give the structure of 
 the interior of the fiber,  Theorem ~\ref{thm:CW}  proves that there is a homeomorphism $h$  from  $[0,1)^s$ to  a union of strata 
   with $s$ equalling  the number of parameters that take ranges of values, and finally  Theorem ~\ref{p:cells}  shows how the 
   homeomorphism $h$  
    restricts to a homeomorphism from $(0,1)^s$ to the largest stratum, which in  our example is the interior of the region in 
    Figure 1.  We have  $s=2$ with  $t_1^{\max } = 5$  and  $t_2^{\max } =  \frac{7}{12-t_1}$ in our example.
This upper bound of 5 on  $t_1$  is determined  using 
Corollary ~\ref{any-l}, while Lemma ~\ref{can-tune-down-parameters}  yields the upper bound of $\frac{7}{12-t_1}$ on $t_2$.  The 
upper bound value of $\frac{7}{12-t_1}$ for $t_2$  may be calculated  using the braid relation
$$x_1(5-t_1)x_2(1)x_1(7) = x_2\left(\frac{1\cdot 7}{(5-t_1)+ 7}\right)x_1((5-t_1)+7)x_2\left(\frac{(5-t_1)\cdot 1}{(5-t_1) + 7}\right),$$ with our desired upper
bound for $t_2$ appearing as the leftmost parameter on the right side of this braid relation; the applicability of this calculation is justified 
by the 
equivalence of $x_1(t_1)x_2(t_2)x_1(t_3)x_2(t_4)x_1(t_5)  = M$ to the equation $$x_2(t_2)x_1(t_3)x_2(t_4)x_1(t_5) = x_1(-t_1)M = x_1(5-t_1)x_2(1)x_1(7)$$ for $0\le t_1 <  5$ and by the fact that $t_5$ is forced to be 0 when $t_1$ does not take its maximal possible value but $t_2$ does take its maximal possible value.      As a word of caution, notice for  $t_1=5$ and $t_2=1$ that  the value for the parameter $t_3$ is not uniquely determined by $t_1$ and $t_2$, whereas for each choice of values $k_1,k_2$ for $t_1,t_2$ with $0\le k_1 < t_1^{\max } = 5$ and $0\le k_2 < t_2^{\max }(k_1) = \frac{7}{12-k_1}$ the values of $t_3,t_4,t_5$ are uniquely determined.  

In case it helps readers better visualize  this example, 
$f_{(1,2,1,2,1)}^{-1}(M)\cap \RR_{\ge 0}^5$ is exactly 
$$ \{ (t_1,t_2,t_3,t_4,t_5)\in \RR_{\ge 0}^5 \mid 
t_1+t_3+t_5=12\hspace{.075in}{\rm and} \hspace{.075in} t_2 + t_4 = 1\hspace{.075in}{\rm and}\hspace{.075in}  t_1t_2 + t_1t_4 + t_3t_4 = 5\} ,$$
as one may verify  by equating entries in the matrix  equation $f_{(1,2,1,2,1)}(t_1,t_2,t_3,t_4,t_5) = M$.  It follows that  $t_4$ and $t_5$ are each functions of $t_1,t_2$ and $t_3$, allowing us to  project  to the first three coordinates.   Substituting $t_4 = 1-t_2$ simplifies the nonlinear equation  into   $t_1 + t_3 - t_2t_3 = 5$. 
In this projection to $\RR_{\ge 0}^3$, one may show that  the boundary of the  fiber $f_{(1,2,1,2,1)}^{-1}(M) $ 
includes the pair of 0-cells  $(0,0,5)$ and $(5,0,0)$ as well as  the straight line segment between them, and likewise includes the straight  line segment from the 0-cell $(5,0,0)$ to the 0-cell  $(5,1,0)$, as well as the straight line segment from  the  0-cell $(5,1,0)$ to the  0-cell $(5,1,7)$.  It is also possible to show that the boundary of the fiber includes the curve from $(0,0,5)$ to the 0-cell  $(0,\frac{7}{12},12)$ given by the pair of equations $t_1 = 0 $ and $0 = 5-t_3 + t_2t_3$.    Finally, the boundary  includes the curve from $(0,\frac{7}{12},12)$ to $(5,1,7)$ given by the  pair of equations $7 = t_2t_3$ and $t_1 + t_3 = 12$.   

The 0-cells in the fiber  are indexed by the subwords of $(1,2,1,2,1)$ that are reduced words for 
$s_1s_2s_1$, since for each such reduced subword  there is a unique choice of nonnegative values for the parameters in the  positions providing the support for the reduced word  when all  other parameters are set to 0.   The 1-cells are indexed  by the subwords  of $(1,2,1,2,1)$ with Demazure product $s_1s_2s_1$  having four letters; in this case, the parameter given by the letter that is absent from the subword  is set  to 0, and there is a curve of possibilities for the 4-tuple of nonzero parameters, with this curve going from one 0-cell to another 0-cell in the fiber.    Our combinatorial results described in Section ~\ref{B-section}  and proven in Section ~\ref{sub:whole-fiber} predict  this set of five 0-cells and five  1-cells  for the boundary of  this  fiber.

\subsection{A pair of conjectures}

Next we  formulate a pair of conjectures regarding the structure of the fibers of $f_{(i_1,\dots ,i_d)}$.  
 If verified, these  would  also yield  a new proof of the Fomin-Shapiro Conjecture, as  justified and explained in 
depth in Section ~\ref{s:fsConj}.

\begin{ConjC}\label{ConjC}
Any fiber $f_{(i_1, \dots, i_d)}^{-1}(p)\cap \RR_{\ge 0}^d$ is either empty or contractible.  
\end{ConjC}

Conjecture C is a consequence of Conjecture D, stated shortly.   To understand better the statement of Conjecture D and how it is connected to  Conjecture  C, we
first state a couple of facts.    Using results of Lusztig discussed in the background section, it  is straightforward to show
that a fiber $f_{(i_1,\dots ,i_d)}^{-1}(p)\cap \RR_{\ge 0}^d$  is nonempty if and only if $p\in U(\delta (i_1,\dots ,i_d))$ where $\delta $ is the Demazure product (defined in Proposition ~\ref{p:demazure} and the discussion thereafter).
By Lemma ~\ref{intersect-with-simplex}, our cell stratification for $f_{(i_1,\dots ,i_d)}^{-1}(p)\cap \RR_{\ge 0}^d$  from 
Theorem B 
 may be regarded as the  
cell stratification for $f_{(i_1,\dots ,i_d)}^{-1}(p)\cap \Delta_K^{d-1}$  induced by the simplicial cell stratification of the 
simplex $\Delta_K^{d-1} = \{ (t_1,\dots t_d)\in \RR_{\ge 0}^d \mid  \sum_{i=1}^d t_i = K\}$, allowing us now  to formulate Conjecture D.

\begin{ConjD}\label{t:subword}
For any word $Q = (i_1, \dots, i_d)$ and any $p\in U(\delta (i_1,\dots ,i_d))$, 
 the cell stratification  for  $f_{(i_1,\dots ,i_d)}^{-1}(p) \cap  
  \RR_{\ge 0}^d$ 
   induced by the simplicial cell stratification  
of  the simplex   $\Delta^{d-1}_K$ 
 is  a
regular CW complex. 
\end{ConjD}

Given that we have shown in Theorem B that the closure  poset for  the cell stratification of $f_{(i_1,\dots ,i_d)}^{-1}(p) \cap \RR_{\ge 0}^{d}$  equals the face poset of the regular CW complex given by the interior dual block complex of the subword complex,  and given that we have also determined  the topological structure of the interior dual block complex of any subword complex, Conjecture D implies that there is a cell preserving homeomorphism between $f_{(i_1,\dots ,i_d)}^{-1}(p) \cap \RR_{\ge 0}^{d}$ 
and 
the interior dual block complex given by $(i_1,\dots ,i_d)$ and $w\in W$ for $p\in U(w)$.   
Hence Conjecture D implies Conjecture C.  

Our results in Section  ~\ref{s:CW} provide what we believe should be significant  steps towards proving Conjecture D, in that the cell stratification given there is the decomposition of $f_{(i_1,\dots ,i_d)}^{-1}(p)\cap \Delta^{d-1}_K$  
that we are conjecturing is a regular CW decomposition in Conjecture D.  Many of the cell incidences are shown to be regular
incidences in Section ~\ref{s:CW}.  
Remark ~\ref{tying-together} explains in more detail  how the results in  Sections ~\ref{sub:whole-fiber}  and ~\ref{s:CW} provide steps towards a proof of 
Conjecture D (and thereby also   towards a new  proof of the Fomin-Shapiro Conjecture).  

\subsection{Fomin-Shapiro Conjecture and related prior work} 

We next describe the Fomin-Shapiro Conjecture as well as  some key known  results regarding the images  
of the map $f_{(i_1,\dots ,i_d)}$ and the images  of other closely related maps.  
The fact that  every finite CW complex is compact leads us now to  turn our attention to the link of the identity in $U_{\geq 0}$.
Let $Q$ denote  a reduced word for the longest element in $W$,  and let $d$ be the length of  $Q$.   Then the link of the identity of $U_{\geq 0}$ is defined to be $f_Q(\Delta^{d-1})$ where $\Delta^{d-1}$ denotes the $(d-1)$-simplex consisting of those points $(t_1,\dots ,t_d)\in \RR_{\ge 0}^d$ whose coordinates sum to 1.  
 The link still has a cell stratification by virtue of the cone structure of the map $f_Q$ resulting from the fact that $f_Q$ commutes with scaling (up to a simple homeomorphism); for example, in type A the impact of scaling each  of the parameters $t_1,\dots ,t_d$ by the same  positive real $k$ is that of scaling each of  the entries $M_{i,i+j}$  in the matrix $M := f_Q(t_1,\dots ,t_d)$  by $k^j$.  The face poset for the link (including the empty cell) is the Bruhat order on the Weyl group.    That is, the strata in the link are defined to be the sets  $U(w)\cap f_Q(\Delta^{d-1})$ indexed by the elements $w\in W$.
     For  $w \in W$, define $Y_w^o = U(w) \cap f_Q(\Delta^{d-1})$ 
     and $Y_w = \overline{U(w)} \cap f_Q(\Delta^{d-1})$.

The theories and the results of Lusztig, Bj\"orner, and the team of Bj\"orner and Wachs all provided pieces of a story that Fomin and Shapiro pieced together, leading Fomin and Shapiro to 
give a cell decomposition $\{Y_w^o\}_{w \in W}$ of the link of the identity in $U_{\geq 0}$
and conjecture that this cell decomposition 
 is a regular CW complex with Bruhat order as its poset of closure relations.    Hersh proved this conjecture known as the 
 Fomin-Shapiro Conjecture in \cite{He} where  she also deduced the  
 corollary that the link of the identity  in $U_{\ge 0}$ 
is homeomorphic to a closed ball.
In type A,  a different proof of the homeomorphism type of the closure of the big cell  for the 
image of $f_{(i_1,\dots ,i_d)}$ 
for the special  case of $(i_1,\dots ,i_d)$ a reduced word for 
longest element 
was later  given  in \cite{GKL}.  
Then Galashin, Karp and Lam made a major advance in this area in  \cite{GKL19}, proving that all cell closures are homeomorphic to closed balls for the totally nonnegative real part of any partial  flag variety.  This included  the case of the Grassmannian which had been conjectured by Postnikov.   Their results also gave  a second proof of the Fomin-Shapiro Conjecture.  In contrast to the more elementary approach of \cite{He},  the proofs in \cite{GKL19}  relied on the Poincar\'e Conjecture. 

 In each of these papers   of Hersh and 
of Galashin, Karp and Lam, the focus was on understanding the images of the maps.  These papers all  left  completely  open the question of understanding the fibers of these maps.  
Theorems A and B  here  shed  light on the topological  structure and the  combinatorial  structure, respectively,  of the fibers. 
Conjectures C and D suggest much more structure that we believe is in fact present in the fibers.

\subsection{Brief layout of the paper} 

In Section ~\ref{s:topology}, 
 we review  background,  doing so in a way that aims 
 to  make this paper accessible to readers coming  from an assortment of fields including combinatorics, topology and representation theory.   
 In Section ~\ref{sub:whole-fiber}, we determine  the combinatorial structure of 
fibers in terms of interior dual block complexes of subword complexes; 
we also prove in Section ~\ref{sub:whole-fiber}  that the unique regular CW complexes
having the same face posets as our fibers are contractible.  Section ~\ref{s:CW} proves that  the standard regular CW decomposition of 
 a simplex restricted to any fiber $f_{(i_1,\dots ,i_d)}^{-1}(p)\cap \RR_{\ge 0}^d$ of $f_{(i_1,\dots ,i_d)}$ 
 gives  a cell decomposition of $f_{(i_1,\dots ,i_d)}^{-1}(p)\cap \RR_{\ge 0}^d$ and proves  this is a cell stratification of 
 $f_{(i_1,\dots ,i_d)}^{-1}(p)\cap \RR_{\ge 0}^d$. 
Finally, we show in Section ~\ref{s:fsConj} how  contractibility of fibers, namely Conjecture C above, 
would combine with other results in this paper and  with  results in the literature regarding approximating maps by 
homeomorphisms  to yield a new   proof of the Fomin-Shapiro Conjecture.

\section{Background}\label{s:topology}

\subsection{Cell decompositions and their  posets of closure relations}\label{strat-defs}


  A   {\em finite CW complex}  is a decomposition of a Hausdorff space into a finite number of cells so that (i) a set is closed if and and only if its intersection with the closure of each cell is closed, (ii) the topological boundary of every $d$-cell is contained in a finite union of cells of dimension strictly less than $d$, and (iii)
 for every $d$-cell there is a continuous surjective map from an $d$-ball to the closure of the cell which restricts to a homeomorphism from the interior of the $d$-ball to the cell.  A map satisfying (iii) above is called a {\em characteristic map} and the restriction of a characteristic map to the sphere $S^{d-1}$ is called an {\em attaching map}.  
A map $f : A \to B$ is an {\em embedding} if it is a homeomorphism onto its image.     A  {\it regular CW
complex} is a CW complex such 
that for each cell there exists an attaching map which is an embedding. 
   The cells of a  regular CW complex comprise a cell stratification (a notion defined in Section ~\ref{main-result-section}).  
     Conversely, the face poset  of the 
     cells
of a regular CW complex determines the CW complex (up to homeomorphism).

A stratification induces a partial order on the index set by defining $\alpha \leq \beta$ if and only if $X_\alpha \subseteq \overline{X_\beta}$.
%
That is, the {\it closure poset} (or {\it face poset}) 
of a cell stratification
is the partial order on cells given by $\sigma \le \tau $ iff $\sigma
\subseteq \ol\tau$. 
  The {\it order complex} of a finite partially ordered set is the abstract simplicial complex whose $i$-faces are the chains
$v_0 < v_1 < \cdots < v_i$ of $i+1$ comparable elements in the poset.  
\begin{remark}
When we speak of a poset having a topological property such as being contractible, we are saying that its order complex has this property.  
\end{remark}

A  poset is {\em graded} if for each $u\le v$, all paths from $u$ to $v$ have the same length.  A graded poset is {\em thin} if each rank 2 interval $[u,w]$ includes  exactly 2 elements $v_1,v_2$ satisfying $u < v_i  < w$.
In \cite{Bj}, Bj\"orner  defines a variant of the face poset of a CW complex by adding the empty set to the stratification given by cells.   We call this variant  the {\em closure poset of a CW complex}, and denote its unique minimal element (corresponding to the empty set) by 
$\hat{0}$.  
Under this convention, the closure poset  of any finite  regular CW complex 
will be a poset  graded by cell dimension 
with each open interval $(\hat{0},v) = \{ z \mid \hat{0} < z < v \} $ having order
complex that is homeomorphic to a sphere $S^{\rk v  -2}$.  Bj\"orner proved in \cite{Bj}
 that this sphericity requirement  together 
with having a unique minimal element and at least one other poset element is enough to 
ensure that a finite, graded poset is the closure poset of a regular CW complex;  
finite  graded posets with these properties are therefore  called {\it CW posets}.  Results of Danaraj and Klee from \cite{DK}  imply that finite graded  posets with unique minimal element and at least one additional element will be CW posets whenever  they are thin and shellable; this was used  in \cite{BW} to prove that Bruhat order is a CW poset, a fact we will use in Section ~\ref{s:fsConj}.

A  map $f:P\rightarrow Q$ from a poset $P$ to a poset $Q$ is a {\it poset map} if
$u\le v$ in $P$ implies $f(u)\le f(v)$ in $Q$.  The main example of a poset map appearing in this paper is a poset map from the poset of subwords of a fixed word $(i_1,\dots ,i_d)$ ordered by containment 
 to Bruhat order, namely a poset map previously studied in \cite{AH} that sends each subword to its Demazure product.

\subsection{Coxeter groups, Demazure product, and subword complexes} \label{s:coxeter}
\commentjd{I corrected a few minor misprints in \ref{s:coxeter}.   I added a discussion of the longest element.  I added the last remark.}


We will make use of numerous well known properties of 
finite Coxeter systems as well as versions of these properties that transfer to 
associated 0-Hecke algebras.  
The unsigned 0-Hecke algebra will emerge out of a need to use  
a non-standard  product on a Coxeter group called the 
Demazure product.  We now review these notions and properties  for later use. 

A  finite  \emph{Coxeter system} $(W,\Sigma)$ 
consists of a finite group $W$ 
and a finite set of generators $\Sigma$ so that $W$ has a presentation of the form
$$
W = \langle s_i\in  \Sigma \mid (s_is_j)^{m(s_i,s_j)} = e \rangle
$$
where the $m(s_i,s_j)$ are positive integers with $m(s_i,s_i) = 1$ and with $m(s_i,s_j) = m(s_j,s_i) \geq 2$ for $i \not = j$.  The set $\Sigma$ is a minimal generating set whose elements are called {\it simple reflections}.  \commenth{I deleted a sentence here that was redundant; it is implied by $m(s_i,s_i)=1$.}  
 The elements $s_i$ and $s_j$ commute if and only if  $m(s_i,s_j) =2 $.  Sometimes we refer to commutation relations $s_i s_j = s_j s_i$ as {\it short braid relations}.  
For any $m(s_i,s_j)$, the 
relation  $(s_is_j)^{m(s_i,s_j)} = e$ is equivalent to the {\it braid relation} $s_i s_j s_i \cdots = s_j s_i s_j \cdots$ where each side of the equation is a product of  $m(s_i,s_j)$ simple reflections  alternating between $s_i$ and $s_j$.
We call this a {\it long braid relation} for  $m(s_i,s_j) > 2$.  See \cite{Hu} or
\cite{BB} for background on Coxeter groups.

When $W\hspace{-1pt} = S_n$ is the symmetric group, we
can take $\Sigma = \{s_1,\ldots, s_{n-1}\}$, where $s_i
= (i~i+1)$ is an adjacent transposition.  In this case, the relations are
$s_i^2 = e$, $(s_is_j)^2 = e$ for $|j-i| > 1$ and
$(s_is_{i+1})^3 = e$ for $1\le i \le n-1$.  

An {\it  expression} for $w \in W$ is a product $s_{i_1}\cdots s_{i_d}$ of simple reflections 
equalling $w$ under the standard group-theoretic product.  
  This is called a {\it reduced 
expression} for $w$ if $d$ is minimal among all possible expressions for $w$; this minimal $d$ is called the \emph{length} of ~$w$, denoted $l(w)$.  
A \emph{word} of size~$d$ is an ordered sequence $Q = (i_1,
\ldots, i_d)$ comprised  of the subscripts on an expression $s_{i_1}\cdots s_{i_d}$, so in other words an ordered sequence  of subscripts  
each indexing an  element of~$\Sigma$.  We write $|Q|$ for the size $d$ of a word $(i_1,\dots ,i_d)$.  Since one may pass 
easily back and forth between an expression and the corresponding word, one often speaks in terms of words just because they encode the same data more compactly.    An ordered subsequence
$P$ of a word~$Q$ is called  a \emph{subword} of~$Q$, written $P \subseteq Q$.  The expression
corresponding to such $P$ is called a {\it subexpression} of the expression corresponding to $Q$.

Subwords of~$Q$ come with their embeddings into~$Q$, so two subwords
$P$ and~$P'$ involving reflections at different positions in~$Q$ are treated as distinct 
 even if the sequences of reflections in $P$ and~$P'$
coincide.  

The next two results may be found, e.g.,  in Section 1.7 of \cite{Hu} and 
and as  Theorem 3.3.1 in \cite{BB}, respectively. 

\begin{lemma}[Exchange Condition]\label{exchange-lemma}
Let $w=s_{i_1}\cdots s_{i_r}$ (not necessarily reduced), where each $s_{i_j}$ is a
simple reflection.  If $\ell(ws_i) < \ell(w)$ for a simple reflection
$s_i \in \Sigma $, then there exists an  index $j$ for which $ws_i  =
s_{i_1}\cdots \hat{s}_{i_j}\cdots s_{i_r}$.  In particular, $w$ has a reduced
expression ending in a simple reflection $s_i\in \Sigma $ if and only if 
$\ell(ws_i) < \ell(w)$.
\end{lemma}

\begin{thm}\label{braid-connectedness}
%
(1) 
Any two reduced expressions for the same element  $w$ of a finite Coxeter group   $W$ 
are connected by a series of (long and short) braid moves, where a short
braid move is $s_is_j \rightarrow s_js_i $ for $m(i,j)=2$ and a long
braid move is $s_is_js_i\cdots \rightarrow s_js_is_j\cdots $  with each  of these 
expressions alternating $s_i$ and $s_j$  consisting of  $m(i,j) > 2$ letters.
(2)
 Any expression for $w$ is connected to some  reduced expression for $w$ (and hence any reduced expression for $w$)
 by a series 
of long and short braid moves together with nil-moves $s_i^2 \rightarrow 1$.  
\end{thm}

The following is an 
immediate consequence of the above results.

\begin{lemma}\label{l:braid}
Fix a reduced expression $s_{i_1} \ldots s_{i_d}$ for
$w \in W$ and a simple reflection
$s\in \Sigma $  such that $s_{i_1} \ldots s_{i_d}s$ is non-reduced.
Then there is a reduced expression $s_{j_1} \ldots s_{j_d}$ for~$w$ with
$s_{j_d} = s$ and a sequence of (long and short) braid moves that transforms $s_{i_1} \ldots s_{i_d}$ 
into $s_{j_1} \ldots s_{j_{d-1}}s$.
\end{lemma}

Another consequence is that multiplication of $w\in W$ by a simple reflection $s_i\in \Sigma $ always changes the length.  More precisely, 
$l (w s_i) = l(w) \pm 1$ and likewise $l(s_i w) = l(w) \pm 1$.

\begin{lemma} \label{longest_word}  Given any finite Coxeter group $W$,  
there is a unique element $w_0 \in W$ whose length is maximal.   This element is called the {\em longest element} of $W$.
\end{lemma}

\begin{defn}\label{Bruhat-def}
The {\em Bruhat order} is the partial order on elements of a Coxeter group $W$ with 
$u\le v$ if and only if there exist reduced expressions for $u$ and $v$ such that the 
reduced expression for $u$ is a subexpression of the reduced expression for $v$.  
\end{defn}

The longest element $w_0 \in W$ for $W$ a finite Coxeter group  is the unique maximal element  in 
the Bruhat order for $W$.  

Next we turn to the Demazure product.  This will play a critical role throughout this paper.

\begin{prop}\label{p:demazure}
There is a unique associative map $\delta: W \times W \to W$ 
 such that
$$%
\delta(w,s_i ) = \left\{
	\begin{array}{@{\,}ll}
	w s_i & \text{if } \ell(ws_i) > \ell(w)\\
	   w    & \hbox{if } \ell(ws_i) < \ell(w)
	\end{array}\right.
$$
for $w \in W$ and $s_i \in \Sigma$. 
\end{prop}
\begin{proof}
See \cite[Section~3]{subword}.
\end{proof}

\commenth{I added the next  two sentences, because I was worried about readers getting unnecessarily  bogged down trying to think how to calculation 
$\delta (w_1,w_2)$  in the case when neither $w_1$ nor $w_2$ is in $\Sigma $.  The proofs e.g.\ in Section 4 are tough enough as it is.}

To understand our proofs, it will suffice to know how to calculate $\delta $ on $W\times \Sigma $ and on $\Sigma \times W$.    One may verify from the above that
$\delta (w,e) = w = \delta (e,w)$  for each $w\in W$ and that  
$$\delta (s_i,w) = \left\{
	\begin{array}{@{\,}ll}
	s_i w  & \text{if } \ell(s_i w) > \ell(w)\\
	   w    & \hbox{if } \ell(s_i w) < \ell(w)
	\end{array}\right.
$$ 
for $w\in W$ and $s_i\in \Sigma $.

\begin{defn}\label{d:demazure}
The map $\delta$ in Proposition~\ref{p:demazure} is the \emph{Demazure
product} on~$W$.  Using associativity,  extend it  to a map $\delta : W^d \to W$ for all positive integers $d$.  Given a word $Q =  (i_1,\dots , i_d)$ 
in the sense of Definition \ref{d:subword}, define $\delta(Q) = \delta(s_{i_1}, \ldots , s_{i_d})$.   The key relations in the case of the symmetric group are $\delta( s_i ,s_{i+1} ,s_ i) = \delta(s_{i+1} ,s_i ,s_{i+1})$, $\delta(s_i, s_j) = \delta(s_j, s_i)$ when $|i - j| > 1$, and $\delta(s_i,s_i) = \delta(s_i)$.  
\end{defn}

The following alternative description for the Demazure product 
 will justify  the equivalence of  this map $\delta $ to 
the standard product for the unsigned 0-Hecke algebra associated to $W$, 
with this equivalence  using the bijective 
correspondence described shortly  between generators of $W$ and generators of  its (unsigned)  0-Hecke algebra.

\begin{lemma}\label{0-Hecke-to-Demazure}
Definition ~\ref{d:demazure}  is equivalent to  the following set of 
requirements  for an associative map $\delta $:
\begin{enumerate}
\item 
$\delta (w,s_i) = ws_i$ if $l(ws_i) > l(w)$
\item
$\delta (s_i,s_i ) =  s_i$ for $s_i \in \Sigma $.
\item
$\delta (s_i,s_j, s_i \dots ) = \delta (s_j,s_i,s_j,\dots ) $ where each side is an alternation of 
length $m(i,j)$ of the simple reflections $s_i$ and $s_j$.
\end{enumerate}
\end{lemma}

\begin{proof}
Each of these three  conditions follows easily from special cases of the conditions given in Definition ~\ref{d:demazure}.  Conversely, we obtain the condition $\delta (w,s_i) = w$ for $l(ws_i ) < l(w)$ from Definition ~\ref{d:demazure} from these three conditions as follows.  We use  the fact that $w$ must  have a reduced expression with $s_i$ as its rightmost letter to have $l(ws_i) < l(w)$ (see Lemma ~\ref{exchange-lemma})  together with the fact (recalled in Theorem  ~\ref{braid-connectedness}) that any reduced expression for $w$ may be obtained from any expression for $w$ via a series of (long and short) braid moves and nil-moves (with each nil-move giving rise to a modified nil-move
 $\delta (s_i, s_i)  \rightarrow s_i$  when using the Demazure product). 
\end{proof}

At times it will be convenient to refer to  
the type (2) relations $\delta (s_i,s_i) \rightarrow s_i$ above  as modified nil moves and the type (3) relations
$\delta (s_i,s_j,s_i,\dots ) \rightarrow (s_j,s_i,s_j, \dots )$ above where each side is an alternation of length $m(i,j)$  as braid moves.  

We work 
exclusively in terms of the Demazure product, but  e.g.\ \cite{He} instead used the language of the 0-Hecke algebra.  We briefly discuss that alternative 
perspective now, in case it may help some readers.  

A finite Coxeter system  $(W,\Sigma )$ gives rise to the Demazure product $(W, \delta)$ which in turn gives rise to a ring, called the (unsigned)  0-Hecke algebra.   Abstracting a bit, let $G$ be a set, $e \in G$ an element, and let  $\phi : G \times G \to G$ be an associative function, so that $\phi(e,g) = g = \phi(g,e)$ for all $g \in G$.   Let $R$ be a ring.   Define a ring $R[G,\phi]$  to be additively  the left free $R$-module with basis $G$ and give it the multiplication 
$$
\left(\sum r_ig_i\right)\left(\sum r_j'g_j' \right) = \sum r_ir_j' \phi(g_i,g_j')
$$
The ring $\F_2[W,  \delta]$ is the  $0$-Hecke algebra over $\F_2$ (henceforth just called the 0-Hecke algebra).   
\commenth{I shortened and simplified the discussion of the 0-Hecke algebra, trying to make it a bit more informal and easier to read, 
both because we won't need it at all in this paper and also because I could not understand much of what had been written.}
The generator $s_i\in \Sigma $ from a Coxeter system $(W,\Sigma )$ naturally corresponds to the generator we denote by $x_i$  in 
the associated  0-Hecke algebra.
These generators 
 satisfy two types of relations:
 \begin{enumerate}
 \item
 $x_i^2 = x_i$ for each $s_i \in \Sigma $ 
  \item
   $x_i x_j x_i \dots = x_j x_i x_j \dots$ for each $x_i,x_j\in \Sigma $ where both sides 
   have degree $m(i,j)$ 
\end{enumerate}
We call  $x_i^2 \rightarrow x_i$ a  {\it modified nil-move} and call  $x_ix_j x_i \dots \rightarrow x_j x_i x_j\dots $ a  {\it braid move}.

 We note that this 0-Hecke algebra over $\F_2$ is exactly the specialization of the usual Hecke algebra over the field of two elements where the usual parameter $q$ is set to 0; working over $\F_2$ allows us to  ignore signs.

One may use the product in the 0-Hecke algebra to determine for an 
expression $x_{i_1}(t_1)x_{i_2}(t_2)\cdots x_{i_d}(t_d) = p $ which $w\in W$ has $p\in U(w)$  as follows.  
Take the subexpression of $x_{i_1}(t_1)x_{i_2}(t_2)\cdots x_{i_d}(t_d)$ comprised of only those terms $x_{i_r}(t_r)$ such that the parameter
$t_r$ is positive and then suppress  the 
parameters to get an expression $x_{i_{j_1}}\cdots x_{i_{j_k}}$ in the 0-Hecke algebra.  This element of the 0-Hecke algebra is naturally indexed by the Coxeter group element $w$, as one may calculate as follows.  
 One may apply braid moves and modified nil-moves to obtain a reduced expression in 
these generators, then replace each $x_i$ by $s_i$ to get a reduced expression for the desired  $w$, thereby determining which $w$ has $p\in U(w)$.   
An equivalent  way to describe how to obtain  this $w\in W$ is as follows.  
Take   the subexpression  of $s_{i_1}s_{i_2}\cdots s_{i_d}$ consisting
of those letters $s_{i_r}$ such that $x_{i_r}$ has a strictly positive  parameter $t_r$, then  calculate the Demazure product of this subexpression
to get $w$.    This  0-Hecke algebra perspective regarding  strata  is used extensively in \cite{He}.
 
 \begin{remark}\label{Demazure-connected}
 Theorem ~\ref{braid-connectedness} may be extended  in a straightforward manner to an analogous statement regarding expressions which have the same Demazure product as each other, by simply replacing nil-moves $s_i s_i \rightarrow e$ with modified nil-moves $(s_i,s_i)\rightarrow (s_i)$.
The main points are  (a)  that any two reduced expressions having the same Demazure product as each other are  connected by a series of braid moves, and (b) that any expression may be transformed to any reduced expression which has the same Demazure product as it by a series of braid moves and modified nil-moves.
 \end{remark}

The next notion 
made an early appearance in \cite[Lemma~3.5.2]{subword}
and was  formally defined (and named) in \cite{He}, where it played 
a key role  in several    proofs. 

\begin{defn}\label{d:deletion-pair}
The letters $s_{i_j}$ and~$s_{i_k}$ in an expression  
$s_{i_1}\ldots s_{i_d}
$ constitute a \emph{deletion pair} if $j < k$ with 
$s_{i_j}\ldots s_{i_{k-1}}$ and
$s_{i_{j+1}}\ldots s_{i_k}$ both reduced expressions for the same Coxeter group element
while $s_{i_j} \ldots s_{i_k}$ is a non-reduced expression.   
\end{defn}

For example, in the symmetric group, $s_{3}s_{1}s_{2}s_{1}s_{2}$ 
has a deletion pair $ \{ s_{i_2},s_{i_5} \} $.  It is proven in \cite{He} that the condition that these two reduced expression are for the same Coxeter group element actually follows from the other parts of the definition.

 \begin{prop}\label{p:braid-deletion}
 Given a deletion pair $\{s_{i_j} , s_{i_k} \} $  in an expression $s_{i_1}\ldots s_{i_d}$, there is 
 a series of braid moves that may be applied to $s_{i_j}\ldots s_{i_{k-1}}$ 
 yielding another reduced expression 
 $s_{i_j'} \ldots s_{i_{k-1}'}$ such that $i_{k-1}'  = i_k$.
 \end{prop}

\begin{proof}
 This is a consequence of the exchange axiom for Coxeter groups along
 with Theorem ~\ref{braid-connectedness} and Lemma 
  \ref{l:braid}.
\end{proof}

\begin{lemma}\label{l:deletion-pair}
Under the conditions of Definition~\ref{d:deletion-pair}, the two
words $s_{i_j}\ldots s_{i_{k-1}}$ and $s_{i_{j+1}}\ldots s_{i_k}$
are reduced expressions for the same Coxeter group element
$\delta(s_{i_j}, \ldots , s_{i_k})$.
\end{lemma}
\begin{proof}
This follows from \cite[Lemma~3.5]{subword} or from \cite[Lemma~5.5]{He}.
\end{proof}

The Demazure product is  quite useful for  understanding  relationships between
reduced subwords for a fixed element~$w$ inside of a given ambient
word~$Q$ \cite{subword}.  These relationships are expressed
topologically using the subword complexes discussed next, complexes 
which were first introduced in 
\cite[Definition~1.8.1]{grobGeom} and \cite[Definition~2.1]{subword}.

\begin{defn}\label{d:subword}
A word $P$ \emph{contains} $w \in W$ if some subsequence of $P$ is a reduced expression for $w$.

The \emph{subword complex} $\Delta(Q,w)$ for a word $Q$ and an element
$w \in W$ is the simplicial complex whose $k$-simplicies are given by $(k+1)$-letter subwords  $R$ of $Q$ so that  $P = Q \minus R$ 
contains~$w$.
\end{defn}

The facets of the
subword complex $\Delta(Q,w)$ are given by  those words $R = Q \minus P$ where $P$ is   a reduced
word  for~$w$. 

\begin{thm}[{\cite[Theorems~2.5 and~3.7 and Corollary~3.8]{subword}}]\label{t:ball}
The subword complex $\Delta(Q,w)$ is shellable and homeomorphic to
either a ball or a sphere.  It is homeomorphic to a ball if and only if  
$\delta(Q) \ne  w$.  A face $Q \minus P$ lies in the boundary
of $\Delta(Q,w)$ if and only if $P$ satisfies $\delta(P)
\neq w$.
\end{thm}

\begin{remark} \label{subword_for_Demazure_product}
Any word $Q$ contains a subword that is a reduced word  for the Demazure product $\delta(Q) \in W$.   Indeed if $Q = (i_1, \dots , i_d)$, such a subword may be obtained by omitting all $i_j$ with the property that $\delta(i_1, \dots, i_j) = \delta(i_1, \dots, i_{j-1})$.   See also Lemma 3.4 of \cite{subword}.  

We will make extensive use of the analogous fact that the subword obtained by omitting each $i_j$ such that $\delta (i_j,\dots ,i_d) = \delta (i_{j+1},\dots ,i_d)$ is a reduced word for $\delta (Q)$ and in fact  is the rightmost subword of $(i_1,\dots ,i_d)$ that is a reduced word for $\delta (i_1,\dots ,i_d)$.  
\end{remark}

\subsection{Total positivity}\label{positivity-background}  \commentjd{I did a major revision of this section.   I added the Coxeter system in the nontype A case, the definition of $U_{>0}$, Lusztig's decomposition of $U$.  We were inconsistent on whether the unipotent radical was $U$ or $U_+$; I settled on $U$.   Revisions on the last 2.5 pages of this subsection were more extensive; you may want to look at them more carefully.}

Let $U$ be the unipotent radical of a Borel subgroup $B$  of a reductive, connected algebraic group $G$.  For example, let  $G = GL_n(\R)$  (i.e. the type A case) and let $B$ be the upper triangular matrices in $GL_n(\R)$, in which case $U = U^n$ is  the ``unipotent" group of upper triangular matrices with 1's on the diagonal.  

In this type A case, a matrix in $M_n(\RR)$ is {\em totally nonnegative} if all of its matrix minors are nonnegative real numbers; it is {\em totally positive} if all of its minors are strictly positive.    We are   interested in the space of totally nonnegative, real  matrices which are upper triangular with ones on the  main diagonal, stratified by specifying which minors are positive and which are zero.    For $t \in \R$ and $j \in \{ 1,  \dots, n-1\}$, define $x_j(t) = I_n + tE_{j,j+1}(n) \in U^n$ with $E_{j,j+1}(n)$ the $n \times n$ matrix which is all zeroes except for a 1 in row $j$ and column $j+1$.    If $t \geq 0$, then $x_j(t)$ is totally nonnegative.   Let $U^n_{\geq 0}$ be the monoid consisting of finite products of matrices $x_j(t)$ with $t$ nonnegative; $U^n_{\geq 0}$ is the intersection of $U^n$ with the set of  totally nonnegative matrices.

Let $\R_{\geq 0}^d$ denote the orthant of $\R^d$ consisting of all points with nonnegative coordinates.    Given any (not necessarily reduced) word $Q = (i_1,\dots ,i_d)$ in the Coxeter system $(W,\Sigma) = (S_n, \{s_1,\dots, s_{n-1}\}$), 
Lusztig proved a number of important properties of the  continuous map
$$
f_Q : \RR^d_{\geq 0} \to U^n_{\geq 0} \subset U^n \subset M_n(\R)
$$
given by $f_Q(t_1, \ldots, t_d) = x_{i_1}(t_1) \dots x_{i_d}(t_d)$.

Next we turn to the general case.
One can associate a Coxeter system $(W, \Sigma)$  
to a reductive algebraic group $G$.  (Here $W = N(T)/T$ where $T$ is a maximal torus in $G$ and $N(T)$ is its normalizer, so $W$ is the Weyl group of $G$.)  $W$ is always generated by a set  $\Sigma = \{s_1, \dots, s_k\}$ of simple reflections.  
 As above, a word $Q = (i_1, \dots, i_d)$ determines the element $w = s_{i_1} \dots s_{i_d} \in W$.

Given a reductive, connected algebraic group $G$ defined and split over $\R$ and the unipotent radical $U$ of a Borel subgroup of $G$, Lusztig makes use of  an extra structure -- a pinning --  that is somewhat analogous to choosing a basis for $\R^n$.  The data of a pinning includes  for each $j \in \{1, \dots, k\}$, a homomorphism $x_j: \R  \to U$.   Lusztig  defines the monoid $U_{\geq 0} \subset U$ to be the monoid generated by the set of elements  $x_j(t)$ for all $j\in \{ 1,\dots ,k\} $ and all $t \geq 0$.

\begin{defn}\label{pinning-def} 
Given a split reductive connected algebraic group $G$, a {\it pinning} for $G$ consists of a $K$-split maximal torus $T$, a pair $B^+,B^-$ of opposite Borel subgroups containing $T$, and isomorphisms $x_i$ and $y_i$ for each $i\in I$ for $I$ the set of simple reflections, where $x_i$ is an isomorphism from $\RR $  to $U_i^+$ and $y_i$ is an isomorphism from $\RR $ to $U_i^-$ (where $U_i^+$ and $U_i^-$ are the simple root subgroups of $U^+$ and $U^-$ indexed by $i$) and where  $x_i$ and $y_i$ together give rise to a homomorphism from $SL_2(\RR )$ to $G$.  There is additional  structure in the full  definition  of a pinning, as described e.g.\ in Section 1 of \cite{Lu}, but this abbreviated version suffices for our purposes.
\end{defn}

Analogous to the type A case, Lusztig  defines, for each word $Q = (i_1, \dots, i_d)$, a continuous map
$$
f_Q : \RR^d_{\geq 0} \to U_{\geq 0}
$$
given by $f_Q(t_1, \ldots, t_d) = x_{i_1}(t_1) \dots x_{i_d}(t_d)$.

We will be especially focused on the structure of each {\em fiber} of the map $f_{(i_1,\dots ,i_d)}$, by which  we mean the set  $ f^{-1}_{(i_1,\dots ,i_d)}(p)
\cap \RR_{\ge 0}^d$ for each  $p\in U_{\ge 0}$. 
Now we recall results of Lusztig  
from \cite{Lu} that 
will prove vital to our work. 

\begin{prop}\label{Lu-0th-prop}[Lus94, Proposition 2.6]
Let $s_i, s_{i'}\in \Sigma$ be distinct simple reflections in $W$.  
Let $m(i,i') \ge 2$ be the order of $s_is_{i'}$ in $W$.
Let $a_1,a_2,\dots ,a_m$ be a sequence in $\R_{>0}$.  Then there exists a unique sequence
$a_1',a_2',\dots ,a_m'$ in $\R_{>0}$ such that 
$$x_i(a_1)x_{i'}(a_2)x_i(a_3)x_{i'}(a_4)\cdots = x_{i'}(a_1')x_i(a_2')x_{i'}(a_3')x_i(a_4)\cdots $$
with each side  of the equation consisting  of a product of 
$m(i,i') $ factors with alternating subscripts $i$ and $i'$.
\end{prop}

\begin{example}
In type A for $1\le i < n$, Proposition ~\ref{Lu-0th-prop} specializes to  the (more precise) 
statement 
$$x_i(t_1)x_{i+1}(t_2)x_i(t_3) = x_{i+1}\left(\frac{t_2t_3}{t_1 + t_3}\right)x_i(t_1 + t_3)x_{i+1}\left(\frac{t_1t_2}{t_1+t_3}\right) $$  for any 
$t_1,t_2,t_3 > 0$, as one may easily confirm  by matrix  multiplication.  
\end{example}

We will use the next result quite extensively, the last part of which appears within the proof of Proposition 2.7 in \cite{Lu}.

\begin{prop}\label{Lu-first-prop}[Lus94, Proposition 2.7]
Let $w\in W$, and  let 
$w= s_{i_1}s_{i_2}\cdots s_{i_d}$ be a reduced expression.  Then the following all hold.
\begin{enumerate}[(a)]
\item%
The map $f_{(i_1,\dots ,i_d)}$  from  $\R_{>0}^d$ to $ U$ given by
$(a_1,a_2,\ldots,a_d)\mapsto x_{i_1}(a_1)x_{i_2}(a_2)\cdots
x_{i_d}(a_d)$ is injective.
\item%
The image of the map $f_{(i_1,\dots ,i_d)}$  
from part  (a) is a subset $U(w)$ of $U_{\ge 0}$
which  depends on $w$ but not on the choice of reduced word 
$(i_1,i_2,\dots ,i_d)$ for $w$.
\item%
If  $w'\ne w$ for $w' , w\in W$, then $U(w)\cap U(w') = \nothing$.
\end{enumerate}\setcounter{separated}{\value{enumi}}
\begin{enumerate}[(a)]\setcounter{enumi}{\value{separated}}
\item%
For $a_1,\dots ,a_d$ nonzero elements of $\R$, we have
$$x_{i_1}(a_1) \cdots x_{i_d}(a_d)\in B^-s_{i_1}B^-s_{i_2}B^-\cdots
s_{i_d}B^- \subset B^-s_{i_1}s_{i_2}\cdots s_{i_d}B^-$$ by virtue of properties
of the Bruhat decomposition.
\end{enumerate}
\end{prop}

\begin{prop} [Lus94, Corollary 2.8]
 $U_{\geq 0} = \sqcup_{w\in W} U(w)$, i.e.,  $\{ U(w)  \}_{w\in W}$ is a decomposition of $U_{\geq 0}$.
\end{prop}

Combining this with Proposition 4.2 of \cite{Lu}  easily yields that this decomposition is a cell decomposition.
Implicit in Lusztig \cite{Lu} is also the following  result (see also   Whitney   
\cite{Wh}   and  Loewner for earlier work   in \cite{Lo} in type A):

\begin{thm}
Given any reduced word $Q$ for the longest element $w_0$ in the symmetric group, the image of 
$f_Q$ as a map on $\RR_{\ge 0}^d$ is the entire space of real-valued upper triangular matrices with 1's on the diagonal having all minors nonnegative.  The analogous statement also
holds more generally for the totally nonnegative part $U_{\geq 0}$  of the unipotent radical of a Borel subgroup in a reductive group defined and split over $\R$.
\end{thm}

In type A, one way to describe the stratum $U(w)$ is to note  that the   
matrices   $M\in U_{\ge 0}$  within any such  stratum are those in which certain minors  are strictly positive and the other minors are all  zero.  
More generally (i.e. in all cases where the Weyl group is a finite group), we use the  notion of {\em generalized minor} that was  
introduced in \cite{FZ} to specify the strata based on which generalized minors are positive and which are 0.  The generalized minors are 
certain types of 
 regular  functions on an  algebraic group, with this  notion  
 coinciding  in type A with the usual minors. 

Another way of defining the strata in all finite types that will also play an important role in our work  is as follows.  
 For each  
$p = f_{(i_1, \cdots, i_d)}(t_{i_1}, \dots, t_{i_d}) =  x_{i_1}(t_1)\cdots x_{i_d}(t_d)$,  the assignment of $p$ to a 
strata is dictated by taking the subword of
$(i_1, \cdots, i_d)$ consisting of exactly those letters $i_j$
such that $t_j>0$ and  letting $w$ be the 
Demazure product of the subword.   Then $p \in U(w)$.

For $W$ a finite group, let  $U_{> 0} = U(w_0)$ where $w_0\in W$ is the longest element in $W$.   Lusztig \cite{Lu} also proved a collection of topological results, parts (a), (b), (c) and (e) below, while Fomin and Shapiro \cite{FS} proved (f) below, and (d) below is a direct consequence of Lusztig's other results, as explained shortly.

\begin{thm} \label{Lu-theorem}
(a) $U_{\ge 0}$ is a closed subset of $U$.\\
(b) $U_{>0}$ is a dense subset of $U_{\ge 0}$.   More generally, $\overline{U(v)}  = \sqcup_{u\le v}  U(u)$.\\
(c) Given any  (not necessarily reduced) word $(i_1,\dots ,i_d)$,   
the map $f_{(i_1,\dots ,i_d)}$ is a proper map from $\RR_{\ge 0}^d$ to $U_{\ge 0}$.  \\
(d) If $Q= (i_1, \dots, i_d)$ is a reduced word for $w \in W$, then $f_Q : \R^d_{>0} \to U(w)$ is a homeomorphism and $f_Q : \R^d_{\geq 0} \to \overline{U(w)}$ is a quotient map.  \\
(e) $\{ U(w)  \}_{w\in W}$  is a cell stratification of $U_{\geq 0}$; more generally, for any $v \in W$, $\{ U(u)  \}_{u\ \leq v}$  is a cell stratification of $\overline{U(v)}$. \\
(f)  $ u \leq v$ (in the Bruhat order) if and only if $U(u) \subseteq \overline{U(v)}$.
\end{thm}

Note that part (d) of Theorem \ref{Lu-theorem} follows from part (a) of Proposition \ref{Lu-first-prop} and part (c) of Theorem \ref{Lu-theorem} since a proper bijective (resp.\ surjective) map from a locally compact space to a  locally compact, Hausdorff space is a homeomorphism (resp.\ quotient map).  To see that the images of these maps are locally compact, note that they are closed subsets of Euclidean space, by virtue of being defined by closed conditions (i.e the requirements  that various generalized minors are nonnegative and that other generalized minors are 0).

We now discuss $f_Q$ where $Q$ is not necessarily a reduced word.    
A key  consequence of Lemma ~\ref{open-cell-image-description} below  
is that that 
the image of $f_Q$  equals the image of $f_R$ for any  reduced word $R$ having $\delta (R) = \delta (Q)$.  On the other hand, we will show later in the paper that  
the fibers for $f_Q$ with non-reduced $Q$  will be larger than for $f_R$ given by  
any  subword $R$ of $Q$  that is reduced and has the same Demazure product as $Q$.  
The following lemma follows from Proposition \ref{Lu-0th-prop}.

\begin{lemma} \label{open-cell-image-description}
 Let $Q$ be a word of size $d$ (reduced or not).\\
 (a) $f_Q(\RR^d_{>0}) = U(\delta(Q))$.\\
 (b) $f_Q(\RR^d_{\geq 0}) = \overline{U(\delta(Q))}$.
\end{lemma}

In particular, let $R$ be a subword of $Q$ which is a reduced word  for $\delta(Q)$ (see Remark \ref{subword_for_Demazure_product}).    Let $d$ be the length of $Q$ and $e$ be the length of $R$.   Consider $\RR^e$ as a subset of $\RR^d$ by putting zero in the slots $Q \backslash R$.    Then $f_Q$ restricts to $f_R$ and the image of $f_Q$ is the same as that of $f_R$.

When we turn to the Fomin-Shapiro Conjecture near the end of the paper, we will 
shift  perspective from the topology of $U_{\geq 0}$ to the topology of the link of the identity in $U_{\ge 0}$. 
Recall that  $\Delta^{d-1}_K = \{ (t_1, \dots, t_d) \in \R^d_{\geq 0} \mid \sum t_i = K\}$ for any positive real constant $K$  is  a  standard $(d-1)$-simplex.     We sometimes denote $\Delta^{d-1}_1$ simply by $\Delta^{d-1}$.

\begin{definition} \label{link-definition}
 Fix a reduced word  $Q$ for the longest element $w_0 \in W$.   The {\em link of the identity in $U_{\geq 0}$} is $f_Q(\Delta^{l(w_0)-1})$.     
 More generally, if $v \in W$ and if $Q$ is a reduced word for $v$, the {\em link of the identity in $\overline{U(v)}$} is $f_Q(\Delta^{l(v)-1})$.
\end{definition}

Restricting the cell stratification of $U_{\geq 0}$ gives a cell stratification of the link of the identity where the dimensions of the cells are  each decreased by one.

Results of Bj\"orner and Wachs \cite{BW} and Bj\"orner  \cite{Bj} showed that the Bruhat order on $W$ is the face poset of a regular CW decompositon of a ball.   Hersh subsequently strengthened this result  to a proof of the Fomin-Shapiro Conjecture, namely the result 
described next.

\begin{thm}[\cite{He}]  The link of the identity in $U_{\geq 0}$ is a regular CW complex with   cells 
$\{U(w) \cap \Delta^{d-1}\} $ 
indexed by the elements  $w\in W$.  The poset of closure relations on these cells  is Bruhat order. 
 More generally, for any $v \in W$, the link of the identity in $\overline{U(v)}$ is a regular CW complex with cells 
$\{U(u) \cap \Delta^{d-1}\}_{u \leq v}$ having closure poset a lower interval in Bruhat order. 
\end{thm}

The above results imply for  $f_{(i_1,\dots ,i_d)}$ restricted to the simplex  $\Delta^{d-1}$  that  $im(f_{(i_1,\dots ,i_d)})$  
decomposes into cells that are each  the image of one or more cells in the standard cell decomposition of $\Delta^{d-1}$.  In particular, this guarantees that the map $f_{(i_1,\dots ,i_d)}$ of stratified spaces induces a map of face posets.   The following  combinatorial description of this poset map was given  in Corollary 5.3  in \cite{He} and studied further  in \cite{AH}.

\begin{prop}\label{Demazure-poset-map}
The Demazure product induces a poset map $f$ from the face poset of a simplex to the face poset for the stratification we  
just described for 
$\overline{U(w)}$.
In other words, it induces a poset map  $f$ from 
the  Boolean lattice  of subwords of a fixed reduced word  $(i_1,\dots ,i_d)$
for a Coxeter group element $w\in W$ to
the  Bruhat order interval $[1,w]$.    
This poset  map $f$ 
sends the  subword  $(i_{j_1},\dots ,i_{j_k})$    to the Bruhat order element
$\delta (s_{i_{j_1}},\dots ,s_{i_{j_k}}) \in W$.    
\end{prop}

\begin{remark}The former poset in Proposition ~\ref{Demazure-poset-map} is a Boolean lattice because 
each subset $\{ j_1,\dots ,j_k\} $ of $\{ 1, \dots , d\} $ 
with $1\le j_1 < \cdots < j_k \le d$  naturally
corresponds to a subword $(i_{j_1},\dots ,i_{j_k})$ of $(i_1,\dots ,i_d)$.    Thus, we identify elements
of the Boolean lattice with subwords.
\end{remark} 

\subsection{Interior dual block complexes}\label{s:interior}

Each fiber of the  map $f_{(i_1,\dots ,i_d)}$ comes with a combinatorial
decomposition (see Definition~\ref{d:natStrat}) induced by  the
stratification of the simplex (or equivalently the nonnegative orthant) by its
polyhedral faces.  The combinatorics of this decomposition of the
fiber precisely matches the block decomposition of the interior dual
block complex (see Definition~\ref{d:interiorDualBlock}) of a subword complex, as we will
show in  Proposition~\ref{p:fiberNerve}.  

\begin{defn}\label{d:dualBlock}
For a nonempty face~$\phi$ of a simplicial complex~$\Delta$, the
\emph{(closed) dual block} of~$\phi$ in~$\Delta$ is the underlying
space of the simplicial complex constructed as~follows:
\begin{itemize}
\item%
take the cone from the barycenter~$\beta$ of the face~$\phi$ over the
link of~$\phi$ in~$\Delta$;
\item%
barycentrically subdivide that cone; and then
\item%
take the star of~$\beta$
(equivalently, delete all vertices in the original link).
\end{itemize}
\end{defn}

\begin{definition}\label{manifold-def}
{\rm 
A {\it topological $n$-manifold with boundary} is a Hausdorff space $M$ having a countable basis of
open sets, with the property that every point of $M$ has a neighborhood homeomorphic
to an open subset of $\HH^n$, where $\HH^n$ is the half-space of points $(x_1,\dots ,x_n)$
in $\RR^n$ with $x_n\ge 0$.  The {\it boundary} of $M$, denoted $\partial M$, is the set
of points $x\in M$ for which there exists a homeomorphism of some neighborhood of $x$
to an open set in $\HH^n$ taking $x$ into $\{ (x_1,\dots ,x_n)\mid x_n = 0\}  = \partial \HH^n$.
}
\end{definition}

\begin{prop}\label{p:block=cell}
If $\Delta$ is a simplicial PL manifold-with-boundary with a nonempty
interior face~$\phi$ (that is, $\phi$ is not contained in the boundary
of~$\Delta$), then the dual block of~$\phi$ is homeomorphic to a
closed ball.
\end{prop}
\begin{proof} 
This is a consequence of basic results on PL balls and spheres
\cite[Theorem~4.7.21]{BLSWZ}.  The link of an interior face~$\phi$ is
a PL sphere.  Hence the cone over the link from the
barycenter~$\beta$ of~$\phi$ is a PL ball, as is the barycentric
subdivision, with~$\beta$ as an interior vertex.  Thus the
star of~$\beta$ in the subdivision is another PL~ball.
\end{proof}

\begin{defn}\label{d:interiorDualBlock}
Fix $\Delta$, a simplicial manifold-with-boundary.  The \emph{interior
dual block complex} of~$\Delta$ is
\begin{enumerate}
\item%
the union of the dual blocks of interior faces of~$\Delta$, if
$\Delta$ has nonempty boundary.
\item%
a ball whose boundary is the dual block complex of~$\Delta$, if
$\Delta$ is
a sphere.
\end{enumerate}
(The \emph{dual block complex} of~$\Delta$ is the union of the dual
blocks of its nonempty faces.)
\end{defn}

\begin{lemma}\label{l:interiorDualBlock}
Fix a simplicial manifold~$\Delta$ with nonempty boundary.  If
$\Delta^\circ$ is obtained from the barycentric subdivision
of~$\Delta$ by deleting all vertices lying on the boundary and all cells 
with any of these vertices in their closure, then the
underlying spaces of~$\Delta^\circ$ and the interior dual block
complex of~$\Delta$ coincide.
\end{lemma}
\begin{proof}
Immediate from Definitions~\ref{d:dualBlock}
and~\ref{d:interiorDualBlock}.
\end{proof}

\begin{prop}\label{p:dual-regularCW}
The interior dual block complex of any simplicial PL sphere or
simplicial PL manifold with nonempty boundary is a regular CW complex.
\end{prop}
\begin{proof}
This is immediate from \cite[Theorem~64.1]{munkres}, given that the relevant
dual blocks are closed cells by Proposition~\ref{p:block=cell}.
\end{proof}

\begin{prop}\label{p:dual-contractible}
The interior dual block complex given by any regular CW decomposition  
of any ball  or sphere  is contractible.
\end{prop}
\begin{proof}
The sphere case is by construction.  For balls, use
Lemma~\ref{l:interiorDualBlock} with \cite[Lemma~70.1]{munkres}:
removing the full closed subcomplex on a vertex set yields an open
subset that retracts onto the full closed subcomplex on the remaining
vertices.
\end{proof}

For further  background and basics on dual blocks and dual block complexes, see
\cite[\S 64]{munkres}.  For a brief introduction to \emph{piecewise
linear} or \emph{PL topology} that suits our purposes, see
\cite[Section~4.7(d)]{BLSWZ}. 
 For a more in-depth  introduction to the notion of
\emph{manifold-with-boundary}, see \cite[\S 35]{munkres}. 

  \commenth{The subsection called ``a topological interlude'' was just  moved from here  to Section 5, as suggested by Jim as a possibility.}

\section{Combinatorics of each  fiber}\label{sub:whole-fiber}

\commenth{I changed the next sentence because we 
should not use the term  ``Lusztig's parametrization'' in the next sentence, since this map was not introduced by Lusztig but is classical.   We could just say
``the maps $f_{(i_1,\dots ,i_d)}$''}

The relevance
 of interior dual block complexes of subword complexes
to 
the maps 
 $f_{(i_1,\dots ,i_d)}$ studied extensively by Lusztig arises via stratifications, as we discuss  now 
in this section.

\commenth{I think we should we be consistent with elsewhere and say $F = f_Q^{-1}(p)\cap \RR_{\ge 0}^d$ 
for some $p\in U(w)$ or maybe for some $p\in U_{>0}$.  That is, should we give some indication of what type of points we are taking fibers over.  I made a pass at doing this next}

\commenth{I was worried readers might be confused in the following definition when we speak of $\RR_{\ge 0}^P$ as to what is meant?  And by what
we mean when we speak of the positive orthant that all other coordinates must be 0.  So I clarified this.  Perhaps what I did is overkill though.  I also added the $(i_1,\dots ,i_d)$ so 
readers would know what $d$ we want to use in speaking of $\RR_{\ge 0}^d$.}

\begin{defn}\label{d:natStrat}
Fix a word~$Q = (i_1,\dots ,i_d)$ and a fiber $F  = f_Q^{-1}(p) \cap \RR_{\ge 0}^d$  
of the 
parametrization~$f_Q$ with $p\in U_{>0}$.  For each subword $P \subseteq Q$, let 
$F_P = F \cap \RR_{>0}^P$ be the intersection of~$F$ with the
strictly positive orthant indexed by~$P$, namely the region in which coordinates not indexed by $P$ are 0 
and those indexed by $P$ are strictly positive.  
The \emph{natural
stratification} of~$F$ has strata~$F_P$ for $P \subseteq Q$ and closed
strata $\ol F_{\!P} = F \cap \RR_{\geq 0}^P$.
\end{defn}

\begin{lemma}\label{l:intersect}
In the setting of the natural stratification as in
Definition~\ref{d:natStrat}, any nonempty intersection of closed
strata is a closed stratum.
\end{lemma}
\begin{proof}
By definition, $\ol F_{\!P} \cap \ol F_{\!P'} = \ol F_{\!P \cap P'}$
for any two subwords~$P$ and~$P'$ of~$Q$.
\end{proof}

\commenth{In this next proof we need to be in a matrix group.  I'm not sure we are assuming that elsewhere in the paper.}

\commenth{In the next statement, we should not call this Lusztig's parametrization, as he did not introduce this map.  
We also should not say ``over any point'' since we need to consider a point in 
$U(w)$ not just any point.  I made an attempt at fixing this below.    I use $(i_1,\dots ,i_d)$ instead of $Q$ since we reference this in the proof.
}
 \begin{lemma}\label{l:semialgebraic}
The fiber $F$ of the 
 parametrization $f_{(i_1,\dots ,i_d)}$ over any point in $U(w)$ 
  is a real
semialgebraic variety.  More precisely, $F$ is obtained by
intersecting the nonnegative orthant with the zero set of a family of
polynomials with real coefficients.
\end{lemma}
\begin{proof}
The condition for a point $(t_1,\ldots,t_d)$ with nonnegative
coordinates to lie in the fiber over a fixed matrix is polynomial
in~$t_1,\ldots,t_d$ because the entries of the product matrix
$x_{i_1}(t_1)\cdots x_{i_d}(t_d)$ are polynomials in $t_1,\ldots,t_d$.
\end{proof}

\commenth{I edited the statement of the next proposition, to say what the order relation is.  It had just said what the  elements of the partially ordered set were without any order relation on them.}

\begin{prop}\label{p:fiberNerve}
The partially ordered set $\{\ol F_{\!P} \mid P \subseteq Q\}$ of
closed strata of the natural stratification of the fiber~$F$  in
Definition~\ref{d:natStrat}, with strata partially ordered by containment,  is naturally isomorphic to the face poset
of the interior dual block complex of the subword complex
$\Delta(Q,w)$, where $w$ indexes the Bruhat cell containing the image
of~$F$.
\end{prop}
\begin{proof}
Notice that the interior faces of a subword complex $\Delta (Q,w)$ 
are given exactly by the subwords  $P$ of 
$Q$ whose complementary word $Q\setminus P$ has Demazure product exactly $w$, as shown 
in \cite[Theorem~3.7]{subword}.
Also recall that this subword complex is a sphere if and only if $\delta (Q) = w$, and recall that
the interior dual block complex of a $PL$-sphere consists of a dual cell to each cell of the 
original $PL$-sphere as well as one additional maximal cell having this sphere as its boundary.
  
With these facts  in mind, observe that the 
interior dual blocks of $\Delta(Q,w)$ are in containment-reversing
bijection with the interior faces of $\Delta(Q,w)$, noting that the empty face is an interior face
of a subword complex $\Delta (Q,w)$  if and only if $\delta (Q) = w$.
 The desired  stratification of $F = f_Q^{-1}(p)\cap \RR_{\ge 0}^d$
   for
 \commenth{I just changed  $Y_w^o$ to  $U(w)$}  $p\in U(w)$ 
 now  follows from the
 description of strata  one obtains by combining  Proposition 2.7 in 
\cite{Lu} with results in \cite{He}: 
that is, the open stratum~$F_P$ given by  a subword $P$  of $Q$ 
is nonempty if and only  if $\delta(P) = w$.
\end{proof}

\begin{prop}\label{c:dualBlock}
The interior dual block complex $\nabla(Q,w)$ of any subword complex
$\Delta(Q,w)$  
is a contractible regular CW complex.
In particular,
the nerve of the cover of 
$\nabla(Q,w)$ by its closed cells is contractible.
\end{prop}
\begin{proof}
The subword complex is a shellable ball or sphere by
Theorem~\ref{t:ball}, and therefore it is PL
\cite[Proposition~4.7.26]{BLSWZ}.  (For a ball,
\cite[Proposition~4.7.26]{BLSWZ} a~priori only implies directly that
it is PL if it has a shelling that can be completed to a shelling of a
sphere; but \cite[Theorem~4.7.21]{BLSWZ} implies that every shelling
of a ball~$B$ can be so extended by adding one closed cell, namely a
second copy of~$B$---thought of as a single closed cell---meeting the
original copy of~$B$ along its boundary.)  The result is now a special
case of Propositions~\ref{p:dual-regularCW}
and~\ref{p:dual-contractible}.
\end{proof}

In Remark ~\ref{non-pure-example}, we  observed 
that the interior dual block complex of the subword complex $\Delta (Q,w) $ for 
$Q = (1,3,2,1,3,2)$ and  $w = s_1 s_3 s_2$ is not pure.  Specifically, 
it has a 2-dimensional maximal cell and a one dimensional maximal cell, as shown in Figure 3. 
What leads to the presence  of  maximal cells of differing dimensions in this case  is the existence of an element $u\in W$ that is 
 less than $w$ in Bruhat order but not in weak order.   This example  also has $\delta (Q) \ne w$.  
It seems natural to ask: 

\begin{question}  
Suppose $w\in W$ has exactly the same elements below it in weak order as in Bruhat order.  In finite type, this is equivalent to assuming
 that $w$ is the longest element of a parabolic subgroup of $W$.  
Suppose  $Q$ is a word satisfying $\delta (Q) = w$.  Does
this imply  that the interior dual block complex of the subword complex $\Delta (Q, w)$ 
is pure (namely has all its maximal cells of the same dimension)? 
\end{question}

We conclude this section by justifying  a fundamental 
property of $f_{(i_1,\dots ,i_d)}^{-1}(p)\cap \RR_{\ge 0}^d$ that relies heavily on the combinatorial structure just described for each fiber.  

\begin{lemma}\label{intersect-with-simplex}
Consider  any $w\in W\setminus \{ e\} $, any $(i_1,\dots ,i_d)$ satisfying $\delta (i_1,\dots ,i_d)=w$, and any $p\in U(w)$.  
Let   $\Delta^{d-1}_K$ denote the subset of $\RR_{\ge 0}^d$ having coordinates summing 
to $K$, namely a particular closed $(d-1)$-simplex within $\RR^d$.
Then 
$$f_{(i_1,\dots ,i_d)}^{-1}(p)\cap \RR_{\ge 0}^d = f_{(i_1,\dots ,i_d)}^{-1}(p)\cap \Delta_K^{d-1}$$
for some real number  
$K>0$.   
\end{lemma}

\begin{proof}

Consider any  $(c_1,\dots ,c_d)\in f_{(i_1,\dots ,i_d)}^{-1}(p)\cap \RR_{\ge 0}^d$.  Since $w\ne e$, we must have $\sum_{i=1}^d c_i = K$ for some positive real number $K$.  Our
plan is to show $\sum_{i=1}^d k_i = K$  
 for every $(k_1,\dots ,k_d) \in f_{(i_1,\dots ,i_d)}^{-1}(p)\cap \RR_{\ge 0}^d$.  This will imply
$$f_{(i_1,\dots ,i_d)}^{-1}(p)\cap \RR_{\ge 0}^d \subseteq f_{(i_1,\dots ,i_d)}^{-1}(p)\cap \Delta_K^{d-1},$$  from which the result will follow since 
$\Delta_K^{d-1}\subseteq \RR_{\ge 0}^d$ by definition.    

We  associate to any point $(k_1,\dots ,k_d)\in f_{(i_1,\dots ,i_d)}^{-1}(p)\cap \RR_{\ge 0}^d$   
a subword $(i_{j_1},\dots ,i_{j_s})$  of $(i_1,\dots ,i_d)$ comprised of  the letters  $i_{j_l}$ for $1\le j_l\le d$  such that $k_{j_l}>0$.   Having 
$f_{(i_1,\dots ,i_d)}(k_1,\dots ,k_d) = p$ for some $p \in U(w)$ implies $\delta (i_{j_1},\dots ,i_{j_s}) = w$. 
 We begin by checking that 
 the changes of coordinates given by modified nil-moves and braid moves each preserve
 the sum of the coordinates.

 If  $i_{j_r} = i_{j_{r+1}}$, 
 then $\delta (i_{j_1},\dots ,i_{j_r}, 
 i_{j_{r+1}},i_{j_{r+2}},\dots ,i_{j_s}) = \delta (i_{j_1},\dots ,i_{j_r},
 i_{j_{r+2}},\dots ,i_{j_s})$.  Applying a modified nil-move deleting $i_{j_{r+1}}$ 
  carries out the change of coordinates 
   which replaces
$(k_1,\dots ,k_d)$ by the point $(k_1',\dots ,k_d')$ in which 
$k_{j_r}' = k_{j_r}+k_{j_{r+1}}$, $k_{j_{r+1}}' = 0$ and for $s \not\in \{ r,r+1\} $ we have
$k_{j_s}' = k_{j_s}$.
 Note   that each such change of coordinates  preserves the sum of the parameters and 
reduces  the number of letters in the associated subword of $(i_1,\dots ,i_d)$ by exactly one. 
Thus, each modified nil-move preserves the desired sum.

If we  take the point $(k_1,\dots ,k_d) \in f_{(i_1,\dots ,i_d)}^{-1}(p)$  and 
apply a braid move to $(i_{j_1},\dots ,i_{j_s})$, then Proposition ~\ref{Lu-0th-prop} guarantees the existence and uniqueness
of a point $(k_1',\dots ,k_d')$ that is the  change of coordinates from $(k_1,\dots ,k_d)$  given by our braid move.    
By definition, any braid move preserves the number of   
 letters in the associated word. 
  The proof of Theorem 3.1 
 shows   for each braid  relation $$x_i(t_1)x_j(t_2)x_i(t_3)\cdots = x_j(t_1')x_i(t_2')x_j(t_3')\cdots $$ that 
$\sum_{i\ge 1} t_{2i-1} = \sum_{i\ge 1}t_{2i}'$ and $\sum_{i\ge 1}t_{2i} = \sum_{i\ge 1} t_{2i-1}'$. 
Adding these two equations yields the desired preservation of the sum of the parameters  under the change of coordinates given by the braid move.
 Alternatively, one may justify  that  the 
   sum of just those  parameters  given by any fixed $i\in I$  (i.e. those parameters $t_j$ appearing  in some $x_{i_j}(t_j)$ with $i_j = i$) 
   is   preserved  under each   modified nil-move and   each   braid move  
 by using the homomorphism $r$ discussed in 2.17 in \cite{Lu} and add the resulting set of equations  
 to deduce that  each of these  moves  also preserves the sum of all the  parameters.  
Thus, each braid move 
preserves the sum of the parameters.
 
Shortly we will  verify  
the following  three claims:   
\begin{enumerate}
\item
Given   any non-reduced subword $Q$ of $(i_1,\dots ,i_d)$ having $\delta (Q)=\delta (i_1,\dots ,i_d)$, there exists  a series of braid moves and modified nil-moves that may be 
applied  to $Q$ to  obtain a subword of $Q$  
  that has  one fewer letter  than $Q$. 
\item
Given  any non-reduced subword $Q$ of $(i_1,\dots ,i_d)$ with $\delta (Q) = \delta (i_1,\dots ,i_d)$ and any  reduced  subword $P$ of $Q$ having exactly one fewer letter than 
$Q$ with $\delta (Q) = \delta (P)$, there exists a series of braid moves and modified nil-moves that may be applied 
to $Q$ to obtain $P$.
\item
Given any two reduced subwords $R,R'$  of $(i_1,\dots ,i_d)$ with $\delta (R) = \delta (R') = \delta (i_1,\dots ,i_d)$, 
there exists a  series $R_1,R_2,\dots ,R_j$  of  steps from $R=R_1$  to $R'=R_j$ 
 where  $R_1,R_2,\dots ,R_j$  are  reduced subwords of $(i_1,\dots ,i_d)$ 
  with the further property that 
 any  two such  consecutive reduced subwords $R_i,R_{i+1}$ 
   are contained in a   non-reduced subword $N_{i,i+1}$ of $(i_1,\dots ,i_d)$  with $\delta (N_{i,i+1}) = \delta (R_i) = \delta (R_{i+1})$ where $N_{i,i+1}$  has 
   one more letter in it  than $R_i$ and $R_{i+1}$ each have. 
   \end{enumerate}
The first claim guarantees that we  may  move from any point in $f_{(i_1,\dots ,i_d)}^{-1}(p)\cap \RR_{\ge 0}^d$  whose associated word is non-reduced to a point 
in $f_{(i_1,\dots ,i_d)}^{-1}(p)\cap \RR_{\ge 0}^d$ whose associated word is reduced, doing so via moves which we have already shown each preserve 
the desired sum of parameters.  The second and third  claim combine to ensure we can move from any point in $f_{(i_1,\dots ,i_d)}^{-1}(p)\cap \RR^d_{\ge 0}$ whose associated word is reduced to any other point  in $f_{(i_1,\dots ,i_d)}^{-1}(p) \cap \RR_{\ge 0}^d$ whose associated word is also reduced,
again via moves which  each preserve the sum of the parameters.    In this manner, we deduce 
the desired  set containment  
 $f_{(i_1,\dots ,i_d)}^{-1}(p) \cap \RR^d_{\ge 0} \subseteq  f_{(i_1,\dots ,i_d)}^{-1}(p)\cap \Delta_K^{d-1}$ from the fact that all elements of 
 $f_{(i_1,\dots ,i_d)}^{-1}(p)$ have the same sum of parameters.  

Now to the proof of  claim 1.  
First note that any subword  $Q = (i_{j_1},\dots ,i_{j_s})$ of $(i_1,\dots ,i_d)$  that is not reduced has a subword  $Q' = (i_{j_l},i_{j_{l+1}},\dots ,i_{j_r})$ 
comprised of consecutive letters in $Q$ such that $Q'$  is not reduced but 
deleting 
 the  first letter  from $Q'$ yields a reduced word $Q'_2 = (i_{j_{l+1}},\dots ,i_{j_r})$ with $\delta (Q'_2) = \delta (Q)$  and
  deleting the last letter 
from $Q'$  also yields a reduced word $Q'_1 = 
(i_{j_l},i_{j_{l+1}},\dots ,i_{j_{r-1}})$ with $\delta (Q'_1) = \delta (Q)$.
 This is equivalent to saying that  $Q$ has a deletion pair 
 comprised of  $i_{j_l}$ and $i_{j_r}$. 
 By Theorem ~\ref{braid-connectedness},  we 
 may apply a series of braid moves to
 $Q'_2$ to convert it to $Q'_1$, leaving the rest of $Q$ unchanged. 
 Applying these moves within $Q$ has the effect of replacing  $Q'_2$ 
  by  a reduced word $Q''_2$  whose first  letter  
   is identical to 
   $i_{j_l}$ with this first letter  appearing  just to the right of $i_{j_l}$. 
 The resulting  pair of consecutive identical letters then  
 allows us to apply a modified nil-move to eliminate $i_{j_l}$. 
 Then we perform the reverse  series of braid moves to convert $Q''_2$ back to $Q_2$  with the net effect of 
 eliminating  one letter from $Q$, doing so in a manner  that exhibits  that claim 1 holds.

Next we show  how the proof of claim 1 can be modified slightly  to prove claim 2.  Consider $Q = (i_{j_1},\dots ,i_{j_s})$ of the form described in
claim 2.
Just as  in the proof of claim 1, $Q$ must include a subword 
$Q' = (i_{j_l},\dots ,i_{j_r})$ comprised of consecutive letters in $Q$ such that $Q'$ is not reduced  but deleting either $i_{j_l}$ or $i_{j_r}$ yields a reduced word, with  $\{ i_{j_j},i_{j_r}\} $ a deletion pair in $Q$.  Since $Q$ is assumed in claim 2  to have one more letter than a reduced word for $\delta (i_1,\dots ,i_d)$, it follows from results of 
\cite{subword} that $Q$ has exactly two subwords that are reduced; moreover, one is  
the subword obtained from $Q$  by deleting $i_{j_l}$ and the  other is the subword obtained from $Q$  by deleting $i_{j_r}$.  The proof
of claim 1 exhibits how to construct a series of braid moves and modified nil-moves transforming $Q$ into its subword in which just  $i_{j_l}$ is deleted.  A similar argument with everything mirrored will handle the case 
in which just $i_{j_r}$ is deleted.

Finally we prove claim 3. 
Here  we  utilize the fact that subword complexes are  gallery connected (due to being vertex decomposable and hence shellable, as is proven  in \cite{subword}).
   In our context, gallery connectedness  of subword complexes says exactly  that one may move from any subword  $R$  of $(i_1,\dots ,i_d)$ that is a reduced word for $w$  to any other subword  $R'$ of $(i_1,\dots ,i_d)$ that is also a reduced word for $w$  by a  series of steps $R = R_1\rightarrow R_2\rightarrow \cdots \rightarrow R_k =  R'$
 where 
$R_i$ and $R_{i+1}$ are each reduced subwords of $(i_1,\dots ,i_d)$ with Demazure product  $w$  such that there exists a non-reduced subword   $N_{i,i+1}$ of 
$(i_1,\dots ,i_d)$ 
where $N_{i,i+1}$    contains both $R_i$ and $R_{i+1}$ as subwords, satisfies 
$\delta (N_{i,i+1}) = w$ and has exactly one more letter than each of the words $R_i,R_{i+1}$.   
Thus,  $N_{i,i+1}$ has a subword $Q_{i,i+1}$  consisting  of consecutive letters of $N_{i,i+1}$  
such that one of the words $R_i,R_{i+1}$ is obtained from $N_{i,i+1}$ by deleting the leftmost letter of $Q_{i,i+1}$  
 and the other of $R_i,R_{i+1}$ is obtained from $N_{i,i+1}$  by deleting the rightmost letter of $Q_{i,i+1}$. 
 This shows that  $R_i,R_{i+1}$ are each obtained from $N_{i,i+1}$ by deleting a single letter from $N_{i,i+1}$ that belongs to a deletion pair within $N_{i,i+1}$, allowing us to use claim 2 to deduce claim 3.   
\end{proof}

\begin{example}
Lemma  ~\ref{intersect-with-simplex} shows for $p = x_1(5)x_2(7)x_1(13)$  that $f_{(1,2,1,1,2,2,1)}^{-1}(p) \cap \RR_{\ge 0}^7$ consists entirely of  points in the hyperplane in 
$\RR_{\ge 0}^7$ in which the parameters sum to $25$ since one point in $f_{(1,2,1,1,2,2,1)}^{-1}(p)\cap \RR_{\ge 0}^7$ is $(5,7,13,0,0,0,0)$ whose coordinates sum to 25.   One may
easily check in this case  that all of the strata, namely all subwords of $(1,2,1,1,2,2,1)$ having Demazure product $s_1s_2s_1$, are connected by braid moves and modified nil-moves.
\end{example}

\section{Cell  Stratification of Each Fiber}
\label{s:CW}

In this section, we prove that the fiber $f_{(i_1,\dots ,i_d)}^{-1}(p)\cap \RR_{\ge 0}^d$  of any point $p \in U(w)$ for $w = \delta (i_1,\dots ,i_d)$  has a stratification 
into open cells given by intersecting the  fiber with the various open cells in the  natural 
cell  decomposition of $\RR_{\ge 0}^d$ based on which coordinates are positive and which are 0.  
We begin by defining these strata, then give an overview of the proof before
turning to the details.

For each subword $Q$ of $(i_1,\dots ,i_d)$ having $\delta (Q) = w$ and each $p\in U(w)$, 
 we define  a stratum $F_Q\subseteq f^{-1}_{(i_1,\dots ,i_d)}(p)\cap \RR_{\ge 0}^d$ 
 as follows:

\begin{defn}\label{F-Q-def}
For each $T\subseteq \{ 1,2,\dots ,d\} $, let
$$\RR_{> 0}^T = \{ (t_1,\dots ,t_d) \in \RR_{\ge 0}^d | 
 t_i=0 \hspace{.1in}{\rm iff}\hspace{.1in}i\not\in T \}.$$  Let $\RR_{\ge 0}^T = \overline{\RR_{> 0}^T}$.
 For $Q = (i_{j_1},\dots ,i_{j_s})$ any subword of $(i_1,\dots ,i_d)$, let $S(Q) = \{ j_1,\dots ,j_s\}$ denote the support of $Q$.  Define the stratum  $F_Q$ as 
 $$F_Q = f_{(i_1,\dots ,i_d)}^{-1}(p) \cap \RR_{> 0}^{S(Q)}.$$  
Denote its closure as  $\overline{F_Q} = \cup_{Q'\subseteq Q} F_{Q'}$.
\end{defn}

The   basic  structure of the proof of Theorem A, the main result of this section,   is as follows.  
For each subword $Q$ of $(i_1,\dots ,i_d)$ having $\delta (Q) = \delta (i_1,\dots ,i_d)$
we will give a homeomorphism $f_{F_Q}$  
 from  $[0,1)^{\dim F_Q}$  to  a union of strata 
in  $f_{(i_1,\dots ,i_d)}^{-1}(p) \cap \RR_{\ge 0}^d$ which includes $F_Q$ as one of these strata.  
 We then  
 restrict  this  homeomorphism   $f_{F_Q}$ 
 to $(0,1)^{\dim F_Q}$ and show that the image of this restricted  homeomorphism   
  is  $F_Q$. 
  Finally, we show in Lemma ~\ref{decomp-is-strat}  that 
 the consequent 
 cell decomposition of $f_{(i_1,\dots ,i_d)}^{-1}(p)\cap \RR_{\ge 0}^d$  is   a cell stratification. 
 Lemmas ~\ref{generalized-Lusztig-homeom}  and  ~\ref{defined-forced-value}
 lay the groundwork  for defining  
  $f_{F_Q}$.   Well-definedness, injectivity and bijectivity   of $f_{F_Q}$ are  proven  in Lemmas ~\ref{thm:CW-map}, 
   ~\ref{f-F-Q-injective}, and  
  ~\ref{thm:CW-inverse}, respectively.  Theorem  ~\ref{thm:CW} proves that  $f_{F_Q}$ is a  homeomorphism.   

  \begin{example}\label{2-simplex-example}
  For $p = x_1(5)$  and $Q =  (1,1,1)$ in type A, the map $f_{F_Q}$ is  a homeomorphism from $[0,1)^2$ to the subset of 
  $f_{(1,1,1)}^{-1}(p)\cap \RR_{\ge 0}^3$ equalling  
  $$\{ (t_1,t_2,t_3)\in \RR_{\ge 0}^3 \mid 0\le t_1 < 5\hspace{.075in}{\rm and}\hspace{.075in}  0\le t_2 < 5-t_1\hspace{.075in}{\rm and}\hspace{.075in}  t_3 = 5-t_1-t_2 \} ;$$  
  this is exactly the subset of $f_{(1,1,1)}^{-1}(p)\cap \RR_{\ge 0}^d$  having third parameter
  positive.  The  map $f_{F_Q}$  restricts to a homeomorphism 
  from $(0,1)^2$ to 
  $$\{ (t_1,t_2,t_3)\in \RR_{\ge 0}^3 \mid  0< t_1 < 5\hspace{.075in}{\rm and}\hspace{.075in}  0 < t_2 < 5-t_1; t_3 = 5-t_1 - t_2 \} .$$   To show that the  other strata in
  $f_{(1,1,1)}^{-1}(p)\cap \RR_{\ge 0}^3$  are also 
  homeomorphic to open balls, we  likewise use 
  maps $f_{F_{Q}}$ given by the  other subwords $Q$ of $(1,1,1)$ having $\delta (Q) = \delta (1,1,1)$, for instance
  $Q = (1,1, - )$ and $Q = (- ,1,- )$.  
  \end{example}

We now  give a significantly  more detailed proof overview, as there are numerous ingredients to  the steps described above.
For any  subword $Q$ of $(i_1,\dots ,i_d)$ with $\delta (Q)=w$, let $supp(Q)$ denote the support of $Q$.  Let 
$S^C$ be the subset of $supp(Q)$ that  indexes 
the rightmost subword of $Q$  that is a reduced 
word for $w$.  Let  $S = supp(Q)\setminus S^C$. 
The domain 
$[0,1)^{\dim F_Q}$ for  the homeomorphism   $f_{F_Q}$ 
consists of  one copy of $[0,1)$ for each parameter with index in $S$, so  $\dim F_Q = |S| = |Q| -l(w)$.  
   This choice of which coordinates should contribute the half open intervals appearing in 
   $[0,1)^{\dim F_Q} $ 
   is justified in   part by  Lemma ~\ref{other-char}, a result which characterizes the rightmost subword of $Q$ 
   that is a reduced word for $w$ as consisting of exactly those  letters $i_{j_l}$ in $Q = (i_{j_1},\dots ,i_{j_{|Q|}})$ having the property  that $i_{j_l}$ is  non-redundant  (see Definition ~\ref{redundant-def})  in $(i_{j_l}, \dots ,i_{j_{|Q|}})$.  
    Lemma ~\ref{other-char} opens the door 
   to using  the results described  shortly   regarding  redundant and non-redundant letters  
    in a word $(i_1,\dots ,i_d)$ 
     to characterize the points in our aforementioned union of strata in a manner that is well-suited for defining  the maps 
     $f_{F_Q}$  and for  proving properties of  these maps  $f_{F_Q}$ for 
     each  subword $Q$ of $(i_1,\dots ,i_d)$ having $\delta (Q) = \delta (i_1,\dots ,i_d)$.
     
      \begin{remark}
  One point worth cautioning readers about before continuing our proof overview  is that we 
   typically state and prove results regarding  the word $(i_1,\dots ,i_d)$ throughout this section  
   rather than stating them  
   for more general subwords $Q$.  This is done because it often helps to be able to refer to the letters in our word even to 
   formulate our results. 
However, these  results 
are  all phrased and proven in such a way  that they apply equally well to any subword $Q$ having $\delta (Q) = \delta (i_1,\dots ,i_d)$.  
 \end{remark}
     
  Lemma    ~\ref{0-iff-max-to-left}  combines many of our  upcoming results 
       to describe in a useful way 
       a subset 
         of $f_{(i_1,\dots ,i_d)}^{-1}(p)\cap \RR_{\ge 0}^{d}$ ,  denoted $D^{<\max }$, 
      that  
      we prove in  Theorem   ~\ref{thm:CW} is 
     homeomorphic to 
     $[0,1)^{|S|}$ and contained $F_Q$. 
     
     \begin{example}
     For $p = x_1(5) $ and $Q=(1,1,1)$,  as in 
     Example  ~\ref{2-simplex-example}, 
     $$D^{<\max } = \{ (t_1,t_2,t_3)\in \RR_{\ge 0}^3 \mid 0 \le t_1 < 5; 0\le t_2 < 5-t_1; t_3 = 5-t_1-t_2\} \cong [0,1)^2.$$    In this 
     case, $t_1$ and $t_2$ each give rise to a half-open interval of possible values.  
     The upper bound on the interval for $t_2$ depends on the value chosen for $t_1$ in a way that is a continuous function
     of this value for $t_1$.  
     \end{example} 
     
   In Corollary ~\ref{closed-interval}, we show  for each $t_j$ with index  $j\in S$ 
   and each suitable  choice of values $k_1,\dots ,k_{j-1}\in \RR_{\ge 0}^d$  
 for  the  parameters
  $t_1,\dots ,t_{j-1}$ 
(i.e.  
extendable to  some $(k_1,\dots ,k_{j-1},t_j,\dots ,t_d)\in f_{(i_1,\dots ,i_d)}^{-1}(p)\cap \RR_{\ge 0}^d$ and satisfying certain non-maximality
constraints) 
that   there exists exactly  a closed interval $[0, t_j^{\max }(k_1,\dots ,k_{j-1})]$  of values $k_j$  that $t_j$  may take so that  there exists
$(k_1,\dots  ,k_j,t_{j+1},\dots ,t_d)\in f_{(i_1,\dots ,i_d)}^{-1}(p)\cap \RR_{\ge 0}^d$.
 Moreover, we prove  $t_j^{\max }(k_1,\dots ,k_{j-1}) >0$. 
   This is  all done  using  upcoming results  in  this section  regarding redundant letters   in a word $(i_1,\dots ,i_d)$.  
 Lemmas  ~\ref{last-letter-determined}, 
 ~\ref{generalized-Lusztig-homeom} and ~\ref{defined-forced-value}  
  show 
 for $t_j$ with index $j \in S^C$ and for  any suitable  choice of values $k_1,\dots ,k_{j-1}$ for the parameters $t_1,\dots ,t_{j-1}$  
(again meaning extendable to a point $(k_1,\dots ,k_{j-1},t_j,\dots ,t_d)\in f_{(i_1,\dots ,i_d)}^{-1}(p)\cap \RR_{\ge 0}^d$ and satisfying certain 
non-maximality constraints)
 that 
 there exists   a unique  value $k_j>0$ that  $t_j$ may take  such  that 
 $k_1,\dots ,k_{j-1},k_j$ is  extendable to a point  $(k_1,\dots ,k_d) \in f_{(i_1,\dots ,i_d)}^{-1}(p)\cap \RR_{\ge 0}^d$. 
  This 
    is done  using  results   in this section regarding   the
 non-redundant letters   in  a word
  $(i_1,\dots ,i_d)$. 
 
 Our  aforementioned non-maximality requirements exactly 
 rule out choosing the maximal possible value for each parameter $t_j$ given by a letter 
 $i_j$ that is redundant in $(i_j,\dots ,i_d)$.
The reason 
we choose to  focus on the  resulting  half-open intervals 
 $[0,t_j^{\max }(k_1,\dots ,k_{j-1}))$ of possible values for these parameters is  
 because  
  the resulting  smaller region $D^{< \max }$ within $f^{-1}_{(i_1,\dots ,i_d)}(p) \cap \RR_{\ge 0}^d$ 
 is  where the combinatorics  is well-behaved while  still being large enough to allow us to justify  that we get a cell decomposition
 for all of $f^{-1}_{(i_1,\dots ,i_d)}(p)\cap \RR_{\ge 0}^d$.  
 To describe $D^{< \max }$, we proceed through parameters from left to right.  Our results show  for the region $D^{<\max }$  that
 each redundant parameter  $t_j$ gives  a half-open interval  $[0,t_j^{\max }(k_1,\dots ,k_{j-1}))$ of possible 
 values 
  while the value for each non-redundant parameter $t_j$  is 
 uniquely determined as a function $f_j$ of  the values  $k_1,\dots ,k_{j-1}$ of the parameters to its left;  these results  
 get  combined  to obtain  a bijective function $f_{F_Q}$  from  $[0,1)^{|S|}$ to  a union of 
 strata.     

  In Lemma ~\ref{cont-forced-value}, we prove   
  for each  $j\in S^C$ that 
  $f_j$   is a continuous function 
  of  the values $k_1,\dots ,k_{j-1}$ chosen for 
 $t_1,\dots ,t_{j-1}$.   In   Lemma ~\ref{cont-max}, we  prove   for each  $j\in S$ 
 that $t_j^{\max }$ is   a  continuous function of the values $k_1,\dots ,k_{j-1}$ chosen for 
  $t_1,\dots ,t_{j-1}$. 
   These are  two of the  main  ingredients 
   in our proof that $f_{F_Q}$ is not just a bijection from $[0,1)^{|S|}$ to a union of strata  but is a homeomorphism.   
   To this end, a  key observation   is given in Lemma ~\ref{max-to-unique} 
   where we  interpret $t_1^{\max }$ as the unique value $t_1$ may take when we replace 
   $(i_1,\dots ,i_d)$ by any  subword $Q$ having $\delta (Q) = \delta (i_1,\dots ,i_d)$ 
    with the property that  $i_1$ redundant in $(i_1,\dots ,i_d)$ but non-redundant in $Q$; 
      this interpretation for $t_1^{\max }$ (and more generally for 
    $t_j^{\max }(k_1,\dots ,k_{j-1})$ for each $j\in S$) allows  us to deduce continuity for $t_j^{\max }$ for $j\in S$ from continuity of $f_j$ for    $j\in S^C$ (and is also used at various other places in this section as a helpful reduction step).

        We refer readers to Figure   ~\ref{big-strat-space} 
    and to  the 
     examples   in Sections ~\ref{intro-section} and ~\ref{sub:whole-fiber} 
    for examples of the types of stratifications that we now   analyze in detail  throughout the remainder of this section.

        \begin{figure}\label{big-strat-space}
\begin{tikzpicture}[scale=0.4]
%
\draw (0,0)--(0,8)--(2,10)--(12,12);
\draw (0,0)--(8,0);
\draw (8,0)--(12,4)--(12,12); 

\draw [thick] (0,0)--(0,8)--(2,10)--(12,12);
\draw [thick] (0,0)--(8,0)--(12,4)--(12,12);
\draw [thick] (0,8)--(8,0);

\draw [thick, dashed] (0,0)--(4,4)--(4,9)--(2,10);
\draw [thick, dashed] (4,9)--(12,12);
\draw [thick, dashed] (4,4)--(12,4);

\node [left] at (0,8){\scriptsize $(0,0,k_1,k_2,k_3,0$)}; 
\node [below left] at (0,0){\scriptsize $(k_1,0,0,k_2,k_3,0$)}; 

\node [above right] at (3.9,7.7){\scriptsize $(k_1,k_2,k_3,0,0,0)$}; 

\node [above right] at (3.9,4){\scriptsize $(k_1,k_2,0,0,k_3,0)$}; 

\node [above left] at (2.2,9.8){\scriptsize $(0,\frac{k_2k_3}{k_1+k_3},k_1+k_3,\frac{k_1k_2}{k_1+k_3},0,0)$}; 

\node [right] at (12,12){\scriptsize $(0,\frac{k_2k_3}{k_1+k_3},k_1+k_3,0,0,\frac{k_1k_2}{k_1+k_3})$}; 

\node [below right] at (11.8,4.2){\scriptsize $(0,\frac{k_2k_3}{k_1+k_3},0,0,k_1+k_3,\frac{k_1k_2}{k_1+k_3})$};
\node [below right] at (7.8,0.2){\scriptsize $(0,0,0,\frac{k_2k_3}{k_1+k_3},k_1+k_3,\frac{k_1k_2}{k_1+k_3})$};

\end{tikzpicture}
%
%
%
\caption{Cell stratification in type A for $f^{-1}_{(1,2,1,2,1,2)}(M)$ for any  
$M = x_1(k_1)x_2(k_2)x_1(k_3) \in U(s_1s_2s_1)$}
\end{figure}

     This section makes extensive use of the Demazure product on 
 a Coxeter system $(W,\Sigma )$, as this nonstandard 
 product naturally reflects the relations $x_i(t_1)x_i(t_2) = x_i(t_1+t_2)$ for $t_1,t_2\in \RR$ amongst exponentiated Chevalley generators $x_i(t)$ 
 as well as  reflecting  the degree $m(i,j)$	braid relations  $$x_i(t_1)x_j(t_2)x_i(t_3)\dots = x_j(t_1')x_i(t_2')x_j(t_3')\dots $$  
 for any  choice of $t_1,t_2,\dots ,t_{m(i,j)}>0$ where 
 $t_1',t_2',\dots ,t_{m(i,j)}'$ are also positive real numbers   (proven to exist  and be uniquely determined by Lusztig in Proposition ~\ref{Lu-0th-prop}).  
 These generators and  relations   turn out to 
govern the structure of the fibers of  the map $f_{(i_1,\dots ,i_d)}$ we are studying, as one might hope based on the fact that 
 $f_{(i_1,\dots ,i_d)}(t_1,\dots ,t_d)  = x_{i_1}(t_1)x_{i_2}(t_2)\cdots x_{i_d}(t_d)$.

\begin{defn}\label{redundant-def}
The letter $  i_{j}$ is {\it  redundant} in the word
$(i_1,\dots ,i_{d}) $ if   
$$\delta (i_1,\dots ,i_{j-1},i_j,i_{j+1}, \dots ,i_d) = \delta (i_1,\dots ,i_{j-1},\hat{i}_j,i_{j+1},\dots ,i_d). $$
The letter  $i_j$ is {\it non-redundant} in $(i_1,\dots ,i_d)$  otherwise, i.e., 
if $$\delta (i_1,\dots ,i_{j-1},i_j, i_{j+1}, \dots ,i_d)\ne \delta (i_1,\dots ,i_{j-1},\hat{i}_j, i_{j+1},\dots ,i_d).$$
\end{defn}

\begin{example}
The last letter in each of the words $(1,1)$ and  $(1,2,1,2)$ is redundant while the last letter of $(1,2,1,2,3)$ is non-redundant.  
\end{example}

\begin{remark}
One may show that if $i_j$ is redundant in $(i_1,\dots ,i_d)$  for some $j$ satisfying  $1 \le j \le d$ then this implies  $i_j$ is redundant in either $(i_1, \dots , i_j)$ or  $(i_j, \dots, i_d)$.   
\end{remark}

Proposition ~\ref{Demazure-cancellation} below  
will be  helpful for reducing statements about Demazure products to  statements about Coxeter theoretic products, a setting where statements  are more readily proven.  The latter product has the advantage that a cancellation law holds, due to the presence of inverses of elements.    


\commenth{I just changed the next proposition since we actually use it  in more generality than how it was stated -- specifically,
we may have different numbers of letters in the two words having equal Demazure product to each other, so I changed some instances of $d$ to $d'$ throughout to fix this.}

\begin{prop}\label{Demazure-cancellation}
If $\delta (i_1,\dots ,i_{d-1},i_d) = \delta (j_1,\cdots ,j_{d'-1},i_d)$ 
with   $i_d $  non-redundant in
$(i_1,\dots ,i_d)$ and in  $(j_1,\dots ,j_{d'-1},i_d)$, then 
$\delta (i_1,\dots ,i_{d-1}) = \delta (j_1,\dots ,j_{d'-1})$.

Likewise, if $\delta (i_1,i_2,\dots ,i_d) = \delta (i_1,j_2,\dots ,j_{d'})$
with  $i_1$  non-redundant in  $(i_1,\dots ,i_d)$ and  in $(i_1,j_2,\dots ,j_{d'})$, 
then $\delta (i_2,\dots ,i_d) = \delta (j_2,\dots ,j_{d'})$.
\end{prop}

\begin{proof}
First consider the statement with $i_d$ non-redundant.
Non-redundancy of $i_d$ implies $\delta (i_1,\dots, i_d) = \delta (i_1,\dots ,i_{d-1}) s_{i_d}$ and 
$\delta (j_1,\dots ,j_{d'-1},i_d) = \delta (j_1,\dots ,j_{d'-1})s_{i_d}$.  Since $\delta (i_1,\dots ,i_d) = \delta (j_1,\dots ,j_{d'-1},i_d)$, we may multiply both sides of 
$\delta (i_1,\dots ,i_{d-1})s_{i_d} = \delta (j_1,\dots ,j_{d'-1})s_{i_d}$ on the right by $s_{i_d}$ to obtain the desired result.
The proof of the statement regarding non-redundant $s_{i_1}$ is similar with everything mirrored.
\end{proof}

\begin{example}
Since  $\delta (1,1,2,1) = \delta (1,2,1)$ with the rightmost letter in each  non-redundant and having the same value in both, this implies $\delta (1,1,2) = \delta (1,2)$.
\end{example}

Next we give a relaxation  of the notion of deletion pair that will be used extensively in upcoming proofs later in this section. 

\begin{defn} 
In $(i_1,\dots ,i_d)$ we say that $i_r$ and $i_s$ are  
{\it  deletion partners},  
 if   they 
%
%
%
become a deletion pair after replacing $(i_r,\dots ,i_{s-1})$ by a reduced  subword  having $i_r$ as its leftmost letter 
which  
  has the same Demazure 
product as $(i_r,\dots ,i_{s-1})$., 
\end{defn}

\begin{example}
In type A,  the first and fifth  
letters in $(1,2,2,1,2)$ 
are deletion partners, but they are not a deletion pair.   The second and fifth letters are neither deletion partners nor a deletion pair.   The first and fifth letters
in $(1,1,1,1,1)$ 
are  deletion partners but not a deletion pair. 
\end{example}


\commenth{I changed the justification in the next remark since the old justification did not seem like the right way to explain this fact.  The old explanation remains commented
out.}

\begin{remark}\label{last-letter-shortens}
Notice that   $w = \delta (i_1,\dots ,i_d) $ implies  $l(ws_{i_d}) < l(w)$ by virtue of  $w$ having 
a reduced expression whose rightmost letter is $s_{i_d}$; such a reduced expression 
may be obtained by proceeding from right to left through $(i_1,\dots ,i_d)$ including each 
simple reflection such that multiplying on the left by this simple reflection 
does not reduce the Coxeter-theoretic length  of the Coxeter group  element obtained so far. 
 By similar reasoning, $l(s_{i_1}w) < l(w)$ for 
$w= \delta (i_1,\dots ,i_d)$.  
\end{remark}

Next are results  regarding $f_{(i_1,\dots ,i_d)}$ in the case where $i_1$ is non-redundant.  These 
 will  
enable  us  inductively to   carry out a reduction 
 from  
$(i_1,\dots ,i_d)$  to a smaller problem given by the shorter word 
$(i_2,\dots ,i_d)$  in the case  when  $i_1$ is non-redundant.

\begin{lemma}\label{last-letter-determined}
Consider any  $w\in W$,  any  word $(i_1,\dots  ,i_d)$ having $\delta (i_1,\dots ,i_d) = w$, and  any $p\in U(w)$.  The letter $ i_{1}$ 
 is non-redundant
in the word $(i_1,\dots ,i_{d}) $
if and only if there is a unique positive value for $t_{1}$ 
among the points 
$(t_1,\dots ,t_d)\in f_{(i_1,\dots ,i_{d})}^{-1}(p)\cap  \RR_{\ge 0}^d$. 

Likewise $i_d$ is non-redundant in the word $(i_1,\dots ,i_d) $ if and only if there is a unique positive value for $t_d$  among the points 
$(t_1,\dots ,t_d) \in f_{(i_1,\dots ,i_d)}^{-1}(p)\cap \RR_{\ge 0}^d$. 
\end{lemma}

\begin{proof}
First suppose that $i_1$ is non-redundant.  Let us begin by proving the existence of a 
solution  to $f_{(i_1,\dots ,i_d)}(t_1,\dots ,t_d) = p$ with $t_1>0$. 
Non-redundancy of $i_1$ implies 
that $\delta (i_1,\dots ,i_d) =  s_{i_1} \delta (i_2,\dots ,i_{d})$. 
This allows us
to construct for  $w$ 
a reduced word
  $(i_1,i_{j_2},i_{j_3},\dots ,i_{j_{d'}})$ where $(i_{j_2},\dots ,i_{j_{d'}})$ is a 
    reduced word for $\delta (i_2,\dots ,i_d)$.  But then  Theorem ~\ref{Lu-theorem} 
gives  that 
$f_{(i_1,i_{j_2},\dots ,i_{j_{d'}})}$ is a homeomorphism from $\RR_{>0}^{d'}$ to $U(\delta (i_1,\dots ,i_d))$.  In particular,  
$f_{(i_1,i_{j_2},\dots ,i_{j_{d'}})}$ is surjective, 
implying  
there exists  $(t_{1},t_{j_2},\dots ,t_{j_{d'}})\in \RR_{>0}^{d'}$ such that
$ f_{(i_1,i_{j_2},\dots ,i_{j_{d'})}}(t_{1},t_{j_2},\dots ,t_{j_{d'}}) = p$.  
We  obtain from  such  $(t_1,t_{j_2},\dots ,t_{j_{d'}})$ 
the desired  $(t_1,\dots ,t_d)\in \RR_{\ge 0}^d$ satisfying 
$f_{(i_1,\dots ,i_d)}(t_1,\dots ,t_d)=p$  by  
 inserting  0's  into  $(t_1,t_{j_2},\dots ,t_{j_{d'}})$  at the positions of those  coordinates in the indexing set for $\RR_{\ge 0}^d$  whose index is in 
  $\{ 1,2,\dots ,d\} \setminus 
 \{ 1,j_2,\dots ,j_{d'} \}$. 
 
Next we prove uniqueness of the   positive value  for 
$t_1$ within $f_{(i_1,\dots ,i_d)}^{-1}(p)\cap \RR_{\ge 0}^d$. 
Suppose
$$%
  f_{(i_1,\dots ,i_{d})}(t_1,\dots ,t_{d}) = f_{(i_1,\dots
  ,i_{d})} (t_1',\dots , t_{d}') = p
$$
for $(t_1,\dots ,t_d), (t_1',\dots ,t_d')\in \RR_{\ge 0}^d$ satisfying  $t_1 \ne t_1'$.  
Suppose $t_{1} < t_{1}'$.  The argument below  is entirely symmetric for $t_1 > t_{1}'$, so that case is left to the reader.  Having $t_1 < t_1'$ implies 
$$x_{i_1}(t_1'-t_1) x_{i_2}(t_2')
 \cdots 
 x_{i_{d}}(t_{d}') =
x_{i_1}(-t_1) p 
 = x_{i_2}(t_2) \cdots x_{i_{d}}(t_{d}).$$ But then
  $x_{i_1}(t_1'-t_1)x_{i_2}(t_2')\cdots x_{i_d}(t_d')  
 \in U(w)$ since we must have $x_{i_2}(t_2')\cdots x_{i_d}(t_d')\in U(s_{i_1}w)$  in order to have 
 $x_{i_1}(t_1')\cdots x_{i_d}(t_d')\in U(w)$.  
 On the other hand,  we likewise must have  $x_{i_2}(t_2)\cdots  x_{i_d}(t_d) 
\in 
U(s_{i_1}w)$ by Remark ~\ref{last-letter-shortens}.  But $U(w) \cap U(s_{i_1}w) = \emptyset $,  contradicting 
the fact we have just shown that  $x_{i_1}(-t_1) p$ is in both $U(w)$ and $U(s_{i_1}w)$.  
 This rules out the possibility of having $t_1'\ne t_1$,  completing one direction. 

Turning now  to  the other direction,  suppose that $i_{1}$ is redundant.  By Remark ~\ref{Demazure-connected}, 
we may apply braid and
modified nil-moves to the subword of $(i_2,\dots ,i_{d})$ comprised of those positions with nonzero parameters 
to obtain from this   a reduced word
$(j_1,\dots ,j_{d'})$ having 
$\delta (i_2,\dots ,i_{d}) = \delta (j_1,\dots ,j_{d'})$.  Applying the change of coordinates maps of Lusztig (see  Proposition ~\ref{Lu-0th-prop}) 
 as we apply each of these moves,
 we thereby  obtain from $x_{i_1}(t_1)x_{i_2}(t_2)\cdots x_{i_d}(t_d)$ some
$x_{i_1}(t_1)x_{j_1}(t_1')x_{j_2}(t_2')\cdots x_{j_{d'}}(t_{d'}')$ also equalling $p$.  
Notice that  $i_{1}$ must be redundant within 
$(i_1,j_1,\dots ,j_{d'})$ 
since $\delta (i_2,\dots ,i_{d}) = \delta (j_1,\dots ,j_{d'})$ and since $i_1$ is redundant in $(i_1,i_2,\dots ,i_d)$; 
this combined with that fact that $(j_1,\dots ,j_{d'})$ is a reduced word yields that  $i_1$ 
must be in a deletion pair with some  $j_r$ within the word $(i_1,j_1,\dots ,j_{d'})$. 
By Theorem ~\ref{braid-connectedness},
we may apply a series of braid moves to $(j_1,\dots ,j_{r})$ to obtain  a new reduced word  $(j_1',\dots ,j_{r}')$ such that 
$j_{1}' = i_1$.   The change of coordinate maps corresponding to these braid moves  then yield some 
$$x_{i_1}(t_1)x_{j_1'}(t_2'')\cdots x_{j_{r}'}(t_r'')x_{j_{r+1}}(t_{r+1}')\cdots x_{j_{d'}}(t_{d'}')$$
also equalling $p$.
Having $j_1'=i_1$  implies that 
we may replace $t_1$ by $t_1+\epsilon $ while replacing $t_{2}''$ by $t_{2}'' - \epsilon $ 
 without changing the value of the product $p$.
 Choosing any   $\epsilon $ satisfying $0 < \epsilon < t_2'' $ will allow us to exhibit the desired 
non-uniqueness of
$t_1>0$  in  this case.  The point is that  after replacing $t_1$ by $t_1+\epsilon $ and $t_2'' $  by $t_2'' -\epsilon $, 
we may undo the series of braid moves and modified nil-moves that transformed $(i_2,\dots ,i_d)$ to $(j_1',\dots ,j_r',j_{r+1},\dots ,j_{d'})$ doing appropriate 
changes of coordinates at each step to obtain an element of $f_{(i_1,\dots ,i_d)}^{-1}(p)\cap \RR_{\ge 0}^d$ with the new value $t_1+\epsilon $ for the first 
parameter, completing our proof of non-uniqueness of the value for this first parameter.  

The proof is entirely analogous but  mirrored  for the statement involving $i_d$ and $k_d$  in place of $i_1$ and $k_1$.
\end{proof}

\begin{remark}
Lemma ~\ref{last-letter-determined} had two parts, the first of which may be regarded as a ``left'' version and the second of which is a ``right'' version 
of the same type of statement with everything mirrored.
Throughout this paper, for any result statement that  can be mirrored to give a different statement, both statements hold by simply mirroring the proof. 
In the remainder of the paper, we omit these mirrored statements, leaving those to the interested  reader to write down.  
\end{remark}

\begin{example}
Consider the word $(1,2,2)$ in type A, in which case the leftmost letter is non-redundant.  Lemma ~\ref{last-letter-determined} shows  every point 
$$(t_1,t_2,t_3)\in f_{(1,2,2)}^{-1}(x_1(3)x_2(12))\cap \RR_{\ge 0}^3$$ has $t_1=3$ since there exists a point $(3,5,7)\in f_{(1,2,2)}^{-1}(x_1(3)x_2(12))\cap \RR_{\ge 0}^3$  having $t_1=3$.

Since the rightmost letter in $(1,2,2)$ is redundant, there is not a unique value for $t_3$ in the fiber
$f_{(1,2,2)}^{-1}(x_1(3)x_2(12))\cap \RR_{\ge 0}^3$.  One may check that $t_3$ may be 
chosen to be any real number in the closed interval $[0,12]$, so is not uniquely specified.    

In our approach to giving a cell  decomposition of this fiber, 
we do not focus on the rightmost letter being redundant here but rather  utilize the fact that the middle letter in $(1,2,2)$ is redundant and 
remains redundant when we restrict to the subword $(2,2)$ having it as the leftmost letter.
This  will imply there  is more than one possible 
value for $t_2$, namely that $t_2$ can take any value in $[0,12]$.  Our cell
decomposition  
will involve choosing any value $k_2\in [0,12)$ for $t_2$   and then showing  that this uniquely determines the value for 
$t_3$ as $t_3 := 12-k_2$.  The part of $f_{(1,2,2)}^{-1}(x_1(3)x_2(12))\cap \RR_{\ge 0}^3$  that is missed  by our only considering $k_2<12$
will be exactly the points having  $t_3=0$.  Requiring $t_3=0$  effectively allows us to restrict attention to the first two parameters and  consider
the fiber  $f_{(1,2)}^{-1}(x_1(3)x_2(12))$ where we may 
apply our results for $(i_1,\dots ,i_d) = (1,2)$ to give a separate cell decomposition for the part of our fiber having  $k_2=12$ (namely the part of the fiber having $k_3=0$).  
\end{example}

A consequence of the proof of Lemma ~\ref{last-letter-determined} that we will need later is the following:

\begin{corollary}\label{unique-value-other-cell}
Consider any $w\in W$, any $p\in U(w)$, and any word $(i_1,\dots ,i_d)$ satisfying $\delta (i_1,\dots ,i_d) = w$ and having 
 $i_1$  non-redundant in $(i_1,\dots ,i_d)$. 
  If
$(k_1,\dots ,k_d)\in f_{(i_1,\dots ,i_d)}^{-1}(p) \cap \RR_{\ge 0}^d$, then 
$$x_{i_1}(-k_1)p\in U(\delta (i_2,\dots ,i_d)).$$  
Conversely,  if $x_{i_1}(-k_1)p\in 
U(\delta (i_2,\dots ,i_d))$ 
for $k_1\in \RR_{\ge 0}$,  then  there exist
$k_2,\dots ,k_d$ such that 
$(k_1,\dots ,k_d)\in f_{(i_1,\dots ,i_d)}^{-1}(p)\cap \RR_{\ge 0}^d$.   It also follows that  $k_1>0$.  
\end{corollary}

\begin{proof}
First observe that $i_1$ non-redundant in $(i_1,\dots ,i_d)$ implies $\delta (i_2,\dots ,i_d) = s_{i_1}w$.
The forward direction, namely the implication that  $(k_1,\dots ,k_d)\in f_{(i_1,\dots ,i_d)}^{-1}(p)\cap \RR_{\ge 0}^d$ implies $x_{i_1}(-k_1)p\in U(\delta (i_2,\dots ,i_d))$, is proven in the second paragraph of the proof of Lemma ~\ref{last-letter-determined}.  The reverse  direction follows  from the definition of $U(\delta (i_2,\dots ,i_d))$  in combination  with 
Theorem ~\ref{Lu-theorem} which ensures the existence of 
$k_2,\dots ,k_d$ such  that $x_{i_2}(k_2)\cdots x_{i_d}(k_d) = x_{i_1}(-k_1)p$.  The $k_1>0$ claim follows from the fact that $k_1=0$ would imply 
$p = x_{i_1}(-k_1)p$ which would contradict $U(w)\cap U(s_{i_1}w) = \emptyset$.    
\end{proof}

\begin{example}
Again consider $p = x_1(3)x_2(12) = f_{(1,2,2)}(3,5,7) \in U(s_1s_2) = U (\delta (1,2,2))$.  Since the leftmost letter in $(1,2,2)$ is non-redundant, we have 
$x_1(-3)p = x_2(12) \in U(s_2) = U(\delta (2,2))$ where 3 is the unique value $t_1$ may take within $f_{(1,2,2)}^{-1}(p)\cap \RR_{\ge 0}^3$, as predicted by Corollary ~\ref{unique-value-other-cell}.    
\end{example}

\begin{lemma}\label{full-theorem}
Consider any  $w\in W$,  any $p\in U(w)$,  and  any word $(i_1,\dots ,i_d)$  satisfying   $\delta (i_1,\dots ,i_d) = w$. 
Suppose that   $i_{1}$  
 non-redundant in $(i_1,\dots ,i_d)$, and let  $u=\delta (i_2,\dots ,i_d)$, which together  implies  $u = s_{i_1}w  < w$.  
 Let  $k_1$ 
be the unique value that $t_1$ 
 takes within
 $f_{(i_1,\dots ,i_d)}^{-1}(p)\cap \RR_{\ge 0}^d$.  
 Then   there exists $p' \in U(u)$  such that 
$f_{(i_1,\dots ,i_{d})}^{-1}(p)\cap \RR_{\ge 0}^{d}  \cong  f_{(i_2,\dots ,i_{d})}^{-1}(p')\cap \RR_{\ge 0}^{d-1}$.
More specifically,  
 $(t_1,\dots ,t_d)\mapsto (t_2,\dots ,t_{d}) $  
 gives rise to such a homeomorphism 
 with  $p' = x_{i_1}(-k_1)p$. 
%
\end{lemma}

\begin{proof}
The  nonnegative real solutions to the  equation
$$x_{i_1}(t_1)x_{i_2}(t_2)\cdots x_{i_{d}}(t_{d}) = p$$ 
all have $t_1 = k_1$ for a unique value $k_1$ by Lemma ~\ref{last-letter-determined}.  
Thus, the 
map  $(t_1,\dots ,t_d) \rightarrow (t_2,\dots ,t_{d})$  gives a bijection from  the set of nonnegative real solutions 
to the equation 
$$x_{i_1}(t_1)\cdots
x_{i_{d}}(t_{d}) = p$$
 to the set of nonnegative real solutions to the equation 
$$x_{i_2}(t_2)\cdots x_{i_{d}}(t_{d}) = p'.$$
 This   is continuous due to being a projection map.
To see that the  inverse  is also 
continuous, note that  its coordinate functions are each continuous since one is a constant map and the others are identity maps.
\end{proof}

Lemma ~\ref{full-theorem}  motivates the next definition. 

\begin{defn}\label{cf-def}
Let {\it  cf}  (short for {\it  ``change-fiber''})  be the  map taking as its input a
point $p\in U(w)$ for some $w\le \delta (i_1,\dots ,i_d)$  and an $r$-tuple of  values 
$k_1,\dots ,k_r$ for  an initial segment of  
parameters $t_1,\dots ,t_r$ for some $r\ge 1$.   The map  $cf$ 
 outputs  a  point  $q \in U(u)$ for some $u\le w$ as follows. 
Let 
$cf(p; k_1,\dots ,k_r) : =  x_{i_r}(-k_r) \cdots x_{i_1}(-k_1) p$. 
\end{defn}

In light of Lemma ~\ref{full-theorem}, the map $cf$  will be used to reduce  the problem of understanding a fiber
$f_{(i_1,\dots ,i_d)}^{-1}(p)\cap \RR_{\ge 0}^d$ in which $i_1$ is non-redundant 
 to instead studying a fiber $f_{(i_2,\dots ,i_d)}^{-1}(p')\cap \RR_{\ge 0}^{d-1}$ 
for $p' =cf(p; k_1) =  x_{i_1}(-k_1)p$  where $k_1$ is the unique value that $t_1$ takes  in $f_{(i_1,\dots ,i_d)}^{-1}(p)\cap \RR_{\ge 0}^d$.

Next we give results  pertaining to the map   $f_{(i_1,\dots ,i_d)}$  when   $i_1$  is redundant in $(i_1,\dots ,i_d)$ and more generally when
$i_l$ is redundant in 
$(i_l,\dots ,i_d)$ for some $l \in \{ 1,2,\dots ,d\} $.  

\begin{lemma}\label{max-achieved}
Consider  any $w\in W$, any $p\in U(w)$, and any  word $(i_1,\dots ,i_d)$ satisfying $\delta (i_1,\dots ,i_d) = w $.  Suppose  $i_1$ 
  is redundant in $(i_1,\dots ,i_d)$.
Then there exists a largest real  value $t_1^{\max }$  
that $t_1$ 
takes within the
set of solutions $(t_1,\dots ,t_d)\in \RR_{\ge 0}^d$ to the 
equation $x_{i_1}(t_1)\cdots x_{i_d}(t_d) = p$. 
 This value   $ t_1^{\max }$  
is the 
unique nonnegative real 
number $k_1$   
such that $x_{i_1}(-k_1) p  \in U(s_{i_1}w)$.   Moreover, $t_1^{\max }$ is positive.  
\end{lemma}

\begin{proof}
First we prove existence of  $k_1\in \RR_{>0}$ satisfying $x_{i_1}(-k_1)p \in U(s_{i_1}w)$.  
 Proceed from left to right in $(i_1,\dots ,i_{d})$ deleting each deletion partner one encounters for $i_1$ in turn from the subword of $(i_1,\dots ,i_{d})$ that already has deleted from it  all of the deletion partners for $i_1$ already encountered.   This yields  a subword of $(i_1,\dots ,i_d)$ that includes $i_1$ as a 
 non-redundant letter and still has  Demazure product  $w$.   
 Now take  a subword $(i_{j_1},\dots ,i_{j_r})$  of this word obtained so far, choosing 
 $(i_{j_1},\dots ,i_{j_r})$ to be a reduced word 
 for $w$ having $j_1=1$; Remark ~\ref{subword_for_Demazure_product} explains why we can always do this.  
 This produces a new word which will have $i_1$ as non-redundant and will be  reduced.  Set the parameters of $(t_1,\dots ,t_d)$  corresponding to the deleted letters to 
 0.  This then  allows us to invoke Lemma \ref{last-letter-determined} and \ref{full-theorem}  to determine a unique 
value $k_1\ge 0$ 
 such that $x_{i_1}(-k_1)p \in U(s_{i_1}w)$. 
Since $p\in U(w)$ while  $x_{i_1}(-k_1)p \in U(s_{i_1}w)$  for this $k_1$  
and since  Proposition \ref{Lu-first-prop} gives us that  $U(w) \cap U(s_{i_1}w)=\emptyset $, the inequality 
 $p\ne x_{i_1}(-k_1)p$ follows from this.  This in turn implies $k_1\ne 0$ since $p = x_{i_1}(0)p \ne x_{i_1}(-k_1)p$.   

Next we show that any  $k_1>0$ satisfying  $x_{i_1}(-k_1)p \in U(s_{i_1}w)$  indeed comes from some $(k_1,\dots ,k_d)\in \RR_{\ge 0}^d$ 
satisfying $x_{i_1}(k_1)\cdots x_{i_d}(k_d) = p$.   
 To this end, choose  a subword  $(i_{j_1},\dots ,i_{j_r})$
of $(i_1,\dots ,i_d)$ having  $j_1 = 1$ such that $(i_{j_1},\dots ,i_{j_r})$  is a reduced 
word for $w$.  This implies  
that $(i_{j_2},\dots ,i_{j_{r}})$ is a reduced word for 
$s_{i_1}w$. 
 Setting all parameters $t_j$ for $j\not \in \{ j_1,\dots ,j_r \} $ to 0 allows us  
then  to use  Lusztig's result that $f_{(i_{j_2},\dots ,i_{j_{r}})}$ is a homeomorphism from
$\RR_{>0}^{r-1}$ to $U(s_{i_1}w)$ since $(i_{j_2},\dots ,i_{j_{r}})$ is reduced. 
 This being a homeomorphism  in particular implies surjectivity to $U(s_{i_1}w)$, hence implies  the existence of a solution 
 $(k_{j_2},\dots ,k_{j_r})$  within
 $\RR_{>0}^{r-1}$ to the equation 
$f_{(i_{j_2},\dots ,i_{j_{r}})}(t_{j_2},\dots , t_{j_{r}}) =  x_{i_1}(-k_1)p 
\in U(s_{i_1}w)$.   
Equivalently, $(k_1,k_{j_2},\dots ,k_{j_r})$ is  
a solution for $f_{(i_{j_1},\dots ,i_{j_r})}(t_{j_1},\dots ,t_{j_r}) = p$ within $\RR_{>0}^{r}$.  Finally, we take this solution for the coordinates indexed by
$\{ 1,j_2,\dots ,j_r \}$ and set all other parameters to 0 to get the desired solution $(k_1,\dots ,k_d)$  
 to $f_{(i_1,\dots ,i_d)}(t_1,\dots ,t_d) = p$ with 
$(k_1,\dots ,k_d)\in \RR_{\ge 0}^d$.  

Given any  $k_1>0$ such that $x_{i_1}(-k_1)p \in U(s_{i_1}w)$, suppose there is also some  $k_1'>  k_1 $  with 
$x_{i_1}(-k_1')p \in U(s_{i_1}w)$.    Let  $(i_{j_2},\dots ,i_{j_r})$ be a  subword of $(i_2,\dots ,i_d)$ that is a reduced word for $s_{i_1}w$, as constructed in the last 
paragraph.  We justified there that $x_{i_1}(-k_1)p\in U(s_{i_1}w)$ implies the existence of $(k_2,\dots ,k_r)\in \RR_{>0}^{r-1}$ such that 
$f_{(i_1,i_{j_2},\dots ,i_{j_r})}(k_1,k_2,\dots ,k_r) = p$.  But then since we also are assuming we  have $x_{i_1}(-k_1')p\in U(s_{i_1}w)$, this likewise implies the 
existence of $(k_2',\dots ,k_r')\in \RR_{>0}^{r-1}$  such that 
$f_{(i_1,i_{j_2},\dots ,i_{j_r})}(k_1',k_2',\dots ,k_r')=p$.  Since $k_1\ne k_1'$ with both positive, the existence of both $(k_1,k_2,\dots ,k_r)$ and $(k_1',k_2',\dots ,k_r')$  
shows that $f_{(i_1,i_{j_2},\dots ,i_{j_r})}$ is not an injective map from $\RR_{>0}^r$ to $U(w)$.  But since  
$(i_1,i_{j_2},\dots ,i_{j_r})$ is a reduced word for $w$, Lusztig showed that 
$f_{(i_1,i_{j_2},\dots ,i_{j_r})}$ is a homeomorphism from $\RR_{>0}^r$ to $U(w)$, hence is injective.  
This gives a contradiction, completing our proof of the uniqueness of 
$k_1>0$ satisfying $x_{i_1}(-k_1)p\in U(s_{i_1}w)$.
\end{proof}

\begin{example}
Consider the word $(1,1,2)$ with $\delta (1,1,2) = s_1s_2$ and the point $p = x_1(3)x_1(5)x_2(7)\in U(s_1s_2)$.  Since the leftmost letter in $(1,1,2)$ is redundant, Lemma
~\ref{max-achieved} shows that 
the set of points $(t_1,t_2,t_3)\in f_{(1,1,2)}^{-1}(p)\cap \RR_{\ge 0}^3$ all must  satisfy $0\le t_1\le  t_1^{\max } = 3+5$.  Also note  that
$x_1(-8)p\in U( s_1s_1s_2) = U(s_2)$, as predicted by the second to last sentence in Lemma ~\ref{max-achieved}.  
\end{example}

\commenth{I just added  a hypothesis that
$cf(p;k_1,\dots ,i_{l-1}) \in U(\delta (i_l,\dots ,i_d))$ in the next corollary, since we do show that holds and since we need something like that for Lemma ~\ref{max-achieved} to 
be applicable.  I also strengthened the statement to say that this value is positive, which now follows.}

\begin{cor}\label{any-l}
Consider   any $w\in W$, any $p\in U(w)$ and any  word $(i_1,\dots ,i_d)$ with $\delta (i_1,\dots ,i_d) = w$.  Suppose  $i_l$ 
  is redundant in $(i_l,\dots ,i_d)$ for some $l\ge 1$.  If   $k_1,\dots ,k_{l-1}\in \RR_{\ge 0}^{l-1}$  satisfy 
 $cf(p;k_1,\dots ,k_{l-1})\in U(\delta (i_l,\dots ,i_d))$,  
then there exists a largest  nonnegative  real  value, denoted $t_l^{\max }(k_1,\dots ,k_{l-1})$,  
that $t_l$ 
takes within 
$$f_{(i_1,\dots ,i_d)}^{-1}(p)\cap \{ (k_1,\dots ,k_{l-1},t_l,t_{l+1},\dots ,t_d)| t_l,t_{l+1},\dots ,t_d\ge 0\} .$$  Moreover, $t_l^{\max }(k_1,\dots ,k_{l-1})>0$. 
\end{cor}

\begin{proof}
We replace  $p$ with $cf(p;k_1,\dots ,k_{l-1})$ and replace $(i_1,\dots ,i_d)$ with 
$(i_l,\dots ,i_d)$ so as to make $i_l$  the leftmost letter, then apply Lemma ~\ref{max-achieved}. 
\end{proof}

\commenth{I just added the next corollary and also added a citation of it.}

\begin{cor}\label{stay-in-strata}
Consider  any $w\in W$, any $p\in U(w)$, and any word $(i_1,\dots ,i_d)$  with $\delta (i_1,\dots ,i_d) = w$.  
 If   $i_1$ 
  is redundant in $(i_1,\dots ,i_d)$, then $x_{i_1}(-k_1)p\in U(w)$ for  $0\le  k_1<t_1^{\max }$.
\end{cor}

\begin{proof}
Note that $$x_{i_1}(-k_1)p = x_{i_1}(t_1^{\max }-k_1)x_{i_1}(-t_1^{\max }) p = x_{i_1}(t_1^{\max }-k_1) p'$$ for some $p'\in U(s_{i_1}w)$ by Lemma
~\ref{max-achieved}.  But
$t_1^{\max } - k_1 > 0$, implying $x_{i_1}(t_1^{\max }-k_1)p' \in  U(s_{i_1}s_{i_1}w) = 
U(w)$ since $l(w)>l(s_{i_1}w)$.  
\end{proof}

\begin{lemma}\label{can-tune-down-parameters}
Consider any $w\in W$, any $p\in U(w)$,  any   word $(i_1,\dots ,i_d)$ with $\delta (i_1,\dots ,i_d) =w $, and  
any $(t_1,\dots ,t_d) \in f_{(i_1,\dots ,i_d)}^{-1}(p)\cap \RR_{\ge 0}^d$. 
If   $i_l$ is redundant in $(i_l,\dots ,i_d)$ for some $l\ge 1$, then   
 there exists
 $t_{l+1}',\dots ,t_{d}'$ such  that $$(t_1,\dots ,t_{l-1},t_l',t_{l+1}',\dots , t_d') \in f_{(i_1,\dots ,i_d)}^{-1}(p)\cap \RR_{\ge 0}^d$$  for 
 each choice of $t_l'$ satisfying $0\le t_l' < t_l$.
\end{lemma}

\begin{proof}
 Let $A = x_{i_1}(t_1)\cdots x_{i_{l-1}}(t_{l-1})$ and let $B = x_{i_{l+1}}(t_{l+1})\cdots x_{i_{d}}(t_{d})$.  Then 
$$p = A x_{i_l}(t_l) B = A x_{i_l}(t_l') x_{i_l}(t_l-t_l') B$$ for any $0\le t_l' \le t_l$.  
By Remark ~\ref{Demazure-connected}, 
we may apply braid moves and modified nil-moves to $x_{i_{l+1}}\cdots x_{i_{d}}$ to produce a reduced expression
$x_{j_{l+1}}\cdots x_{j_{d'}}$.  There must then exist parameters $(t_{l+1}',\dots ,t_{d'}')$ with 
$x_{j_{l+1}}(t_{l+1}')\cdots x_{j_{d'}}(t_{d'}') = B$.   The redundancy of $i_l$ within $(i_l,\dots ,i_d)$ 
 together with reducedness 
of  $(j_{l+1},\dots ,j_{d'})$  implies that $x_{i_l}$ forms a deletion pair with some letter 
$x_{j_s}$ within  $x_{i_l}x_{j_{l+1}}\cdots x_{j_{d'}}$.  But then we may apply 
braid moves to $x_{j_{l+1}}\cdots x_{j_s}$ yielding $x_{j_{l+1}'}\cdots x_{j_s'}$ with 
$j_{l+1}' = i_l$ by virtue of Theorem ~\ref{braid-connectedness}.  This implies the
existence of parameters $u_{l+1},\dots ,u_s$ with 
$$x_{j_{l+1}'}(u_{l+1})\cdots x_{j_s'}(u_s) = x_{j_{l+1}}(t_{l+1}') \cdots x_{j_s}(t_s').$$
Thus, we have 
$$p  
= Ax_{i_l}(t_l') x_{j_{l+1}'}(t_l-t_l' + u_{l+1}) x_{j_{l+2}'}(u_{l+2})\cdots x_{j_s'}(u_s)x_{j_{s+1}}(t_{s+1}')\cdots x_{j_{d'}}(t_{d'}').$$ 
  In particular, this exhibits the existence of a suitable choice of values for the  parameters when $t_l$ is replaced by any $t_l'$ satisfying $0\le t_l'\le t_l$; this existence may be verified by using the fact that  the braid and modified-nil moves we used  may  all 
  be reversed with appropriate changes of coordinates in the parameters for each move.  
\end{proof}

\begin{example}
Since $(4,2)\in f_{(1,1)}^{-1}(p)$ for $p = x_1(6)$,  Lemma ~\ref{can-tune-down-parameters} guarantees the existence of 
$(t_1,t_2)\in f_{(1,1)}^{-1}(p)\cap \RR_{\ge 0}^2$ for each choice of $t_1\in [0,4]$.  For example, for $t_1=3$ we have 
$(3,1)\in f_{(1,1)}^{-1}(p)\cap \RR_{\ge 0}^2$.  In this example, the quantity $t_1^{\max }$  
 equals $6$.
\end{example}

Combining Lemma ~\ref{max-achieved}, Corollary  ~\ref{any-l}, 
and  Lemma ~\ref{can-tune-down-parameters} yields 
the following result.  

\begin{cor}\label{closed-interval}
Consider any $w\in W$, any $p\in U(w)$,  and any $(i_1,\dots ,i_d)$ with  $\delta (i_1,\dots ,i_d)=w$. 

Let  $i_1$ 
be redundant in $(i_1,\dots ,i_d)$, and consider the set  $V$  of values   $t_1$ takes within 
the set of solutions to $f_{(i_1,\dots ,i_d)}(t_1,\dots ,t_d) = p$  with   $(t_1,\dots ,t_d)\in \RR_{\ge 0}^d$.  
Then   $V= 
[0,t_1^{\max }]$ for  some 
  $t_1^{\max } \in \RR_{>0}$.   
  
  More generally, let $i_l$ be redundant in $(i_l,\dots ,i_d)$ for $l\ge 1$,
  let  $k_1,\dots ,k_{l-1}$ be  fixed  
  nonnegative constants chosen so that $cf(p;k_1,\dots ,k_{l-1})\in U(\delta (i_l,\dots ,i_d))$, 
   and  
  let  $V$ denote the set of values   $t_l$ takes within
  the set of nonnegative real  solutions to $$f_{(i_1,\dots ,i_d)}(k_1,\dots ,k_{l-1},t_l,\dots ,t_d) = p.$$ 
   Then $V = [0,t_l^{\max }(k_1,\dots ,k_{l-1})]$ for  some  
  $t_l^{\max }(k_1,\dots ,k_{l-1})\in \RR_{> 0}$.  
  %
 %
\end{cor}

This allows us to establish Definition ~\ref{t-j-max} (parts of which we already defined).

\begin{defn}\label{t-j-max}
Consider any $w\in W$, any point $p\in U(w)$ and any  word $(i_1,\dots ,i_d)$ with  $\delta (i_1,\dots ,i_d)=w$. 
  For each $j\in \{ 1,\dots ,d\} $ such that 
$i_j$ is redundant in $(i_j,\dots ,i_d)$ and each $k_1,\dots ,k_{j-1}\in \RR_{\ge 0}$ such that 
$cf(p;k_1,\dots ,k_{j-1})) \in U(\delta (i_j,\dots ,i_d))$, 
let  
$t_j^{\max }(k_1,\dots, k_{j-1})$ 
be  the largest value  $t_j$ 
takes  within  the set: 
$$ f_{(i_1,\dots ,i_d)}^{-1}(p) \cap  
 \{ (t_1,\dots ,t_d)\in \RR_{\ge 0}^d 
\mid t_i = k_i 
\hspace{.05in} {\rm for } \hspace{.05in} i< j\}. $$  
%

We also sometimes will speak of $t_j^{\max }$ 
without listing parameters $k_1,\dots ,  k_{j-1}$, 
in which case 
we are regarding  $t_j^{\max }$ 
as  a function from $ \RR_{\ge  0}^{j-1}$ 
to $\RR_{> 0}\cup \{ \emptyset \} $. 


It will also be convenient sometimes to  speak of $t_j^{\max }$ as a function of only  those parameters $t_{j_l}$  with $j_l<j$  
such that   $i_{j_l}$ is redundant in 
$(i_{j_l},\dots ,i_d)$, which is justified by the consequence of   Lemma ~\ref{last-letter-determined} 
that  the  values of these parameters will determine the values of all of the other parameters to the left of $t_j$.
\end{defn}

\begin{remark}
 It is unfortunate that the parameters $t_j$  such that the letter
$i_j$ is redundant in $(i_j,\dots ,i_d)$, namely the parameters that one might be tempted to call  ``redundant parameters'', are exactly the ones that we will use to 
parametrize the points in the fiber
$f_{(i_1,\dots ,i_d)}^{-1}(p)\cap \RR_{\ge 0}^d$ whereas the parameters given by the letters $i_j$ that are  non-redundant in $(i_j,\dots ,i_d)$
are the ones that will not be used in our parametrization, since these will be  functions of the redundant parameters.  
Definition ~\ref{t-j-max} 
is one of numerous places where one might expect the roles  of redundant and non-redundant parameters to be reversed based on nomenclature.  
\end{remark}

Next we give a useful    alternate  characterization  
for the set $S\subseteq \{ 1,2,\dots ,d\} $ of indices such that  $S^C := \{ 1,2,\dots ,d\} \setminus S$  
indexes the rightmost subword of $(i_1,\dots ,i_d)$ that is a reduced word for $\delta (i_1,\dots ,i_d)$.

\begin{lemma}\label{other-char}
Consider $w\in W$ and   $(i_1,\dots ,i_d)$ with 
$\delta (i_1,\dots ,i_d) = w$.  
Let  $S $ be the subset of $ \{ 1,\dots ,d \} $ whose complement 
$S^C = \{ j_1',\dots ,j_{d-s}' \} $ with $j_1'<\cdots < j_{d-s}'$  indexes the rightmost 
subword 
 of $(i_1,\dots ,i_d)$ 
that is a reduced word for $w$.   
Then $S$ 
has the following 
 characterization:
$ j \in S$ if and only if $i_j$ is redundant in  
 $(i_{j},i_{j+1},\dots ,i_d)$. 
\end{lemma}

\begin{proof}
By definition,
$i_j$ is redundant in $(i_j,\dots ,i_d)$  if and only if $\delta (i_j,\dots ,i_d) = \delta (i_{j+1},\dots ,i_d)$.  This holds if and only if multiplying $\delta (i_{j+1},\dots ,i_d)$ on the left by $s_{i_j}$ decreases Coxeter-theoretic length.   But this implies that if  we proceed from right to left creating a word comprised of  
exactly those letters indexed by $S^C$, each inserted letter increases the Coxeter-theoretic length.  
Moreover,  the end result once we reach the leftmost letter is a word with Demazure product $w$, implying  that at each step we must  
have a word  with Demazure product that is less than or equal to $w$ in Bruhat order.   
This constructively 
shows that the elements of $S^C$ index a  subword of $(i_1,\dots ,i_d)$ that is a reduced word for $w$.  By our greedy construction, the first place encountered from
right to left where any other subword  $Q'$  that is a reduced word for $w$ differs from our chosen word  indexed by $S^C$ must be a position included in $S^C$ and omitted from 
$Q'$.  This 
completes our proof of
the  claim that we 
have chosen the rightmost subword of $(i_1,\dots ,i_d)$  that is a reduced word for $w$.  
\end{proof}

\begin{example}
In type A we have  $\delta (1,1,2,4,2,1,2,1,5,2) = s_1s_2s_1s_4s_5$.  The rightmost reduced word for $s_1s_2s_1s_4s_5$ appearing as a subword of 
$(1,1,2,4,2,1,2,1,5,2)$ is $(-,-,-,4,-,-,2,1,5,2)$.  Thus,  $S^C = \{ 4,7,8,9,10\} \subseteq \{ 1,2,\dots ,10\} $ in this case.  Indeed the  subword 
$(2,1,5,2)$ comprising the  four rightmost positions is reduced whereas multiplying  $\delta (2,1,5,2)$ on the left by $s_1$ decreases Coxeter-theoretic length which means
that the leftmost letter in $(1,2,1,5,2)$ is redundant.
\end{example}

Throughout the remainder of this section, when considering a fiber $f^{-1}_{(i_1,\dots ,i_d)}(p)\cap \RR_{\ge 0}^{d}$ we will consistently use   
$S = \{ j_1,\dots ,j_s\}  \subseteq \{ 1,2,\dots ,d\}$ 
to  denote the set of indices for those letters
$i_j$ such that $i_j$ is non-redundant in the word $(i_j,\dots ,i_d)$,  and we consistently denote by 
 $S^C = \{ j_1',\dots ,j_{d-s}' \} \subseteq \{ 1,2,\dots ,d\} $  the complementary set of indices, the latter of which we have 
just shown is exactly  the indices such that $(i_{j_1'},\dots ,i_{j_{d-s}'})$ is the rightmost subword of $(i_1,\dots ,i_d)$ that is a reduced word for $\delta (i_1,\dots ,i_d)$.   
More generally,  we will  use the notation $S= \{ j_1,\dots ,j_s\} $ and $S^C  = \{ j_1',\dots ,j_{|Q|-s}' \} $ 
when $(i_1,\dots ,i_d)$ is replaced by a subword $Q$ with $\delta (Q) = \delta (i_1,\dots ,i_d)$, 
in which case $S^C$ will index the rightmost subword of $Q$ that is a reduced word
for $\delta (Q)$ and  $S$ will  equal $supp(Q) \setminus S^C$.  

\begin{lemma}\label{redundant-non-max-which-delta}
Consider any $w\in W$, any  point  $p\in U(w)$, and any word $(i_1,\dots ,i_d)$ satisfying $\delta (i_1,\dots ,i_d)=w$.  
 If  $i_1$ is redundant in $(i_1,\dots ,i_d)$ and 
$(k_1,\dots ,k_d)\in f_{(i_1,\dots ,i_d)}^{-1}(p)\cap \RR_{\ge 0}^d$ such that $k_1 < t_1^{\max }$, then $cf(p;k_1)\in U(\delta (i_2,\dots ,i_d))$.   
If $i_1$ is non-redundant in $(i_1,\dots ,i_d)$ and $k_1$ is the unique nonnegative real number such that there exists
$(k_1,k_2,\dots ,k_d)\in f^{-1}_{(i_1,\dots ,i_d)}(p)\cap \RR_{\ge 0}^d $, then $cf(p;k_1)\in U(\delta (i_2,\dots ,i_d))$.  

More generally, 
if $(k_1,\dots ,k_d)\in f_{(i_1,\dots ,i_d)}^{-1}(p)\cap \RR_{\ge 0}^d$ 
satisfies $k_l < t_l^{\max }(k_1,\dots ,k_{l-1})$ for each $l\in S$, then 
$cf(p;k_1,\dots ,k_j) \in U(\delta (i_{j+1},\dots ,i_d))$ for $1\le j\le d-1$.
\end{lemma}

\begin{proof}
We showed in Lemma ~\ref{max-achieved}  that $cf(p;k_1)\not \in U(s_{i_1}w)$
 in the case of 
$i_1$ redundant  in $(i_1,\dots ,i_d)$ and $k_1 < t_1^{\max }$. 
By Corollary ~\ref{stay-in-strata}, this implies  $cf(p;k_1)\in U(w)$ in this case.
But  $i_1$  redundant in $(i_1,\dots ,i_d)$  implies 
$\delta (i_2,\dots ,i_d) = \delta (i_1,i_2,\dots ,i_d)=w$.
When  $i_1$ is non-redundant in $(i_1,\dots ,i_d)$, then 
$cf(p;k_1)\in U(s_{i_1}w)$.  But  $s_{i_1}w = \delta (i_2,\dots ,i_d)$ in this case.  This shows 
$cf(p;k_1) \in U(\delta (i_2,\dots ,i_d))$ in either case.

Combining these two cases  allows us to deduce the more general  
final claim by induction on $j$ by  proceeding from  left to right through the first $j$ letters, 
repeatedly applying whichever of the two cases is pertinent for the letter under consideration; which case applies when $i_l$ is the leftmost letter 
is determined by 
whether the letter  $i_l$ under consideration  is redundant or non-redundant within
the word  $(i_l,\dots ,i_d)$. 
\end{proof}

\begin{example}
Consider the word $(i_1,\dots ,i_d) = (1,1,1)$ in type A  with $\delta (1,1,1) = s_1$ and consider $p = x_1(7) \in U(s_1)$.  Notice that the middle letter $i_2$ is redundant in this case.  We have 
$t_1^{\max } = 7$.   If   $k_1 = 4$,  then $t_2^{\max }(4) = 3$ which indeed is strictly  greater than 0, as desired.  In this case
$cf(p; 4) = x_1(-4)p = x_1(3)\in U(s_1)= U(\delta (1,1))$.  But then 
$x_1(-t_2^{\max }(4)) cf(p; 4) = x_1(-3)x_1(3) = x_1(0) \in U(e) = U(s_1 \delta (1,1))\ne U(\delta (1))$.  
\end{example}

\begin{lemma}\label{any-l-positive}
Consider any $w\in W$, any $p\in U(w)$, and  any  word $(i_1,\dots ,i_d)$ with $w=\delta (i_1,\dots ,i_d)$. 
 Suppose  $i_l$ is redundant in $(i_l,\dots ,i_d)$,  which implies   $l\in S$.
Consider $(k_1,\dots ,k_{l-1})\in \RR_{\ge 0}^{l-1}$ such that 
$$f_{(i_1,\dots ,i_d)}(k_1,\dots ,k_{l-1},t_l,\dots ,t_d) = p$$ has a solution with $t_l,\dots ,t_d\ge 0$.  
If $k_j < t_j^{\max }(k_1,\dots ,k_{j-1})$ for each $j<l$ such that  $j\in S$, then
$t_l^{\max }(k_1,\dots ,k_{l-1})>0$. 
In this case,  
$t_l^{\max }(k_1,\dots ,k_{l-1})$  is 
the  unique positive   real 
number   $k_l$  such that 
$x_{i_l}(-k_l) cf(p;k_1,\dots ,k_{l-1}) 
 \in U(s_{i_l}\delta(i_l,\dots ,i_d))$. 
\end{lemma}

\begin{proof}
Suppose $l>1$.  If the leftmost parameter is non-redundant,  
Corollary ~\ref{unique-value-other-cell} allows us to replace $p$ by $cf(p;k_1)$, replace $(i_1,\dots ,i_d)$ by $(i_2,\dots ,i_d)$ 
and eliminate the leftmost parameter.  
If the leftmost parameter is redundant, we may again 
replace $p$ by $cf(p;k_1)$ and again replace $(i_1,\dots ,i_d)$ with $(i_2,\dots ,i_d)$ and eliminate the leftmost parameter.
We repeatedly apply Lemma  ~\ref{redundant-non-max-which-delta}
as we proceed  from left to right eliminating 
parameters in this manner 
to get that $cf(p;k_1,\dots ,k_{l-1})\in U(\delta (i_l,i_{l+1},\dots ,i_d))$.  But then we have reduced both 
the positivity claim and the claim that $cf(p;k_1,\dots ,k_l)\in U(s_{i_l}\delta (i_l,\dots ,i_d)) $  to case in which $l=1$, namely the case 
of the  leftmost position.
But this is handled in Lemma ~\ref{max-achieved}.  
\end{proof}

\begin{lemma}\label{existence-by-induction} 
Consider any $w\in W$, any $p\in U(w)$, and any   $(i_1,\dots ,i_d)$ with $\delta (i_1,\dots ,i_d)=w$. 
Given $j\in S^C$, implying $i_j$ is non-redundant in $(i_j,\dots ,i_d)$,
and given any $k_1,\dots ,k_{j-1}\ge 0$ chosen so that $cf(p;k_1,\dots ,k_{j-1} ) \in U(\delta (i_{j},\dots ,i_d))$ for each $l\ge j-1$,
then  there exists a unique $k_j\ge 0$ such that $cf(p;k_1,\dots ,k_j)\in U(\delta (i_{j+1},\dots ,i_d))$.  
This value $k_j$ is positive with 
$$f_{(i_1,\dots ,i_d)}^{-1}(p)  
\cap \{ (k_1,\dots ,k_j,t_{j+1},\dots ,t_d) |  t_{j+1},\dots , t_d\ge 0\} \ne \emptyset  .$$ 
\end{lemma}

\begin{proof}
This is a direct consequence of Lemma ~\ref{last-letter-determined}, by virtue of the fact justified in  Lemma  ~\ref{full-theorem} 
that  
$$\{(t_{j+1},\dots ,t_d)\in \RR_{\ge 0}^{d-j} |( k_1,\dots ,k_j,t_{j+1},\dots ,t_d)\in f_{(i_1,\dots ,i_d)}^{-1}(p)\} = 
\RR_{\ge 0}^{d-j} \cap f_{(i_{j+1},\dots ,i_d)}^{-1}(p')$$ for $p' = cf(p;k_1,\dots ,k_j).$
\end{proof}

The upcoming Lemma ~\ref{generalized-Lusztig-homeom} together with Lemmas ~\ref{defined-forced-value}  and  ~\ref{0-iff-max-to-left}   will  soon allow us to define 
a map  $f_{F_Q}$  for each subword $Q$ of $(i_1,\dots ,i_d)$ having $\delta (Q) = \delta (i_1,\dots ,i_d)$, doing so in a way 
that will restrict to our  desired homeomorphism from an open ball to the stratum  $F_Q$.
These will  allow us  
to use the parameters indexed by the subset $S$ of $supp(Q)$ whose complement  within $supp(Q)$ 
indexes the rightmost reduced word for $w$ within  $Q$ 
to parametrize the points contained  in a collection of strata within  $f_{(i_1,\dots ,i_d)}^{-1}(p)\cap \RR_{\ge 0}^d$, with $F_Q$ being  one of these strata. 
To this end, we  will 
use the fact  (shown  in Lemma 
~\ref{defined-forced-value}) 
that  the parameters indexed by  $supp(Q)\setminus S$  
 are  uniquely determined
by  the  
point $p\in U(w)$ and the values of the parameters indexed by $S$.  

\begin{lemma}\label{generalized-Lusztig-homeom}
Consider any $w\in W$, and any word   $(i_1,\dots , i_d)$ with  
$\delta (i_1,\dots ,i_d)=w$.  
Consider 
$$D_{k_{j_1},\dots ,k_{j_s}}  = \{ (t_1,\dots ,t_d) \in \RR^{d}_{\ge  0} \mid  t_{j_l'}  > 0 \hspace{.03in} {\rm for}\hspace{.03in} l=1,\dots ,d-s
\hspace{.075in}{\rm and}\hspace{.075in}  
 t_{j_r} = k_{j_r}  \hspace{.03in}{\rm   for}\hspace{.03in} 1\le r\le s \}  $$
for  
any  fixed choice of constants  $k_{j_1},\dots ,  k_{j_s}\ge 0$. 
Then   $f_{(i_1,\dots ,i_d)}|_{D_{k_{j_1},\dots ,k_{j_s}}}$ is 
 a bijection  $h_{k_{j_1},\dots ,k_{j_s}}: D_{k_{j_1},\dots ,k_{j_s}} \rightarrow im(h_{k_{j_1},\dots ,k_{j_s}})\subseteq U(w)$.
\end{lemma}

\begin{proof}
  Consider any  $p\in im (h_{k_{j_1},\dots ,k_{j_s}})$.  Let us now  prove  $p = h_{k_{j_1},\dots ,k_{j_s}}(x) $ for a  unique choice of $x\in D_{k_{j_1},\dots ,k_{j_s}}$.  By definition of 
  $D_{k_{j_1},\dots ,k_{j_s}}$, this is equivalent to proving uniqueness of 
  $x|_{S^C}  = (t_1',\dots ,t_{d-s}')  \in \RR_{> 0}^{d-s}$  with $x$
  satisfying $h_{k_{j_1},\dots ,k_{j_s}}(x) = p$ for fixed $p\in im(h_{k_{j_1},\dots ,k_{j_s}})$. 
   To this end, we proceed from left to right through all of the parameters $t_1,\dots ,t_d$ appearing in $x$ 
   showing that the value of each such  parameter  $t_i$  in turn  is uniquely determined by the point $p$ and the values of the  parameters 
 $ t_1, \dots ,t_{i-1}$ to its left, doing this for every $t_i$ regardless of whether $i\in S$ or $i\in S^C$. 

First let us show that $t_1$ is uniquely determined, noting that there does indeed exist at least one valid choice for $t_1$ by virtue of $p$ being in $im(h_{k_{j_1},\dots ,k_{j_s}})$.  
  If $j_1=1$, then our restriction to domain $D_{k_{j_1},\dots ,k_{j_s}}$ necessitates  $t_{j_1} = k_{j_1}$.
  On the other hand, if $j_1\ne 1$, then $1\in S^C$.  But then by Lemma ~\ref{other-char}, this  implies that  $i_1$ is non-redundant in 
$(i_1,\dots ,i_d)$.  Lemma ~\ref{last-letter-determined} then guarantees there is a unique value $k_1$ for $t_1$ such that there 
exists a point $(k_1,t_2,\dots ,t_d) \in \RR_{\ge 0}^d$ with $f_{(i_1,\dots ,i_d)}(k_1,t_2,\dots ,t_d) = p$.  
Our choice of $p\in im (h_{k_{j_1},\dots ,k_{j_s}})$ then ensures
we may choose this  $t_2,\dots ,t_d)$ in such a way have $(k_1,t_2,\dots ,t_d)  \in D_{k_{j_1},\dots ,k_{j_s}}$. 

Next we show how to proceed from left to right through the parameters $t_2,t_3,\dots ,t_d$,  proving each  such $t_r$  takes a unique value $k_r$  given that the values for $t_1,\dots ,t_{r-1}$ have already been shown to be uniquely determined with values we call  $k_1,\dots ,k_{r-1}$.  
To this end, observe 
 that 
the equation 
$$x_{i_1}(k_1)\cdots x_{i_{r-1}}(k_{r-1}) x_{i_r}(t_r)\cdots x_{i_d}(t_d) = p$$ has exactly the same nonnegative real solutions for $t_r,\dots ,t_d$ as the equation
$$x_{i_r}(t_r)\cdots x_{i_d}(t_d) = x_{i_{r-1}}(-k_{r-1})\cdots x_{i_1}(-k_1)p.$$  
In particular, once $t_1$ is set to  its uniquely determined  value $k_1$, this 
reduces the task of uniquely determining $t_2$ to the case we just handled of uniquely determining  the value of the leftmost parameter.  
Applying this reduction repeatedly, we  can likewise make each parameter $t_r$  
in turn   leftmost to deduce uniqueness of the value of that  parameter from our argument for the leftmost parameter.
In this manner,  we deduce injectivity of $h_{k_{j_1},\dots ,k_{j_s}}$.  
Surjectivity of $h_{k_{j_1},\dots ,k_{j_s}}$  onto its image   
follows tautologically.
This completes the proof of 
 bijectivity of $h_{k_{j_1},\dots ,k_{j_s}}$ as a  map   from $D_{k_{j_1},\dots ,k_{j_s}}$  to $im(h_{k_{j_1},\dots ,k_{j_s}})$. 
\end{proof}

The proof of Lemma ~\ref{generalized-Lusztig-homeom} yields the following  useful fact:

\begin{cor}\label{early-f-j}
The nonnegative real value of each parameter $t_j$ for $j\in S^C$ is uniquely  determined
by $p\in U(w)$ together with  the 
values of those parameters $t_{j_r}$ satisfying both  $j_r\in S$ and  $j_r<j$, provided that 
these parameters $t_{j_r}$ all are assigned  nonnegative  
values and 
that a nonnegative real value for $t_j$ exists for this  choice of values for the parameters  $t_{j_r}$ with $j_r\in S$ and $j_r<j$.  
 \end{cor}
 
 \begin{example}
 Consider $(i_1,\dots ,i_d) = (1,1,2,2)$ and $p = x_1(17)x_2(5)$.  Since the righmost reduced word for $s_1s_2$ within $(1,1,2,2)$ uses positions 
 2 and 4, we have $S^C = \{ 2,4\} $ and   $S = \{ 1,3 \} $. 
 Setting $k_1=7$ and $k_3=1$ for  
 $(k_1,k_2,k_3,k_4)\in f_{(1,1,2,2)}^{-1}(p)\cap \RR_{\ge 0}^4$  for $p=x_1(17)x_2(5)$ forces  $k_2 = 10$ and $k_4 = 4$. 
 \end{example}

In  Definition ~\ref{f-j-def},  
we  introduce 
 a function $f_j$ for each $j \in \{ 1,\dots ,d\} $ such that $i_j$ is non-redundant in $(i_j,\dots ,i_d)$, formalizing 
the idea from Corollary ~\ref{early-f-j}.  That is, for $j\in S^C$ the function   $f_j$  
will  express the  value for $t_j$ for $j\in S^C$ as a function of 
  the values of the  parameters  to the  left of $t_j$ which have indices in $S$.  
We define  $f_j$ on a sufficiently restrictive domain so 
that  a suitable value indeed provably will   exist  for $t_j$, namely on the domain $D^{<\max }_{S,r-1}$ introduced next, letting
 $r-1 = |S\cap \{ 1,2,\dots ,j-1\} |$. 

\begin{defn}\label{useful-domain-def}
Consider any word  $(i_1,\dots ,i_d) $ with $\delta (i_1,\dots ,i_d) = w$.  
Let  $S = \{ j_1,\dots ,j_{s}\}\subseteq \{ 1,\dots ,d\} $ with $j_1 < \cdots < j_s$  be the set whose complement $S^C = \{ 1,\dots ,d\} \setminus S$ 
indexes the rightmost reduced word for $w$ appearing as  a subword of  $(i_1,\dots ,i_d)$.  Let   
$$D^{<\max }_S = \{ (c_{j_1},\dots ,c_{j_s})\in \RR_{\ge 0}^{d-l(w)} \mid  c_{j_r} < t_{j_r}^{\max }(k_1,\dots ,k_{j_r-1}) \hspace{.04in}{\rm for}\hspace{.04in} 1\le r\le s \} $$
where for  each $r\in \{ 1,\dots ,s\} $ we calculate $k_1,k_2,\dots ,k_{j_r-1}$   from left to right 
as follows:  
\begin{enumerate}
\item
 For each $i\in S$ with $i<j_r$, let   $k_i=c_i$,
 \item
 For each $i\not\in S$ with $i<j_r$,  we let  $k_i$   be the unique nonnegative real value  such that 
the set of  points $(k_1,\dots ,k_i,t_{i+1},\dots ,t_d)\in f_{(i_1,\dots ,i_d)}^{-1}(p)\cap \RR_{\ge 0}^d$  
 is nonempty.  
\end{enumerate}
Likewise define for $1\le i\le s$ the domain 
$$D_{S,i}^{<\max } = \{ c_{j_1},\dots ,c_{j_i})\in \RR_{\ge 0}^i \mid  c_{j_r}<t_{j_r}^{\max }(k_1,\dots ,k_{j_r-1})\hspace{.04in}{\rm for}\hspace{.04in} 0\le r\le i\},$$
so in other words $D_{S,i}^{<\max }$ is the projection of $D_S^{<\max }$ onto its first $i$ coordinates.  

Define $D^{<\max }_{S,0} = \{ \emptyset \} $.  
\end{defn}

\begin{remark}
The domain  $D^{< \max }_S$ is 
nonempty by virtue of  Lemma ~\ref{any-l-positive}, since that  showed positivity of  
 $t_{j_r}^{\max }(k_1,\dots ,k_{j_r-1})$ for each  $r>1$  provided that 
$k_{j_l} < t_{j_l}^{\max }(k_1,\dots ,k_{j_l-1})$ for $l=1,2,\dots ,r-1$, and  it also showed positivity of $t_{j_1}^{\max }$.
\end{remark}

\begin{remark}
The non-maximality constraints in the definition of $D_S^{<\max }$ also  
give  the very useful property, made precise in Lemma ~\ref{redundant-non-max-which-delta}, 
 that changing the values of the  parameters 
 indexed by $S$ in a way that stays within the  domain $D^{<\max }_S$  will not change certain key   combinatorial structure.
\end{remark}

\begin{defn}\label{f-j-def}
Consider any  $w\in W$, any point  $p\in U(w)$, and any  word $(i_1,\dots ,i_d)$ with $\delta (i_1,\dots ,i_d)=w$.  As usual, let 
$S = \{ j_1,\dots ,j_s \} \subseteq \{ 1,2,\dots ,d\} $  be chosen so that 
$S^C$ indexes the rightmost subword of $(i_1,\dots ,i_d)$ that is a reduced word for $w$.  
Consider any   $j\in S^C$.   Let  $r-1 = |S\cap \{ 1,2,\dots ,j-1\} |$.  
 
 Define  
 $f_{j}: D^{<\max }_{S,r-1} \rightarrow \RR_{>0}$ for $r-1>0$ 
 to be   the map which sends $(k_{j_1},\dots ,k_{j_{r-1}})\in D^{<\max }_{S,r-1}$ to the unique positive real value
the parameter $t_j$ takes 
within $f_{(i_1,\dots ,i_d)}^{-1}(p) \cap \RR_{\ge 0}^d$ 
subject to the  requirements  that $t_{j_l} = k_{j_l} $ for $l=1,2,\dots ,r-1$. 

When $r-1=0$, then the input  to $f_j$ is $\emptyset $ and we let $f_j(\emptyset )$ be 
the unique positive real value the parameter $t_j$ takes within $f_{(i_1,\dots ,i_d)}^{-1}(p)\cap \RR_{\ge 0}^d$.  In other words,
$f_j(\emptyset )$ is  the unique value   the leftmost 
parameter takes in $f_{(i_j,\dots ,i_d)}^{-1}(cf(p;k_1,\dots ,k_{j-1}))\cap \RR_{\ge 0}^{d-j+1}$ where $k_1,\dots ,k_{j-1}$ are the unique values taken by the parameters
$t_1,\dots ,t_{j-1}$ for all points in  $f_{(i_1,\dots ,i_d)}^{-1}(p)\cap \RR_{\ge 0}^d$. 
\end{defn}

\begin{example}\label{specific-instance}
Let $(i_1,\dots ,i_d) = (1,2,1,2)$ with  $\delta (1,2,1,2) = s_1s_2s_1$ and  $S = \{ 1\}  \subseteq \{ 1,2,3,4\} $.
Consider a point 
$(t_1,t_2,t_3,t_4) \in f_{(1,2,1,2)}^{-1}(p)\cap \RR_{\ge 0}^4$ for $p = x_1(9)s_2(3)x_1(7)$.  If $t_1 = 4$,  then 
$(t_2,t_3,t_4)\in f_{(2,1,2)}^{-1}(x_1(-4)p)\cap \RR_{\ge 0}^3$ which 
uniquely 
determines the other values for the parameters  as $t_2 = \frac{3\cdot 7}{5+7}, t_3 = 5+7$ and 
$t_4 = \frac{5\cdot 3}{5+7}$.  In other words,  $f_2(4) = \frac{21}{12}, f_3(4) = 12$ and $f_4(4) = \frac{15}{12}$.  Similar
reasoning shows $f_2(t_1) = \frac{3\cdot 7}{16-t_1}, f_3(t_1) = 16-t_1$ and $f_4(t_1) = \frac{3(9-t_1)}{16-t_1}$.
\end{example}

\begin{lemma}\label{defined-forced-value} 
Consider any  $w\in W$, any point  $p\in U(w)$,  any word  $(i_1,\dots ,i_d)$ with $\delta (i_1,\dots ,i_d)=w$, 
and  any $j\in S^C$.    Let $r-1 = |S\cap \{ 1,2,\dots ,j-1\} |$.
The function $f_{j}$  
is a  well-defined function from $D^{< \max }_{S,r-1}$  to the positive reals.
\end{lemma}

\begin{proof}
The input to $f_j$ is a  list of values $k_{j_1},k_{j_2},\dots ,k_{j_{r-1}}$  for the parameters 
indexed by those  elements of  $S$ with index less than $j$, provided  that we have $r-1>0$.    We may therefore  use Corollary ~\ref{early-f-j} to get uniqueness of the value for $t_j$ when $r-1>0$.   We specified the unique value for 
$f_{j-1}(\emptyset )$ in the case that $r-1=0$ within Definition ~\ref{f-j-def}.  
What remains is to prove existence and  
positivity of the value $t_j$ for any  $j\in S^C$ such that $r-1>0$, or equivalently to show and positive of the value  for the leftmost parameter in 
$f_{(i_j,\dots ,i_d)}^{-1}(p_{j-1})$ for 
 $p_{j-1} = cf(p;k_1,\dots ,k_{j-1})$ for any such  $j\in S^C$.  
 
 Our choice of domain $D^{<\max }_{S,r-1}$  ensures that 
$p_{j-1} \in U(\delta (i_j,\dots ,i_d))$ by Lemma ~\ref{redundant-non-max-which-delta}. 
%
%
Thus, if we replace $(i_j,\dots ,i_d)$ by any reduced subword $Q'$ of $(i_j,\dots ,i_d)$ that contains $i_j$, then
existence, uniqueness  and  positivity  of  a value $k_j'$ for the leftmost parameter for  points in $f_{Q'}^{-1}(p_{j-1})\cap \RR_{\ge 0}^{|Q'|}$ 
follows from 
Lemma \ref{last-letter-determined} applied to $f_{Q'}^{-1}(p_{j-1})\cap \RR_{\ge 0}^{|Q'|}$.  
But since $\delta (i_j,\dots ,i_d)
= \delta (Q')$  and since $i_j$ is non-redundant in $(i_j,\dots ,i_d)$ (by definition of  $S^C$), 
this same positive  value 
$k_j'$  is the unique value for the leftmost parameter
in $f_{(i_j,\dots ,i_d)}^{-1}(p_{j-1})\cap \RR_{\ge 0}^d$, completing our proof of existence and positivity. 
\end{proof}

The next  lemma 
 pulls together  many of   our results so far  to give a useful 
description of an important  
subset of  each fiber.  We will eventually  show  that this subset of the fiber, denoted $D^{<\max }$ below, 
is homeomorphic to $[0,1)^{d-l(w)}$ and is 
a union of strata $F_Q$ for various subwords of $(i_1,\dots ,i_d)$ each satisfying $\delta (Q) = \delta (i_1,\dots ,i_d)$.

\begin{lemma} \label{0-iff-max-to-left}
Consider any $w\in W$, any  word $(i_1,\dots ,i_d)$ with  $\delta (i_1,\dots ,i_d) = w$, and  any  $p\in U(w)$.
Then  the  set 
$f^{-1}_{(i_1,\dots ,i_d)}(p) \cap \RR_{\ge 0}^d$ contains the set $D^{<\max }$ of points  $(k_1,...,k_d)\in \RR_{\ge 0}^d$  
calculated from left to right such that (1)   $k_j \in [0,t_j^{\max}(k_1,\dots ,k_{j-1}))$  
  for each  $j\in S$,  and  (2) each $k_j$ for $j \in S^C$ 
 is uniquely determined by $k_1,k_2,\dots ,k_{j-1}$ as the unique nonnegative real such that 
 $f_{(i_1,\dots ,i_d)}^{-1}(p)$ includes at least one nonnegative real point $(t_1,\dots ,t_d)$ satisfying $t_i=k_i$ for $k=1,2,\dots ,j$. 

All points in    $D^{<\max }$ satisfy    $t_j^{\max }(k_1,\dots ,k_{j-1})>0$ for each $j\in S$.  
Moreover,  the stratum $F_{(i_1,\dots ,i_d)}$ of $f_{(i_1,\dots ,i_d)}^{-1}(p)\cap \RR_{\ge 0}^d$  in which all parameters are positive is contained in $D^{<\max }$.
%
\end{lemma}

\begin{proof}
We proceed  through the  parameters from   left to right, repeatedly updating which parameter is our leftmost  undetermined parameter  as follows.   Suppose  
we have already chosen values $k_1,\dots ,k_{j-1}$ for $t_1,\dots ,t_{j-1}$ such that $cf(p;k_1,\dots, k_{j-1}) \in U(\delta (i_j,\dots ,i_d))$.
If 
 the current 
leftmost undetermined  parameter $t_j$  has associated letter $i_j$ that is non-redundant in $(i_j,\dots ,i_d)$, then 
we may apply Lemmas ~\ref{defined-forced-value} and
~\ref{existence-by-induction}  to deduce
the uniqueness, existence and positivity  of a value $k_j$ for $t_j$ such that there exists  an element of $f_{(i_1,\dots ,i_d)}^{-1}(p)\cap \RR_{\ge 0}^d$ having $t_i=k_i$ for 
$i=1,2,\dots ,j$.  But this  value $k_j$ is also the unique nonnegative real  value that the leftmost 
parameter takes  for each point in $f_{(i_j,\dots ,i_d)}^{-1}(cf(p;k_1,\dots ,k_{j-1})) \cap \RR_{\ge 0}^d$, as shown by multiplying both sides of the equation
$f_{(i_1,\dots ,i_d)}(k_1,\dots ,k_j,t_{j+1},\dots ,t_d) = p$ on the left by $x_{i_j}(-k_j)\cdots x_{i_1}(k_1)$ and then applying Lemma ~\ref{last-letter-determined}; 
this reinterpretation for $k_j$ allows us to 
apply Lemma ~\ref{full-theorem}   to deduce that $cf(p;k_1,\dots ,k_{j-1},k_j)\in U(\delta (i_{j+1},\dots ,i_d))$, validating our inductive assumption in this case.  
%
%
 Thus, if  $i_j$  is non-redundant in $(i_j,\dots ,i_d)$ we may proceed rightward  to consider leftmost undetermined parameter $t_{j+1}$. 

Whenever  we have a leftmost undetermined  parameter $t_j$  such  that $i_j$ is redundant in $(i_j,\dots ,i_d)$, 
we  invoke Corollary ~\ref{closed-interval} to deduce that there is  a closed interval  $[0, t_j^{\max }(k_1,\dots ,k_{j-1})]$ of possible values for that parameter $t_j$.  
Lemma ~\ref{any-l-positive}  guarantees $t_j^{\max }(k_1,\dots ,k_{j-1})>0$. 
This positivity implies  the existence of non-maximal values within $[0,t_j^{\max }(k_1,\dots ,k_{j-1})]$.  
When  we choose such   a non-maximal value  $k_j$ for $t_j$, 
as is required for our domain $D^{<\max }$, we may 
apply the map $cf$ again to eliminate this parameter $t_j$  and turn our attention to  the new leftmost undetermined  parameter $t_{j+1}$;
lemma  ~\ref{redundant-non-max-which-delta}  tells us  that  $cf(p; k_1,\dots ,k_j) \in U(\delta (i_{j+1},\dots ,i_d))$ in this case, validating the reduction. 
   In this manner, we proceed through all of the 
parameters from left to right by using Lemmas ~\ref{last-letter-determined} and ~\ref{full-theorem} for non-redundant parameters and Lemmas ~\ref{closed-interval} and 
~\ref{any-l-positive} for redundant parameters.  
Lemma ~\ref{other-char} shows that  the parameters $t_j$ given by  letters  $i_j$ which are redundant in 
$(i_j,\dots ,i_d)$  are exactly those parameters with indices  
in $S$, justifying our usage of  $S$ and $S^C$ in the statement of this lemma.

 The final claim is proven as follows.  Suppose some $(k_1,\dots ,k_d)\in f^{-1}_{(i_1,\dots ,i_d)}(p)\cap \RR_{\ge 0}^d$ is not in $D^{<\max }$, which means
it  has  some $j\in S$ with  $k_j = t_j^{\max }(k_1,\dots ,k_{j-1})$.  Then 
 $$cf(p;k_1,\dots ,k_j)\in U(s_{i_j}\delta (i_{j+1},\dots ,i_d)) \ne U(\delta (i_{j+1},\dots ,i_d)).$$
 But in order to have $(k_{j+1},\dots ,k_d)\in f^{-1}_{(i_{j+1},\dots ,i_d)}(cf(p;k_1,\dots ,k_j))$ we would then need to have   
 $k_l=0$ for  some  $l>j$, implying $(k_1,\dots ,k_d)\not\in F_{(i_1,\dots ,i_d)}$.
%
\end{proof}

\begin{remark}
Continuity of $f_j$, proven later in the paper in  Lemma 
~\ref{cont-forced-value}, may be combined with continuity of projection maps to show   $D^{<\max } \cong D^{<\max }_S$.
\end{remark}

Next we  define for each stratum $F_Q$ (see Definition ~\ref{F-Q-def})  of $f^{-1}_{(i_1,\dots ,i_d)}(p)\cap \RR_{\ge 0}^d$ a
map  
  $$f_{F_Q}: [0,1)^s \rightarrow f_{(i_1,\dots ,i_d)}^{-1}(p)\cap \RR_{\ge 0}^d .$$ 
Although  this will not be a surjection to all of $f_{(i_1,\dots ,i_d)}^{-1}(p)\cap \RR_{\ge 0}^d$, we will eventually prove  
that this map   $f_{F_Q}$ 
restricts to a homeomorphism  from $(0,1)^s$ to  $F_Q$,  
allowing us to deduce that each stratum $F_Q$ is a cell, i.e.,  is  homeomorphic to an open ball.

\begin{defn}\label{f-F-def}
Given any  $w\in W$, any point  $p\in U(w)$, any  word $(i_1,\dots ,i_d)$ with $w = \delta (i_1,\dots ,i_d)$, and given  any  
stratum  $F_Q \subseteq  f_{(i_1,\dots ,i_d)}^{-1}(p)\cap \RR_{\ge 0}^{d}$ satisfying  $\delta (Q) = w$,    
we define a  map 
$f_{F_Q}$ as described below.  See 
Lemma ~\ref{thm:CW-map} for  well-definedness of $f_{F_Q}$.

%
For $Q\subset \{ 1,2,\dots ,d\} $ with $\delta (Q) = \delta (i_1,\dots ,i_d)$, 
 let $|Q|$ be the size of $supp(Q)$ and let $s = |Q|-l(w)$.   
Consider  $S = \{ j_1,\dots ,j_s\} \subseteq supp(Q)$ with $1\le j_1 < \cdots < j_s \le d$ 
chosen such that the complementary set  $S^C := supp (Q) \setminus S$ is exactly
the support of 
the rightmost subword of $Q$ that is a reduced word for $w$. 

Define  $$f_{F_Q}: [0,1)^{s} \rightarrow f_{(i_1,\dots i_d)}^{-1}(p)\cap \RR_{\ge 0}^d$$ as follows.   
Given any $u_1,\dots ,u_s \in [0,1)$, let 
$f_{F_Q} (u_1,\dots ,u_{s})  =  (k_1,\dots ,k_d) $ where the values $k_1,\dots ,k_d$ for the parameters $t_1,\dots ,t_d$ are 
determined 
as follows.     For $i\not\in supp(Q)$, set $t_i=0$.  Proceeding  left to right through the remaining parameters, if the 
parameter $t_j$  currently under consideration 
has index  $j\in S^C$, then 
the value $k_j$ of the parameter $t_j$  is uniquely determined. 
If, on the other hand, the parameter currently under consideration has index $j_r\in S$, 
then 
we use  
$u_r \in [0,1)$  together with  the values $k_1,\dots ,k_{j_r-1}$  already assigned to 
$t_1,\dots ,t_{j_r-1}$  to determine the value  $k_{j_r}$ for the parameter $t_{j_r}$  by way of the formula 
$k_{j_r} = u_{r}\cdot t_{j_r}^{\max }(k_1,\dots ,k_{j_r-1})$.  
\end{defn}

\commenth{In line above $u_r$ had been incorrectly written as $u_{j_r}$, so I fixed it.  We have potentially fewer $u$ coordinates than $t$ coordinates.}

\begin{example}
Given $x_1(5) \in U(s_1)$,  given the word $(1,1,1,1)$ with $\delta (1,1,1,1) = s_1$ and given the stratum $F_Q$ for  the subword 
$ Q = (1,1,-,1)$ with support $\{ 1,2,4\} $, we have $|Q| = 3$ and $s=2$.  
In this case $S^C$ has support $\{ 4\} $. 
We get 
$f_{F_Q}(2/5,1/3) = (2,1,0,2)\in f_{(1,1,1,1)}^{-1}(x_1(5))\cap \RR_{\ge 0}^4$.     
\end{example}

Our next results show  that $f_{F_Q}$ is well-defined and is injective.

\begin{lemma}\label{thm:CW-map}
Given $w\in W$,    
$p\in U(w)$ and a word $(i_1,\dots ,i_d)$ with  $\delta (i_1,\dots ,i_d)=w$, consider any 
stratum  $F_Q$ within $f_{(i_1,\dots ,i_d)}^{-1}(p)\cap \RR_{\ge 0}^d$
 indexed by 
a subword $Q$  of $(i_1,\dots ,i_d)$ satisfying  $\delta (Q) = w$. Let $|Q|$ denote the number of letters
in  $Q$. 
Then  $f_{F_Q}$ 
 is a  well-defined  map  
$$f_{F_Q}: 
[0,1)^{|Q|-l(w)} \rightarrow   f_{(i_1,\dots ,i_d)}^{-1}(p)\cap \RR_{\ge 0}^d.$$
\end{lemma}

\begin{proof}
Whenever we encounter  a parameter $t_j$ for $j \in S^C$ as we proceed from left to right through the parameters as in the definition of 
$f_{F_Q}$ above,  
Lemma ~\ref{defined-forced-value} guarantees  
a  value $k_j$ for $t_j$   exists and 
is uniquely determined 
(among  the nonnegative reals)  by the choice of values $k_1,k_2,\dots ,k_{j-1}$   for the  parameters $t_1,t_2,\dots ,t_{j-1}$ 
to its left. 
Whenever we reach a parameter  $t_{j_r}$ with  $j_r \in S$,  we invoke 
Lemma ~\ref{closed-interval} 
   (in combination with Lemma ~\ref{other-char}) 
to deduce that we may  
set  $k_{j_r} = u_{r} \cdot  t_{j_r}^{\max }(k_1,\dots ,k_{j_r-1})$ for any $u_r\in [0,1)$ where $k_1,\dots ,k_{j_r-1}$ are the values of the parameters  that
have already been determined to the left of $t_{j_r}$; more specifically, after suitably  applying the  fiber-changing map 
$cf(p;k_1,\dots ,k_{j_r-1})$ we note that 
Lemma ~\ref{closed-interval}  ensures that there exists a 
point  $x\in f_{(i_1,\dots ,i_d)}^{-1}(p)\cap \RR^{d}$ having both  this value $u_r\cdot t_{j_r}^{\max }(k_1,\dots ,k_{j_r-1})$  
for $t_{j_r}$ and  also  the values $k_1,\dots ,k_{j_r-1}$ already specified for the  parameters $t_1,\dots ,t_{j_r-1}$.  
\end{proof}

\begin{lemma}\label{f-F-Q-injective}
The map $f_{F_Q}$ is injective.
\end{lemma}

\begin{proof}
Given distinct inputs $(u_1,\dots ,u_s)$ and $(u_1',\dots ,u_s')$ to the map $f_{F_Q}$, it suffices to show that $f_{F_Q}(u_1,\dots ,u_s) \ne f_{F_Q}(u_1',\dots ,u_s')$.  
 Let $f_{F_Q}(u_1,\dots ,u_s) = (k_1,\dots ,k_d)$ and let $f_{F_Q}(u_1',\dots ,u_s') = (k_1',\dots ,k_d')$.  Suppose that $u_i=u_i'$ for $i=1,2,\dots ,l-1$ whereas $u_l\ne u_l'$.  Then by definition of $f_{F_Q}$,  we have $k_i = k_i'$ for $i=1,2,\dots ,j_l-1$. 
This implies  that $t_{j_l}^{\max }(k_1,\dots ,k_{j_l-1}) = t_{j_l}^{\max }(k_1',\dots ,k_{j_1-1}')$.   
But then 
$k_{j_l}  = u_lt_{j_l}^{\max }(k_1,\dots ,k_{j_l-1})$ and   $k_{j_l}' = u_l't_{j_l}^{\max }(k_1,\dots ,k_{j_l-1})$.  This allows us to deduce that 
 $k_{j_l}\ne k_{j_l}'$ from $u_l\ne u_l'$, 
 provided that $t_{j_l}^{\max }(k_1,\dots ,k_{j_l-1})>0$.  But $t_{j_l}^{\max }(k_1,\dots ,k_{j_l-1})>0$ holds by  Lemma  
 ~\ref{any-l-positive}, a result which applies by virtue of our choice of domain $D^{<\max }_S$.
\end{proof}

\begin{example}
Consider $f_{(1,1,2,2,2)}$, in which case the rightmost subword that is a reduced word for $\delta (1,1,2,2,2) = s_1s_2$ has 
support $\{ 2,5\} $, which means $s=3$.  
Consider the fiber $f_{(1,1,2,2,2)}^{-1}(p)\cap \RR_{\ge 0}^5$ for $p = x_1(5)x_2(12)$.  Let
$(u_1,u_2,u_3) = (\frac{1}{2},\frac{1}{3},\frac{1}{4})\in [0,1)^3$ and let $(u_1',u_2',u_3')  = (\frac{1}{2},\frac{1}{3},\frac{3}{4})\in [0,1)3$. 
 Note that these first differ in their third coordinate.

Let us first calculate 
$f_{F_Q}\left(\frac{1}{2},\frac{1}{3},\frac{1}{4}\right).$  
We have $k_1 = \frac{1}{2} t_1^{\max } =  \frac{5}{2}$ which forces $k_2 = \frac{5}{2}$ as well.  Next we 
get $k_3 = \frac{1}{3} t_3^{\max }(k_1,k_2) 
= 4$, and then  $k_4 = \frac{1}{4} t_4^{\max }(k_1,k_2,k_3) 
= 2$.  This forces $k_5 = 6$.  Thus,
$f_{F_Q}(\frac{1}{2},\frac{1}{3}, \frac{1}{4}) = (\frac{5}{2}, \frac{5}{2}, 4, 2, 6)$.

Next we calculate the values $(k_1',k_2',k_3',k_4',k_5')$ given by 
$$f_{F_Q}\left(\frac{1}{2}, \frac{1}{3}, \frac{3}{4}\right) = (k_1',k_2',k_3',k_4',k_5').$$  
We get $k_1' = \frac{1}{2}  t_1^{\max } = \frac{5}{2}$,  $k_2' = \frac{5}{2}$ and 
$k_3' = \frac{1}{3} t_3^{\max }(k_1',k_2') = 4$.  However,  $k_4' = \frac{3}{4} t_4^{\max }(k_1',k_2',k_3') 
 =  6\ne k_4$.  This forces $k_5' = 2$.  Thus, 
$f_{F_Q}(\frac{1}{2},\frac{1}{3},\frac{1}{4}) 
 \ne f_{f_Q}(\frac{1}{2},\frac{1}{3},\frac{3}{4})$.   
\end{example}

Next we define a pair of maps, $h$ and $rtn$.  These will give another  way to think about the map $f_{F_Q}$, one  that will clarify its relationship to the map
$h_{k_1,\dots ,k_s}$ from Lemma ~\ref{generalized-Lusztig-homeom}.   
Readers  not needing this added  perspective 
 could consider skipping ahead to 
 Lemma ~\ref{any-then-all}.

\begin{defn}\label{h-map}
From the  maps $h_{k_{j_1},\dots ,k_{j_s}}$  introduced in Lemma ~\ref{generalized-Lusztig-homeom} we obtain a map 
$h: \RR^s\times \RR^{d-s} \rightarrow \RR^s \times Y_w^o$
given by the formula: 
$$h((k_{j_1},\dots ,k_{j_s}),(t_1,\dots ,t_{d-s})) =((k_{j_1},\dots ,k_{j_s}), h_{k_1,\dots ,k_s}(t_1,\dots, t_{d-s})). $$ 
\end{defn}

\begin{cor}\label{h-bijective}
The map $h$ is a bijection to its image.
\end{cor}

\begin{proof}
 Injectivity of $h$  follows from bijectivity of the identity map on $\RR^s$  along with the injectivity result from Theorem ~\ref{generalized-Lusztig-homeom}
  for the map $h_{k_{j_1},\dots ,k_{j_s}}$ 
 for each choice of $k_{j_1},\dots ,k_{j_s}\in \RR_{\ge 0}$. 
 Bijectivity to $im(h)$ then follows tautologically.  
\end{proof}

Lemma ~\ref{defined-forced-value}
  justifies the  well-definedness of the map $rtn$ defined next.


\begin{defn}\label{rtn-def}
Given $(i_1,\dots ,i_d)$ satisfying $\delta (i_1,\dots ,i_d) = w$, 
we define a map {\it  rtn} (short for {\it ``redundant-to-non-redundant''}) that takes as its input  
a point $(k_{j_1},\dots ,k_{j_s}) \in D_S^{<\max }$ and a point $p\in U(w)$.
%
The map  {\it rtn}   outputs the unique  $(t_1,\dots ,t_d)\in f_{(i_1,\dots ,i_d)}^{-1}(p) \cap \RR_{\ge 0}^d$ 
satisfying   $t_{j_r} = k_r$ for $j_r\in S$.     In other words, $rtn $ determines the unique possible value for 
each parameter $t_j$ for $j \in S^C$ given any 
choice of values for the parameters indexed by $S$ such that $(t_1,\dots ,t_d)\in F_{(i_1,\dots ,i_d)}^{-1}(p)\cap \RR_{\ge 0}^d$ exists.  
By Corollary ~\ref{early-f-j},  each $t_j$ for $j\in S^C$ is a function of only those parameters in $S$ having strictly smaller index than $j$.
 \end{defn}

From the definitions of $rtn$, $h$, and $D_S^{<\max }$ 
together with the bijectivity result for $h_{k_1,\dots ,k_s}$ appearing in Lemma ~\ref{generalized-Lusztig-homeom} for
each choice of $k_1,\dots ,k_s\in \RR_{\ge 0}^s$, 
one may easily observe the following relationship between $rtn $ and  the map $h$ from Definition ~\ref{h-map}.

\begin{cor}\label{defining-rtn}
The map $rtn$ equals $h^{-1}|_{D_S^{<\max }}$, 
 hence is  well-defined and injective.  
\end{cor}

Before describing  the relationship between $f_{F_Q}$ and $rtn$, we  first define a connecting  map called $sc$. 

\begin{defn}
Given $p\in U(w)$ and given $(i_1,\dots ,i_d)$ satisfying $\delta (i_1,\dots ,i_d) = w$, 
define the bijective  map $sc :[0,1)^s \rightarrow D_S^{<\max }$ by setting $sc(u_1,\dots ,u_s) = (t_{j_1},\dots ,t_{j_s})$ where 
$t_{j_r} = u_r t_{j_r}^{\max }(k_1,\dots ,k_{j_r-1})$ where $k_i=t_i$ for $i<j_r$ with $i\in \{j_1,\dots ,j_{r-1}\} $ and otherwise $k_i$ is the unique value that $t_i$ can take within $f^{-1}_{(i_1,\dots ,i_d)}(p)\cap \RR_{\ge 0}^d$ subject to our choice of values for the parameters to its left.  In other words, $sc$ is the projection of the map 
$f_{F_Q}$ to just those coordinates  indexed by $S = \{ j_1,\dots ,j_s\}$.  
\end{defn}

\begin{prop}\label{relating-maps}
Given $w\in W$,    
$p\in U(w)$ and $(i_1,\dots ,i_d)$ with   $\delta (i_1,\dots ,i_d) =w$, 
the maps $f_{F_Q}, rtn$ and $sc$ satisfy the relationship  $f_{F_Q} = rtn \circ sc$ with maps composed from right to left. 
\end{prop}

\begin{proof}
This follows directly  from how  the maps $f_{F_Q}, rtn$ and $sc$ are defined. 
\end{proof}

Next are results that  will constrain 
which strata 
are in the image of  $f_{F_Q}$.

\begin{lemma}\label{any-then-all}
Consider any $w\in W$, any word $(i_1,\dots ,i_d)$ with $\delta (i_1,\dots ,i_d)=w$, and any subword 
$Q$ 
with $\delta (Q) = w$.  
  If any 
point $(k_1,\dots ,k_d)$  in the stratum $F_Q$  has a parameter $k_j$  with  $j\in S$ satisfying $k_j = t_j^{\max }(k_1,\dots ,k_{j-1})>0$, 
then all points $(k_1',\dots ,k_d')$  in $F_Q$  satisfy $k_j' = t_j^{\max }(k_1',\dots ,k_{j-1}')>0$.   
\end{lemma}

\begin{proof}  
Consider any point $(k_1,\dots ,k_d)\in F_Q$ having 
$k_j = t_j^{\max }(k_1,\dots ,k_{j-1}) >0$ for some $j\in S$. 
Let  $w'$ denote the element of $W$ such that $cf(p;k_1,\dots ,k_{j-1}) \in U(w')$. 
Notice that $cf(p;k_1,\dots ,k_j) = cf(cf(p;k_1,\dots ,k_{j-1});k_j)$.  
This lets us use  Lemma ~\ref{any-l-positive} to deduce that  
our choice of $k_j$ satisfying  $k_j = t_j^{\max }(k_1,\dots ,k_{j-1})$   implies  
$cf(p;k_1,\dots ,k_j) \in U(s_{i_j}w')$ with $s_{i_j}w' < w'$ in Bruhat order.  
But $s_{i_j}w' < _{Bruhat} w'  $ means that $s_{i_j}w'$ does not have any reduced expressions with leftmost letter $s_{i_j}$. 

Let $Q_{j+1}$ denote  the subword of $(i_{j+1},\dots ,i_d)$ consisting of those letters 
also in $Q$, and let $Q_{j,k}$ denote the subword of $(i_j,\dots ,i_k)$ consisting of those letters also in $Q$.   We deduce from
$cf(p;k_1,\dots ,k_j)\in U(s_{i_j}w')$  and from $f_{(i_1,\dots ,i_d)}(k_1,\dots ,k_d) = p$ that 
$\delta (Q_{j+1}) = s_{i_j}w'$.  This 
forces $Q$ not to have any deletion partners for $i_j$ to the right of $i_j$.  In other words,  if we proceed from left to right through the letters in $Q$, 
this forces any  parameter $k_l$ with $l>j$  to be 0   
if  
  $Q_{j,l-1} \cup i_l$ 
is a word in which $i_j$ would  have a deletion partner to its right 
 (implying $i_j$ would also have a deletion partner to its right in $Q$).    
But all points in the stratum  $F_Q$ have the same parameters set to 0 as each other, so in particular have the same support as $(k_1,\dots ,k_d)$.  
%
 Having these parameters $t_l$  to the right of $t_j$ which would be deletion partners for $i_j$ set to $0$  at all points in $F_Q$   makes $i_j$ nonredundant in  the word obtained
 by appending $Q_{j+1}$ to the right of $i_j$.   This forces  $t_j$ 
to take its maximal  possible value at  every   point in $F_Q$, given our assumption that $k_j = t_j^{\max }(k_1,\dots ,k_{j-1})$  for some 
point $(k_1,\dots ,k_d)\in F_Q$.  
\end{proof}

\begin{example}
Let $(i_1,\dots ,i_d) = (1,2,1,2,1)$ and  let 
$Q = (1,2,1,2)$ be the subword with support $\{ 1,2,3,4\} $.    Then all points in 
 $F_Q\subseteq f_{(1,2,1,2,1)}(p) \cap \RR_{\ge 0}^5$  for $p = x_1(3)x_2(5)x_1(7)$ have $0 < t_1 < t_1^{\max }=3$ since  the conditions 
 $t_1 =  t_1^{\max }$, 
$t_2>0$, and $t_3>0$ would combine to  force $t_4=0$, contradicting our choice of support for $Q$.
\end{example}

\begin{lemma}\label{last-redundant}
Consider any  point  $p\in U(w)$ and any  word $(i_1,\dots ,i_d)$ satisfying $\delta (i_1,\dots ,i_d)=w$ and  having 
$i_1$ 
redundant in $(i_1,\dots ,i_{d})$. 
Consider any  nonempty  open stratum  $F_Q$ of $f_{(i_1,\dots ,i_d)}^{-1}(p)\cap \RR_{\ge 0}^d$ 
which has the property for some point $(k_1,\dots ,k_d)$ in $F_Q$ (equivalently for all points in $F_Q$)  that 
$0 < k_1 < t_1^{\max }$.  
This stratum  is   indexed by a  subword $Q = (i_{l_1},i_{l_2},\dots ,i_{l_r})$ of $(i_1,\dots ,i_d)$
   satisfying all of the following conditions: 
  \begin{enumerate} 
  \item $l_1 = 1$ 
  \item   $\delta  (i_1,i_{l_2},\dots , i_{l_r}) = w$
  \item 
$i_1$ is redundant in $(i_1,i_{l_2},\dots ,i_{l_r})$.
\end{enumerate}
\end{lemma}

\begin{proof}
The equivalency between $t_1$ being non-maximal  for one point and for all points in $F_Q$  is proven in Lemma ~\ref{any-then-all}.
Conditions 1 and 2 follow from our $k_1>0$ assumption and our assumption that $F_Q$ is nonempty, respectively.
 Condition 3 follows from Lemma ~\ref{last-letter-determined}.
\end{proof}

\begin{cor}\label{middle-redundant}
Consider any  $p\in U(w)$ and any  word $(i_1,\dots ,i_d)$ with  $\delta (i_1,\dots ,i_d) =w$ that has 
$i_l$  redundant in $(i_l,\dots ,i_d)$ for some $l\ge 1$. 
 Consider any nonempty  stratum $F_Q$  of $f_{(i_1,\dots ,i_d)}^{-1}(p)\cap \RR_{\ge 0}^d$ 
having the property for some point $(k_1,\dots ,k_d)$ in $F_Q$  (equivalently for all points  in $F_Q$) that  $$0 < k_l < t_l^{\max }(k_1,\dots ,k_{l-1}),$$ i.e., 
 $k_l$ is positive but does not achieve 
the maximal value it takes within $f_{(i_1,\dots ,i_d)}^{-1}(p) \cap \RR_{\ge 0}^{d}$ for our  given choice of values 
$k_1,\dots ,k_{l-1}$  
for the parameters $t_1,\dots ,t_{l-1}$.    
This stratum is indexed by a subword $Q = (i_{l_1},i_{l_2},\dots ,i_{l_r})$ of $(i_1,\dots ,i_d)$ satisfying all of the following conditions:
\begin{enumerate}
\item 
$l \in supp(Q)$,  so   $l = l_m$ for some $1\le m\le r$.
\item 
$\delta (i_{l_1},\dots ,i_{l_r}) = w$
\item
$i_{l_m}$ is redundant in $(i_{l_m},\dots ,i_{l_r})$.
\end{enumerate}
\end{cor}

\begin{proof}
Let  $p' = cf(p; k_1,\dots ,k_{l-1})$.   Since $p' \in U(\delta (i_l,\dots ,i_d))$ by Lemma ~\ref{redundant-non-max-which-delta}, 
the result follows from Lemma ~\ref{last-redundant} applied to $f_{(i_j,\dots ,i_d)}^{-1}(p')\cap \RR_{\ge 0}^d$.
\end{proof}

Now we are ready to prove that $f_{F_Q}$ is a bijection from $[0,1)^s$  to a union of strata.  

\begin{lemma}\label{thm:CW-inverse}
Consider any word $(i_1,\dots ,i_d)$ with $\delta (i_1,\dots ,i_d)=w$ and any  $p\in U(w)$. 
Let $Q$ be any subword of $(i_1,\dots ,i_d)$ that satisfies  $\delta (Q)=w$.  
Let  $v$ be the 0-cell  in $f_{(i_1,\dots ,i_d)}^{-1}(p)\cap \RR^{d}_{\ge 0}$
indexed by the rightmost subword of $Q$ that is a reduced word for $w$.
The map $f_{F_Q}$ 
is a bijection 
from $[0,1)^{|Q|-l(w)}$ to 
$\bigcup_{ v  \subseteq  \overline{F_{Q'}  }\subseteq \overline{F_Q}} F_{Q'}$.  
Moreover, $f_{F_Q}^{-1}$ may be calculated explicitly. 
\end{lemma}

\begin{proof}
Injectivity of $f_{F_Q}$ was proven in Lemma ~\ref{f-F-Q-injective}, allowing us to focus henceforth on surjectivity.
 Let $S = \{ j_1,\dots ,j_s\}$  with $j_1< \dots <j_s$ be those positions not in the support of $v$.
 We apply Corollary 
 ~\ref{middle-redundant}  
 to  deduce that $im(f_{F_Q})$ can only intersect nontrivially with 
 those strata  $F_{Q'}$ satisfying $v\subseteq \overline{F_{Q'}}\subseteq \overline{F_Q}$. 
 Our choice of  $u_1,\dots ,u_s$ each strictly less than 1 forces 
$$k_{j_1}< t_{j_1}^{\max }(k_1,\dots ,k_{j_1-1}); k_{j_2} < t_{j_2}^{\max }(k_1,\dots ,k_{j_2-1}); \dots ;
k_{j_s}<t_{j_s}^{\max }(k_1,\dots ,k_{j_s-1}).$$  By  the positivity part of Lemma ~\ref{defined-forced-value}
combined with Lemma ~\ref{other-char}, these inequalities
 imply  that the support of $Q'$ must contain the support of $v$.  Thus, we are assured that  any stratum  
$F_{Q'}$  in the image of $f_{F_Q}$ must satisfy  $v\subseteq \overline{F_{Q'}}$.

 We  will repeatedly apply 
 Corollary ~\ref{closed-interval}, 
 as described in detail next,  to  deduce surjectivity of $f_{F_Q}$  onto each  strata $F_{Q'}$ 
 satisfying $v\subseteq \overline{F_{Q'}}\subseteq \overline{F_Q}$. 
That is,    for each  point $(k_1,\dots ,k_d) \in F_{Q'}$ for  each stratum  $F_{Q'}$ satisfying 
 $v\subseteq \overline{F_{Q'}}\subseteq \overline{F_Q}$
we  will determine a point $(u_1,\dots ,u_s)\in [0,1)^s$ that satisfies 
 $f_{F_Q}(u_1,\dots ,u_s) = 
 (k_1,\dots ,k_d)$.  
   
   We must have $t_{j_1}^{\max }(k_1,\dots ,k_{j_1-1}) > 0$ as a consequence of 
 Lemma \ref{max-achieved} (after  applying the map $cf$ from Definition \ref{cf-def}  with inputs $p,k_1,\dots ,k_{j_1-1}$ so as to make $t_{j_1}$ the leftmost parameter).    We will use induction on
 $r$ to show likewise that we have $t_{j_r}^{\max }(k_1,\dots ,k_{j_r-1}) > 0$ for  $1\le r\le s$.  
Suppose we have 
$t_{j_r}^{\max }(k_1,\dots ,k_{j_r-1}) = 0$  for  some $r$,  choosing such $r$  to be as small as possible.  By Corollary  ~\ref{middle-redundant} and
by  the definition of the  Demazure
product, this would imply 
$k_{j_i} = t_{j_i}^{\max }(k_1,\dots k_{j_i-1})$  
for some $i<r$ which is a deletion partner for $i_{j_r}$. 
But that would  force $u_i=1$ for this same $i$, since by definition of $f_{F_Q}$ we must have 
$$u_i = \frac{k_{j_i}}{t_{j_i}^{\max }(k_1,\dots ,k_{j_i-1})}$$ 
with  
$t_{j_i}^{\max }(k_1,\dots ,k_{j_i-1})\ne 0$ for each $i<r$.   The $t_{j_i}^{\max }(k_1,\dots ,k_{j_i-1})\ne 0$ requirement for $i<r$  follows from our minimality assumption on $r$.  
But we cannot have $u_i=1$ for any  $i<r$, by virtue of 
our choice of domain 
$[0,1)^{|Q|-l(w)}$, giving the desired contradiction.  

Having $t_{j_r}^{\max }(k_1,\dots ,k_{j_r-1}) > 0$  for each $r$  allows us to set 
  $u_r  = \frac{k_{j_r}}{t_{j_r}^{\max }(k_1,\dots ,k_{j_r-1})}\in [0,1)$.  By Corollary ~\ref{closed-interval}, we may obtain any element of $[0,1)$ as the value for 
  $u_r$ in this manner, since Corollary ~\ref{closed-interval} guarantees that 
  $k_{j_r}$ can  achieve any value within $[0,t_{j_r}^{\max }(k_1,\dots ,k_{j_r-1}))$.  
%
It is  now straightforward to  check that setting 
$ f_{F_Q}^{-1}(t_1,\dots ,t_d) = (u_1,\dots ,u_{|Q|-l(w)} )$ with the right side  calculated in this manner  
indeed gives  the valid inverse
map for $f_{F_Q}$.   The existence of this inverse map  implies bijectivity. 
%
\end{proof}

Our next result, Lemma ~\ref{max-to-unique},  relates the maximal value of a redundant  parameter to the unique value for this parameter when we pass to a subword in 
which it is  non-redundant.  
This will help us prove continuity of the function $t_{j_r}^{\max }$ in Lemma ~\ref{cont-max}.
 
\begin{lemma}\label{max-to-unique}
Consider   any point $p\in U(w)$ and any  word $(i_1,\dots ,i_d)$ satisfying   $\delta (i_1,\dots ,i_d)=w$ and having   $i_1$ redundant in $(i_1,\dots ,i_d)$.
Let $k_1$ 
 be the largest value that $t_1$ 
 takes at any point $(t_1,\dots ,t_d)\in f_{(i_1,\dots ,i_d)}^{-1}(p)\cap \RR_{\ge 0}^d$.
If  a  subword $(i_{j_1},i_{j_2},\dots ,i_{j_m})$ of $(i_1,\dots ,i_d)$ meets all  the conditions
\begin{enumerate}
\item 
$j_1=1$,
\item
 $i_{j_1}$ non-redundant in $(i_{j_1},\dots ,i_{j_m})$, and 
\item
$\delta (i_{j_1},\dots ,i_{j_m}) = w$,
\end{enumerate} 
then 
 $k_1$ is  the unique value  $t_{j_1}$ takes among points $(t_{j_1},\dots ,t_{j_m})\in 
f_{(i_{j_1},\dots ,i_{j_m})}^{-1}(p)\cap \RR_{\ge 0}^d$. 
%
\end{lemma}

\begin{proof}
This follows directly from Corollary ~\ref{unique-value-other-cell}  combined with  Lemma ~\ref{max-achieved}. 
\end{proof}

\begin{example}\label{largest-becomes-unique}
Notice for  $(t_1,t_2,t_3) \in f_{(1,1,2)}^{-1}(x_1(3)x_2(7))\cap \RR_{\ge 0}^3$, the parameter $t_1$ is indexed by the letter $1$ that is redundant in $(1,1,2)$, and observe 
that $t_1$ takes  maximal value of 3 within $f_{(1,1,2)}^{-1}(x_1(3)x_2(7))\cap \RR_{\ge 0}^3$.  
  If we instead consider
$(t_1,t_3) \in f_{(1,2)}^{-1}(x_1(3)x_2(7))$, then the parameter $t_1$ is non-redundant with unique value 3.  The fact that we get the same value of 3 in both cases 
follows from the fact that setting $t_1$ to its maximal possible
value in $f_{(1,1,2)}^{-1}(x_1(3)x_2(7))$  forces $t_2$ to be 0, in effect replacing the word $(1,1,2)$ by the word $(1,2)$ as our way of recording 
the fact that  $x_1(t_1)x_1(0)x_2(t_3) = x_1(t_1)x_2(t_3)$.  
\end{example}

Next we prove 
that $f_{F_Q}$ and $f_{F_Q}^{-1}$ are continuous functions; this is much more subtle when $Q$ is non-reduced than for $Q$ reduced, a point we highlight now 
because  our proof  will justify and make use of the fact  that $f_{F_Q}^{-1}$ is continuous when $Q$ is reduced. 
As stepping stones,  we first prove continuity 
of $t_{j_r}^{\max }$  for each $j_r\in S$, 
and  we prove continuity 
of the function $f_{j}$ from Definition ~\ref{f-j-def} for each $j\in S^C$.

\begin{lemma}\label{cont-forced-value}
Consider any $w\in W$, any word $(i_1,\dots ,i_d)$ with $\delta (i_1,\dots ,i_d) = w$, and any $p\in U(w)$.  
For each $j\in S^C$, 
the function $f_{j}$ 
is  
continuous. 
\end{lemma}

\begin{proof}
We already showed  in the proof of Lemma ~\ref{defined-forced-value} that  $f_j$ outputs 
 the unique  value for the  
leftmost coordinate  for all points in $f_{(i_j,\dots ,i_d)}^{-1}(p_{j-1}) \cap \RR_{\ge 0}^{d-j+1}$ 
for  $p_{j-1} =   x_{i_{j-1}}(-k_{j-1}) \cdots x_{i_1}(-k_1)p$. 
Using this, we  will show how to deduce the desired   continuity result 
from the   following  facts: 
\begin{enumerate}
\item
Let $Q'$ be a  subword of $(i_j,\dots ,i_d)$ that is reduced, satisfies $\delta (Q') = \delta (i_j,\dots ,i_d)$  and contains $i_j$.  Then $f_{Q'}$ is a homeomorphism from  
$\RR_{>0}^{|Q'|}$ to its image within $U(w)$   (by part (d) of Theorem \ref{Lu-theorem}), implying $f^{-1}_{Q'}$ is continuous and hence that each of the coordinate
functions comprising $f^{-1}_{Q'}$ is also continuous.  
\item
Also note that $p_{j-1} = x_{i_{j-1}}(-k_{j-1})\cdots x_{i_1}(-k_1) p $
 is  a continuous function of  $k_1,\dots ,k_{j-1}$, a fact which follows from the 
definition of a pinning (see e.g.\ Definition ~\ref{pinning-def} or  Section 1 of \cite{Lu} or Section 2.1 of \cite{RW}) by the following reasoning.
 Since homomorphisms of Lie groups are smooth functions, this ensures  continuity of each of the functions $x_{i_s}$ in a single variable $k_s$  for $s=1,2,\dots, j-1$, and hence  
continuity of the desired  product of such functions (composed with negation functions)   in the variables $k_1,\dots ,k_{j-1}$.
 \end{enumerate}
 Arguing by induction on $j$, it follows from the second  fact above  
 that $p_{j-1}$ is a continuous function of $k_{j_1},\dots ,k_{j_{r-1}}$ where $\{ j_1,\dots ,j_{r-1} \} $ are the elements of $S$ that are less than $j$ when at least one element of $S$
 is less than $j$, as we now explain.   The point is to use  the fact that we may assume by induction  for each $j'<j$ with $j'\in S^C$ that $k_{j'}$ is a continuous function of 
 $k_{j_1},\dots ,k_{j_{r'}}$ where 
 $j_{r'}$ is the largest element of $S$ such that   $j_{r'} < j' $.
 The base case for the induction is any $j$ having no elements of $S$ to its left, in which case $f_j$ is a 
 constant function and hence continuous.  The upshot is that $f_j$ is a continuous function in a collection of parameters that by induction are each continuous functions in
 just the allowed  parameters indexed by $S$.  
 Since  $f_j$ is one of the coordinate functions of $f_{Q'}^{-1}$, it follows from  fact 1 that it is a continuous function of $p_{j-1}$, hence is a continuous function of the 
 allowed parameters.
\end{proof}

For the next example, we build upon Example ~\ref{specific-instance}.   That is, Example ~\ref{specific-instance} gives the functions we described next evaluated at the point
$t_1 = 4$.   

\begin{example}
Let $(i_1,\dots ,i_d) = (1,2,1,2)$ so $S = \{ 1\}  \subseteq \{ 1,2,3,4\} $ since the rightmost reduced word for  $\delta (1,2,1,2) = s_1s_2s_1$  within $(1,2,1,2)$
has support $S^C = \{ 2,3,4\} $.  Consider
$(t_1,t_2,t_3,t_4) \in f_{(1,2,1,2)}^{-1}(p)\cap \RR_{\ge 0}^4$ for $p = x_1(9)s_2(3)x_1(7)$. 
We have $f_2(t_1) = \frac{21}{16-t_1}, f_3(t_1) = 16-t_1$ and  $f_4(t_1) = \frac{(9-t_1)\cdot 3}{16-t_1}$.
\end{example}

\begin{remark}\label{new-t-j-max}
In upcoming results, it will be useful to speak of $t_{j_r}^{\max }(k_{j_1},\dots ,k_{j_{r-1}})$ for $j_r\in S$.
  That is, we define $t_{j_r}^{\max }$ similarly to before, but now utilizing the fact that we 
  may regard it as a function of only those parameters to the left of $t_{j_r}$ which are indexed by positions in $S$.  That is, we 
let   $t_{j_r}^{\max }(k_{j_1},\dots ,k_{j_{r-1}})$ 
equal 
the maximal value that   $t_{j_r}$ takes  within
$f^{-1}_{(i_{j_r},\dots ,i_d)}(p_r) \cap \RR_{\ge 0}^{d-(j_r-1)}$ 
where  
$p_r :=
x_{i_{j_r-1}}(-k_{j_r-1}) \cdots x_{i_2}(-k_2)x_{i_1}(-k_1)p$ where each $k_i$ with $i<j_r$ and $i\in S^C$ is uniquely determined by the values of the parameters to its left. 
\end{remark}

\begin{lemma}\label{cont-max}
Consider any $w\in W$, any word $(i_1,\dots ,i_d)$ with $\delta (i_1,\dots ,i_d) = w$, and any  $p\in U(w)$.
For  each $j_r  \in S$, 
the function  $t_{j_r}^{\max } $ 
restricted to  the domain $D^{< \max }_{S,r-1}$ is a continuous function with image contained in  the positive reals.
\end{lemma}

\begin{proof}
In what follows, we use that $t_{j_r}^{\max }$ may be regarded as a function of the values of the parameters to the left of $t_{j_r}$ that are 
 in $S$, as explained in 
Remark ~\ref{new-t-j-max} and in the definition of $t_{j_r}^{\max }$ itself.  
This allows us to speak of $t_{j_r}^{\max }(k_{j_1},\dots ,k_{j_{r-1}})$ rather than  
$t_{j_r}^{\max }(k_1,\dots ,k_{j_r-1})$, using that $k_{j_1},\dots ,k_{j_{r-1}}$ together with $p$ 
 uniquely determine $k_1,\dots ,k_{j_r-1}$.  
Lemma ~\ref{any-l-positive}  guarantees  $t_{j_r}^{\max }(k_{j_1},\dots ,k_{j_{r-1}})>0$ for 
each $(k_{j_1},\dots ,k_{j_{r-1}})\in D^{<\max }_{S,r-1}$.  But then 
 Lemma ~\ref{max-achieved} ensures that  setting $t_{j_r}$ equal to   $t_{j_r}^{\max }(k_{j_1},\dots ,k_{j_{r-1}})$  does give at least one solution  
  to   $f_{(i_1,\dots ,i_d)}(t_1,\dots ,t_d) =p$ with $(t_1,\dots ,t_d)\in \RR_{\ge 0}^d$ satisfying the constraints $t_{j_r} = t_{j_r}^{\max }(k_{j_1},\dots ,k_{j_{r-1}}) >0$ and 
  $ t_i=k_i$ for $i=j_1,\dots ,j_{r-1}$.
 \commenth{I think  someone had changed the next statement, making it not what we need.  It looks like they were 
 trying to make it sound exactly like the corollary being cited  instead of having 
 incorporated the indices/variables we actually need.
   We had wanted to say something that is not what one would expect, but which is what we need and which  indeed follows immediately from this corollary we use.  
    I just  switched it back to what we need so we get the conclusion we actually need.} 
 As already discussed, the values $k_{j_1},\dots ,k_{j_{r-1}}$ along with $p$ uniquely determine the values for all of the parameters $k_1,\dots ,k_{j_r-1}$ (as 
 justified by  Corollary ~\ref{early-f-j} and  Lemma ~\ref{defined-forced-value}), 
  giving a unique choice of 
  point $cf(p;k_1,k_2, \dots ,k_{j_r-1})$ that is consistent with the point  $p$ and the 
  values $k_{j_1},\dots ,k_{j_{r-1}}$.
  By Corollary ~\ref{any-l},  
  setting $t_{j_r}$ to its maximal possible value   $k_{j_r} := t_{j_r}^{\max }(k_{j_1},\dots ,k_{j_{r-1}})$ causes  the resulting point  
   $cf(p; k_1,\dots ,k_{j_r})$  
   to be in $U(s_{i_{j_r}}\delta(i_{j_r+1},\dots ,i_d))$.  Moreover, 
   we must have $s_{i_{j_r}}\delta(i_{j_r+1},\dots ,i_d) < 
   \delta (i_{j_r+1},\dots ,i_d)$  in Bruhat order by virtue of  $j_r$ being in $S$ and hence $i_{j_r}$ being redundant in $(i_{j_r},\dots ,i_d)$.   
  
   This containment  $cf(p;k_1,\dots ,k_{j_r}) \in U(s_{i_{j_r}}\delta (i_{j_r+1},\dots ,i_d))$ we have just deduced   implies 
  further constraints  on all  points  
  $(k_1,\dots ,k_{j_r-1},t_{j_r}^{\max }(k_1,\dots ,k_{j_r-1}),t_{j_r+1},\dots ,t_d)$ that are in 
  $ f_{(i_1,\dots ,i_d)}^{-1}(p)\cap \RR_{\ge 0}^d,$ 
 namely it implies   that enough parameters to the right of $t_{j_r}$ must be set to 0 so that the  resulting 
   subword of $(i_{j_r+1},\dots ,i_d)$  consisting of exactly  the  letters  $i_s$  in $(i_{j_r+1},\dots ,i_d)$ 
   with $t_s>0$ is a subword   having  Demazure product 
   $s_{i_{j_r}}\delta(i_{j_r+1},\dots ,i_d)$. 
   One such subword $Q'$ is the leftmost subword for $s_{i_{j_r}}\delta (i_{j_r+1},\dots ,i_d)$ within 
   $(i_{j_r+1},\dots ,i_d)$.  Let $Q''$ be the word obtained from this $Q'$ by appending $i_{j_r}$ to the immediate left of the word $Q'$, 
   so in other words $Q''$ is the leftmost 
   subword of $(i_{j_r},\dots ,i_d)$ that is a reduced word for $\delta (i_{j_r+1},\dots ,i_d)$. 
Since   
$Q''$ is reduced,   the letter $i_{j_r}$ is non-redundant in $Q''$.  By Lemma ~\ref{last-letter-determined}, 
this  implies both that 
there is a unique value for
$t_{j_r}$ in $f_{Q''}^{-1}(p_r)\cap \RR_{\ge 0}^{|Q''|}$ where $p_r = cf(p;k_1,\dots ,k_{j_r-1})$ and also  
that this value for $t_{j_r}$  is positive. 

By Lemma ~\ref{max-to-unique},  the maximal value for $t_{j_r}$ in 
$$f_Q^{-1}(p)\cap \{ (k_1,\dots ,k_{j_r-1},t_{j_r},t_{j_r+1},\dots ,t_d) \in \RR_{\ge 0}^d  \} $$  
is exactly the unique value for  the leftmost parameter (i.e. the parameter we are still calling  $t_{j_r}$) in $f_{Q''}^{-1}(cf(p_r)) \cap \RR_{\ge 0}^{d-k_r+1}$.  

Thus, it suffices to prove 
that the unique value for the leftmost parameter $t_{j_r}$  in 
$f_{Q''}^{-1}(p_r)\cap \RR_{\ge 0}^{d-(j_r-1)}$ is a continuous function of   $k_{j_1},\dots ,k_{j_{r-1}}$ on the 
domain  $D^{ < max}_{S,r-1}$.
Since   Lemma ~\ref{redundant-non-max-which-delta} implies that  
$Q''$   does not depend on  the   choice of  values 
$k_{j_1},\dots ,k_{j_{r-1}}$ within 
 $D^{< max}_{S,r-1}$, 
the result now follows from Lemma \ref{cont-forced-value}. 
%
\end{proof}

We are finally ready to show how   the maps $f_{F_Q}$ given by  the  subwords $Q$ of $(i_1,\dots ,i_d)$  
having $\delta (Q) = w$  
give rise to   our desired family of  homeomorphisms:

\begin{thm}\label{thm:CW}
Consider 
any   $p\in U(w)$,  any word $(i_1,\dots ,i_d)$ with  $\delta (i_1,\dots ,i_d)=w$, and any  
stratum  $F_Q\subseteq  f_{(i_1,\dots ,i_d)}^{-1}(p)\cap \RR_{\ge 0}^{d}$ 
given by a subword $Q$ of $(i_1,\dots ,i_d)$ with $\delta (Q) = w$.  
Let  $v  $  be   the  stratum (consisting of a single point) 
in $ \overline{F_Q}$ whose support 
indexes the rightmost  subword  of  $Q$  that is a reduced word for $w$.  
 Then  the   map $f_{F_Q}:
[0,1)^{|Q|-l(w)} \rightarrow f_{(i_1,\dots ,i_d)}^{-1}(p) \cap \RR_{\ge 0}^d$ 
is  a homeomorphism from $[0,1)^{|Q|-l(w)}$
 to  
 $\bigcup_{v\subseteq \overline{F_{Q'} } \subseteq \overline{F_Q} } F_{Q'} $.
\end{thm}

\begin{proof}
Recall that $f_{F_Q} (u_1,\dots ,u_s) = (t_1,\dots ,t_d)$.  
Note that $S^C$ is  
the support of $v\in \overline{F_Q}$.
In Lemma ~\ref{thm:CW-inverse}, we  prove that $f_{F_Q}$  is  a bijection from $[0,1)$  to  
$\bigcup_{v\subseteq \overline{F_{Q'}}\subseteq \overline{F_Q}} F_{Q'}$. 
Continuity of $f_{F_Q}$   will follow from the following  three  facts:
\begin{enumerate}
\item  
Each parameter $t_i$  with $i\not\in S$ is  uniquely determined by our choice of values for the parameters to its left  indexed by $S$, as this  is  
proven  in Lemma ~\ref{defined-forced-value}.
\item 
Each parameter $t_i$  with $i\not \in S$  
is  a continuous function of the set of   
parameters $\{ t_j| j<i;  j\in S\}  $ 
by virtue of  Lemma  ~\ref{cont-forced-value} combined with 
the assumption we may make by induction  that  each $t_j$ for $j< i$  with $j\not\in S$ 
is  a continuous function of  those variables among  $t_1,\dots , t_{j-1}$ with index in $S$. 
\item 
For  $1\le r \le d'$, the 
quantity  $t_{j_r}^{\max }(k_{j_1},\dots ,k_{j_{r-1}})$  
is a continuous function 
of the values $k_{j_1},\dots ,k_{j_{r-1}}$ 
chosen for the parameters $t_{j_1},\dots ,t_{j_{r-1}}$ 
 to the left of $t_{j_r}$  which are 
  indexed by elements of $S$, as this is proven in Lemma ~\ref{cont-max}.
 \end{enumerate}
 These facts allow us to deduce  for fixed $p$ and for  $f_{F_Q}(u_1,\dots ,u_{|Q|-l(w)}) = (t_1,\dots ,t_d)$ that each $t_i$ is a 
 continuous function of $u_1,\dots ,u_{|Q|-l(w)}$ 
 whether $i\in S^C$ or $i\in S$, by proceeding from left to right as follows.  For $j_r\in S$, we solve for $t_{j_r}$ in the formula 
 $u_{j_r} = \frac{t_{j_r}}{t_{j_r}^{\max }(k_{j_1},\dots ,k_{j_{r-1}})}$ that appeared in proof of Lemma \ref{thm:CW-inverse}.  We use  the third claim  above to get continuity of $t_{j_r}$ in this case (using  the inductive assumption
 that  $k_{j_1},\dots ,k_{j_{r-1}}$ were already calculated as continuous functions of $u_1,\dots ,u_{|Q|-l(w)}$ earlier in proceeding from left to right).  We use the first and second claims above to deduce 
 continuity for   those $t_i$ with $i\not\in S$.  Combining these continuity results for the coordinate functions comprising $f_{F_Q}$  yields  continuity of $f_{F_Q}$ itself.  
 
 Finally, we use  the fact that any  continuous, bijective map from a compact space to a Hausdorff space 
  is  a homeomorphism to show that $f_{F_Q}$ is a homeomorphism. 
    \end{proof}

\begin{example}
Let $(i_1,\dots ,i_d) = (1,2,1,2,1)$ and consider $p = x_1(3)x_2(5)x_1(7) \in U(s_1s_2s_1)$.  We now describe in this case  how $f_{(1,2,1,2,1)}^{-1}(p) \cap \RR_{\ge 0}^5$ breaks into various unions of strata that are each homeomorphic to $[0,1)^s$ for suitable positive integers $s$.  The stratum $F_{(1,2,1,2,1)}$ consisting   of those points 
$(t_1,t_2,t_3,t_4,t_5)\in f_{(1,2,1,2,1)}^{-1}(p) \cap \RR_{\ge 0}$ with  all five parameters  positive has its points  
indexed by the ordered pairs $(t_1,t_2)$ satisfying $0<t_1< 3$ and 
$0< t_2 < \frac{35}{7(3-t_1)}$ since there is exactly one point in the stratum for each such choice.  Theorem ~\ref{thm:CW} gives a homeomorphism from $[0,1)^2$ to  the union of strata  
$$F_{(1,2,1,2,1)} \cup  F_{(-,2,1,2,1)} \cup F_{(1,-,1,2,1)} \cup F_{(-,-,1,2,1)}.$$  
One may  check that this homeomorphism restricts to a homeomorphism from  $(0,1)^2$ to $F_{(1,2,1,2,1)}$.  The 
other strata in this collection are obtained by setting $t_1=0$ or $t_2=0$ or both.  

Theorem ~\ref{thm:CW} also gives a homeomorphism from  $[0,1)^1$ to    $F_{(1,2,-,2,1)} \cup F_{(1,-,-,2,1)}$, as well as a homeomorphism from 
$[0,1)^1$ to $F_{(1,2,1,-,1)} \cup F_{(1,2,-,-,1)}$ and another homeomorphism 
from  $[0,1)^1 $ to  $F_{(1,2,1,2,-)}\cup F_{(-,2,1,2,-)}$.  Finally we have the individual 
stratum  $F_{(1,2,1,-,-)}$ which is homeomorphic to $[0,1)^s$ for $s=0$, namely to a point.    Thus, each strata in our stratification for 
$f_{(1,2,1,2,1)}^{-1}(p)\cap \RR_{\ge 0}^d$ is in one of these collections of strata. 

In fact, we may apply Theorem ~\ref{thm:CW} choosing $Q$ to index any subword $Q$ of $(1,2,1,2,1)$ that is a reduced word for $s_1s_2s_1$. 
For instance,  $Q = (-,2,1,2,1)$ gives  rise to a homeomorphism 
$f_{F_Q}$ from 
$[0,1)^1$ to  $F_{(-,2,1,2,1)}\cup F_{(-,-,1,2,1)}$ that restricts to a homeomorphism from $(0,1)$ to this $F_Q$. 
This last point is important in that it allows us to use the next result, Theorem ~\ref{p:cells}, to show that every  
stratum is homeomorphic to an open ball, since each stratum  is indexed by some $Q$ which can  be used  as the input to Theorem ~\ref{thm:CW}.  
\end{example}

Now  we come to the main result in this section. 

\begin{thm}\label{p:cells}
Consider any 
$w \in W$, any $p\in U(w)$, any word $(i_1,\dots ,i_d)$ with  $\delta (i_1,\dots ,i_d) = w$, and any subword $Q$ of $(i_1,\dots ,i_d)$.
The stratum  $F_Q$ 
is either empty or
a cell of dimension $|Q| - \ell(w)$.
 $F_Q$ is empty  precisely when
$\delta(Q) \not = w$.   $F_Q$  is a single point 
exactly when  $Q$ is a reduced word for $w$. 
 The standard cell stratification of  $\RR_{\ge 0}^{d}$ induces a 
cell stratification of $f_{(i_1,\dots ,i_d)}^{-1}(p)\cap \RR_{\ge 0}^{d}$ into the cells $F_Q$ having $Q$ a subword of $(i_1,\dots ,i_d)$ with $\delta (Q) =w$.   

More generally, for any  $p\in U(w)$ and any
word $(i_1,\dots ,i_{d'})$ having $\delta (i_1,\dots ,i_{d'})\ge w$, 
the standard cell stratification of $\RR_{\ge 0}^{d'}$ induces a cell stratification of 
$f_{(i_1,\dots ,i_{d'})}^{-1}(p)\cap \RR_{\ge 0}^{d'}$ into the cells $F_Q$ having $Q$ a subword of $(i_1,\dots ,i_{d'})$ with  
$\delta (Q) = w$.   On the other hand, for any word $(i_1,\dots ,i_{d'})$ not greater than or equal to $w$ in Bruhat order, 
$f_{(i_1,\dots ,i_{d'})}^{-1}(p)\cap \RR_{\ge 0}^{d'} = \emptyset $.  
\end{thm}

\begin{proof}
First suppose $\delta (i_1,\dots ,i_d) = w$.  When $\delta (Q) = w$, the fact that $F_Q$ is homeomorphic to an open ball of 
dimensional $|Q|-l(w)$ 
follows from Theorem ~\ref{thm:CW} by noting that   $f_{F_Q}|_{(0,1)^{|Q|-l(w)}}$ has image exactly the  stratum 
$F_Q \subseteq \RR_{\ge 0}^{d}$.  This shows that the set of  such  strata $F_Q$ comprise 
 a cell decomposition of $f_{(i_1,\dots ,i_d)}^{-1}(p)\cap \RR_{\ge 0}^d$.
The  further requirement  needed for this cell decomposition  
$$\{ F_Q \mid  Q \hspace{.05in}{\rm is}\hspace{.05in}{\rm a}\hspace{.05in}{\rm subword}\hspace{.05in}{\rm of} \hspace{.05in} (i_1,\dots ,i_d)\hspace{.05in}{\rm and}\hspace{.05in} \delta (Q)=w\} $$ to be 
a cell stratification for $f_{(i_1,\dots ,i_d)}^{-1}(p)\cap \RR_{\ge 0}^{d}$
  is  proven in Lemma \ref{decomp-is-strat}.

Our  claims regarding  empty strata  for $\delta (Q) \ne w$ and empty cell stratifications for 
$(i_1,\dots ,i_{d'})$ with $\delta (i_1,\dots ,i_{d'})$ not greater than or equal to $w$ in Bruhat order   
follow from Proposition ~\ref{Lu-first-prop}.

To see our claim  regarding words $(i_1,\dots ,i_{d'})$ having $\delta (i_1,\dots ,i_{d'}) >  w$, 
we use the consequence of  Proposition ~\ref{Lu-first-prop}  that a  stratum $F_Q$ can only by nonempty if $\delta (Q) = w$, allowing us to restrict attention to the set of  such strata. 
Now consider each maximal subword  $Q$ of $(i_1,\dots ,i_{d'})$ having Demazure product $w$.  We showed already  that we have the  desired cell stratification  if $(i_1,\dots ,i_{d'}) = Q$, and therefore also if we restrict to the support of $Q$ for any fixed subword $Q$ of 
$(i_1,\dots ,i_{d'})$ having $\delta (Q) = w$.  We claim that the  cell stratification   obtained  by restricting to  the support of such a maximal  subword $Q$ of $(i_1,\dots ,i_{d'})$  having $\delta (Q)=w$  is compatible with the  cell stratification obtained  by instead  restricting to the support of $Q'$ for any other  maximal subword $Q'$ of $(i_1,\dots ,i_{d'})$ also  having 
$\delta (Q') = w$, by which we mean  that any subword $Q''$  of $(i_1,\dots i_{d'})$ that is a subword of  both $Q$ and of $Q'$ and satisfies $\delta (Q'') = w$   gives rise to  strata in the restrictions to $supp(Q)$ and to  $supp(Q')$  that exactly coincide; the claim that these strata  coincide follows from the fact that in either case the resulting  stratum $F_{Q''}$ is obtained by intersecting the stratum of 
$\RR_{\ge 0}^{d'}$  whose positive coordinates are exactly those in $supp(Q'')$  
with $f_{(i_1,\dots ,i_{d'})}^{-1}(p)$.
\end{proof}

\begin{lemma}\label{decomp-is-strat}
Consider any  $p\in U(w)$ and any word $(i_1,\dots ,i_d)$ with $\delta (i_1,\dots ,i_d)=w$.  
The decomposition of $f_{(i_1,\dots ,i_d)}^{-1}(p)\cap \RR_{\ge 0}^{d}$    into strata  
$$\{ F_Q| Q \hspace{.05in} {\rm is}\hspace{.05in}{\rm a}\hspace{.05in}{\rm subword}\hspace{.05in}{\rm of}\hspace{.05in} (i_1,\dots ,i_d)\hspace{.05in}{\rm with}\hspace{.05in}\delta (Q) = \delta (i_1,\dots ,i_d) \} $$ 
has the property that whenever any point of a strata $F_P$ is in the closure of a stratum $F_Q$ then  $F_P\subseteq \overline{F_Q}$.  In other words, 
$F_P\cap \overline{F_Q}\ne \emptyset $ implies $F_P\subseteq \overline{F_Q}$.  
\end{lemma}

\begin{proof}
Our assumption that  $F_P\cap \overline{F_Q}\ne \emptyset $ ensures  that 
the support of $F_P$, namely  the set of  parameters giving rise to the strictly positive  coordinates  in the points of  $F_P$, 
must be a subset of the support of $F_Q$. 
%
We will construct a series of strata 
$F_{P_0}, F_{P_1},F_{P_2},\dots ,F_{P_{s-1}},F_{P_{s}}$ with  
$ \dim F_{P_{i+1}} = \dim F_{P_i} + 1$ for $i=0,1,\dots ,s-1$  and with $F_{P_0}=F_P$ and $F_{P_s}=F_Q$; 
we will show $F_{P_i}\subseteq \overline{F}_{P_{i+1}}$  (and hence  $\overline{F}_{P_i}\subseteq \overline{F}_{P_{i+1}}$) for  $i=0,1,\dots ,s-1$. 
To accomplish this  goal, we will construct  for any $x = (k_1,\dots ,k_d) \in F_{P_i}$ 
 a path in $F_{P_{i+1}}$ having $x$ as a limit point. 
We focus on the case of $i=0$, noting that the same reasoning applies  for each $i\in \{ 0,1,\dots ,s-1\} $.  

Given $x = (k_1,\dots ,k_d) \in F_{P_0}$, we 
choose a   stratum $F_{P_1}$ and a  path $p_{P_1}$ in $F_{P_1}$  having $x$ as a limit point
 as follows.  Take the rightmost parameter $t_r$   in the support of $F_Q$ that is  not in the support of $F_{P_0}$.  Let $F_{P_1}$ be the stratum  
  whose support is the support of $F_{P_0}$ with $t_r$ added to it.   
 Notice that $i_r$ must have a deletion partner 
  in  the subword $P_0\cup i_r$  of $Q$ since $\delta (P_0) = \delta (Q)$.  Without loss of generality, assume this deletion partner is to the left of $i_r$.  
 Denote by $i_l$ the rightmost  deletion partner for $i_r$ that is  to the left of $i_r$ in  
 $P\cup i_r$.    The  path $p_{P_1}$ 
  will hold fixed the 
 values of all parameters to the left of $t_l$ as well as the values of all parameters to the right of $t_r$, giving them the same values they have at the point $x$.  
 This  allows us to multiply  $p$  on the left by $x_{i_{l-1}}(k_{l-1})\cdots x_{i_1}(-k_1)$ and on 
 the right by  $x_{i_d}(-k_d)\cdots x_{i_{r+1}}(-k_{r+1})$ to obtain a new point $p'\in U(w')$ for some $w'\in W$ so as to reduce what needs to be done to the task of 
 describing a suitable projected path $p_{P_1}^{proj}$ 
 in $f_{(i_l,\dots ,i_r)}^{-1}(p')$. 
 This 
 will be the projection of our path $p_{P_1}$ to the coordinates 
 $t_l,t_{l+1},\dots ,t_r$ and so will  complete  the description of the  path $p_{P_1}$.
 First note that  $k_l$ must satisfy 
  $k_l = t_l^{\max } 
  > 0$ since $k_r=0$ in $x = (k_1,\dots ,k_d)$ and since $i_l$ is non-redundant in $(i_l,\dots ,i_{r-1})$;  it may help some readers to note that 
  we are not regarding  $t_l^{\max }$ as a function of $k_1,\dots ,k_{l-1}$ because of our step 
  replacing $(i_1,\dots ,i_d)$ by $(i_l,\dots ,i_r)$ and replacing $p$ by $p'$. 
 
  \commenth{I just corrected a typo, changing $(1,k_l]$ to $(0,k_l)$.}
    Our  projection  $p_{P_1}^{proj}$ 
   is obtained by considering each real number $k_l' \in (0,k_l)$ and  describing  how to obtain the 
   values of all the parameters $k'_{l+1},\dots ,k_r'$ as continuous functions of $k_l'$.
    We hold fixed at 0 any parameter  $t_j$ for $l< j <  r$ such that $k_j=0$, i.e.,  any parameter  $t_j$ 
   whose value at the point  $x$ is 0.
Each   parameter  $t_j$  for $l < j <  r$   whose value has not already been fixed at 0   
 either (a)  has $i_j$ non-redundant within $(i_j,\dots ,i_r)$ or (b) has $i_j$ redundant within $(i_j,\dots ,i_r)$.
In case (a),  we apply Lemma ~\ref{last-letter-determined} to deduce that  the value of the parameter $t_j$ is 
uniquely determined as a function of $k'_{l},k'_{l+1},\dots ,k_{j-1}',p'$. 
 In case (b), we choose  for each 
  $t_j$ such that $i_j$ is redundant  
(proceeding left to right)  the unique value that  preserves the ratio 
$k_j/t_j^{\max }(k_l,\dots ,k_{j-1}) $, in the following sense;   
we let $$k_j' = k_j\cdot \frac{t_j^{\max }(k_l',\dots ,k_{j-1}')}{t_j^{\max }(k_l,\dots ,k_{j-1})}$$ or in other words we 
set $k_j'$ to  the unique value  satisfying
$$\frac{k_j'}{t_j^{\max }(k_l',\dots ,k_{j-1}')} = \frac{k_j}{t_j^{\max }(k_l,\dots ,k_{j-1})}$$ for our given point $(k_1,\dots ,k_d)$ using the  values $k_l',\dots ,k_{j-1}'$ 
already determined from left to right.  
Since $i_r$ is the rightmost letter in $(i_l,\dots ,i_r)$, we always have 
$i_r$ non-redundant in the word  $(i_r)$, implying  that 
$k_r'$ is  uniquely determined by $p'$ together with  the values $k_l',k_{l+1}',\dots ,k_{r-1}'$  of the parameters to its left.  

Next we  confirm  that the values   $k_{l+1}',\dots ,k_{r-1}',k_r'$  chosen in this manner are each continuous functions of  the value $k_l'$  chosen for
the parameter $t_l$.  
For  each parameter   $t_j$  such that 
 $i_j$  is 
 non-redundant in $(i_j,\dots ,i_r)$, this is
immediate from  Lemma \ref{cont-forced-value}. 
 For those  parameters $t_j$  with  $i_j$ redundant in $(i_j,\dots ,i_r)$, 
  Lemma \ref{cont-max} yields   that   $t_j^{\max }$ is a continuous function of the values for the parameters $t_l,\dots ,t_{j-1}$
  to its left; by induction, we may assume that the values of the parameters $t_{l+1},\dots ,t_{j-1}$  
  are all themselves 
 continuous functions of the value  $k_l'$ of the parameter $t_l$. 
  Since $k_j$ and $t_j^{\max }(k_l,\dots ,k_{j-1})$ are constants, this shows that $k_j'$ is a continuous function of  the value $k_l'$ for the parameter 
 $t_l$ in this case   as well.  
A consequence of all this is that  the parameter $t_r$  which takes the value   0 at the point  $x$ and is being increased  from 0 along this path as we move away from $x$   is a continuous function of  $t_l$ since it is a continuous function of parameters that are each themselves continuous functions of $k_l'$. 
 By construction and by continuity, we may observe   that as  $t_l$   gets arbitrarily close to its maximal value $k_l^{\max } := k_l$,   the parameter 
 $t_r$ gets arbitrarily close to 0; 
 likewise observe  that  each 
 parameter $t_j$ for $l<j<r$  gets 
 gets arbitrarily close to its  value at the point $x$ as $t_r$ approaches 0.   
 Thus, we have exhibited the existence of  the desired path within $F_{P_1}$ having  $x$ as a limit point.
%
\end{proof}

\begin{remark}
To prove that our stratification for $f_{(i_1,\dots ,i_d)}^{-1}(p)\cap \RR_{\ge 0}^{d}$ is actually a CW decomposition, 
it would  suffice to provide characteristic maps.   One might hope to use the maps $\{ f_{F_Q} \} $ that we have defined, extending them to  the  larger domain 
$[0,1]^{d-l(w)}$. 
The resulting  
maps  $\{ f_{F_Q} \} $  are well-defined.   
 However,  they are not continuous  in general.  Thus, these maps 
 do  not endow  our  cell stratification  for $f_{(i_1,\dots ,i_d)}^{-1}(p)\cap \RR_{\ge 0}^{d}$ with  the structure of a 
 CW complex, let alone a regular CW complex.  Nonetheless, it  seems quite  plausible our cell decomposition is a regular CW decomposition with some other choice of 
 characteristic maps.
\end{remark}

We conclude this section by proving  that the bijection from Lemma ~\ref{generalized-Lusztig-homeom} is a homeomorphism.  This
is not needed  
for  our other results, but could  be of interest in its own right, as a sort of generalization of part of part (d) of  Theorem ~\ref{Lu-theorem}. 

\begin{thm}\label{generalized-Lusztig-homeo}
Given $w\in W$, let   $(i_1,\dots , i_d)$ be any   word with  
$\delta (i_1,\dots ,i_d) =w$.  
Consider 
$$D_{k_{j_1},\dots ,k_{j_s}}   =    
 \{ (t_1,\dots ,t_d) \in \RR^{d}_{\ge  0} \mid  t_{j_l'}  > 0 \hspace{.03in} {\rm for}\hspace{.03in} 1\le l \le d-s \hspace{.075in}{\rm and}\hspace{.075in} 
 t_{j_r} = k_{j_r}  \hspace{.03in}{\rm   for}\hspace{.03in} 1\le r\le s \}  $$
for  
any  fixed choice of constants  $k_{j_1},\dots ,  k_{j_s}\ge 0$ where  $\{ j_1, \dots ,j_s\}  = S$. 
Then   $f_{(i_1,\dots ,i_d)}|_{D_{k_{j_1},\dots ,k_{j_s}}}$ is 
 a homeomorphism $h_{k_{j_1},\dots ,k_{j_s}}: D_{k_{j_1},\dots ,k_{j_s}} \rightarrow im(h_{k_{j_1},\dots ,k_{j_s}})\subseteq U(w)$ 
 with 
 $im(h_{k_{j_1},\dots ,k_{j_s}}) \cong U(w) $. 
\end{thm}

\begin{proof}
Within the proof of Proposition 4.2 in \cite{Lu},  Lusztig shows that the continuous function  $f_{(i_1,\dots ,i_d)}$ is proper on the domain
$\RR_{\ge 0}^d$.  This implies 
that $h_{k_{j_1},\dots ,k_{j_s}}$  is continuous and proper on   $D_{k_{j_1},\dots ,k_{j_s}}$,  since  $h_{k_{j_1},\dots ,k_{j_s}}  = f_{(i_1,\dots ,i_d)}|_{D_{k_{j_1},\dots ,k_{j_s}}}$. 

Next observe  that $im(f_{(i_1,\dots ,i_d)})$ 
is a closed subset of Euclidean space, by virtue of being defined by closed conditions, namely being defined 
by a set of equalities and  weak inequalities; more specifically,  $im(f_{(i_1,\dots ,i_d)})$ 
is specified by requiring  that certain regular  functions (i.e. a set of generalized minors in the sense of \cite{FZ}) are nonnegative while other generalized minors must  equal  0.  
The fact that $im(f_{(i_1,\dots ,i_d)})$ is  a closed subset of Euclidean space implies it  is locally compact.  

 Restricting $f_{(i_1,\dots ,i_d)}$  to  the domain $D$ in which we  
 require the parameters  indexed by elements of 
 $S^C$  be positive 
 yields the space  $im(f_{(i_1,\dots ,i_d)}|_D)$ which equals  $U(w)$.  But $U(w)$ is itself 
 specified by a combination of open and closed conditions, namely the requirements that certain generalized minors must 
 be positive and that others must be 0.  This implies it is locally compact, since it may be regarded as  a closed subset of an 
 open subset of Euclidean space, hence a closed subset of a locally compact space. 
 This shows that  $im(f_{(i_1,\dots ,i_d)}|_D)$ 
  is locally compact.

Restricting the domain $D$  further to a domain $D' = D_{k_{j_1},\dots ,k_{j_s}}$ by also requiring $t_{j_1} = k_{j_1};t_{j_2}=k_{j_2};\dots ;t_{j_s}=k_{j_s}$  
gives a closed subset of $im (f_{(i_1,\dots ,i_d)}|_D)$  as the image of $f_{(i_1,\dots ,i_d)}|_{D'}$, since these further conditions  
are closed conditions, as we justify next.  We   apply the continuous map $r$ from  2.17 in \cite{Lu} to 
$im(h_{k_{j_1},\dots ,k_{j_s}})$ and observe that in this case the image of $r$ is a closed  subset $U'$  of $r(U(w))$, since it is a  subset  of $ r(U(w)) \subseteq 
\RR_{\ge 0}^{|I|}$ where  $I$ is the indexing set for our simple reflections and  $U'$ is given by closed conditions, namely by  weak inequalities $u_i\ge K_i$ (or in some cases equalities $u_i = K_i$ as explained shortly)  for $i\in I$ where $K_i$ is the sum of those values $k_{j_r} \in \{ k_{j_1},\dots ,k_{j_s}\} $ which satisfy the further condition that  $i_{j_r} = i$; the case of equality arises for any $i\in I$ such that  every $i_j$ within $(i_1,\dots ,i_d)$ that satisfies $i_j=i$ has $j\in S = \{ j_1,\dots ,j_s\} $.  
Since $r$ is continuous, the inverse image of the  closed set $U'$  
is closed, and hence 
  locally compact.     But $r^{-1}(U') = im(h_{k_{j_1},\dots ,k_{j_s}})$, so this completes our proof that $im(h_{j_{j_1},\dots ,k_{j_s}})$ is locally compact.  
 
 Because  $D_{k_{j_1},\dots ,k_{j_s}}$ is a  locally 
 compact Hausdorff space and $im(h_{k_{j_1},\dots ,k_{j_s}})$ 
 is  also   locally compact and 
 Hausdorff, properness of $h_{k_{j_1},\dots ,k_{j_s}}$ 
implies    $h_{k_{j_1},\dots ,k_{j_s}}$ is a closed map. 
This  together 
with bijectivity of $h_{k_{j_1},\dots ,k_{j_s}}$ 
implies  that $h_{k_{j_1},\dots ,k_{j_s}}$ is 
an open map.  
That combines with our earlier results and observations that $h_{k_{j_1},\dots ,k_{j_s}}$ is bijective and 
continuous to show that 
$h_{k_{j_1},\dots ,k_{j_s}}$ is 
a homeomorphism from $D_{k_{j_1},\dots ,k_{j_s}}$  to $im(h_{k_{j_1},\dots ,k_{j_s}})$.

To show   $im(h_{k_{j_1},\dots ,k_{j_s}}) \cong U(w) $, we will utilize  the fact 
that our  proof  that $h_{k_{j_1},\dots ,k_{j_s}}$ is a homeomorphism to its image   applies equally well  for any  
choice of nonnegative real values for 
$k_{j_1},\dots ,k_{j_s}$.  
In  particular, in  the case with  $k_{j_1} = k_{j_2} = \cdots = k_{j_s}  = 0$ our proof above yields
 exactly $U(w)$ 
 as $im(h_{0,\dots ,0})$.  
 Changing the choice of nonnegative real values 
 $k_{j_1},\dots ,k_{j_s}$  
 does not change the 
 homeomorphism type of the domain $D_{k_{j_1},\dots ,k_{j_s}}$.  
  To construct our desired homeomorphism $H$ from $im(h_{k_{j_1},\dots ,k_{j_s}})$ to $U(w)$, 
 we use  the  map $H = h_{0,\dots ,0} \circ g \circ h_{k_{j_1},\dots ,k_{j_s}}^{-1}$ 
 with $g$ as defined next.  
 Let $g: D_{k_{j_1},\dots ,k_{j_s}} \rightarrow D_{0,0,\dots ,0} $
  be the  map 
  which  takes 
   any point  $(t_1,\dots ,t_d) \in D_{k_{j_1},\dots ,k_{j_s}}$ and sets 
 $t_{j_1}=\cdots = t_{j_s} = 0$ while leaving the values of  all other parameters unchanged.   This is clearly a homeomorphism.   The desired
 result that $H$ is a homeomorphism then  follows
 from the fact  that  $g, h_{0,\dots ,0}$ and $h_{k_{j_1},\dots ,k_{j_s}}$ (as well as their inverse maps) are  all homeomorphisms,
 implying that $H$ is  a composition of three homeomorphisms.
 \end{proof}

\section{Fiber conjecture implies 
 Fomin--Shapiro Conjecture}\label{s:fsConj}
 
We now conclude the paper   by showing how a 
proof of Conjecture ~\ref{t:subword}
  would  yield as a corollary a  new, short  proof of the Fomin-Shapiro Conjecture.  In fact,  what we show is how
 Conjecture D  
   would suffice.  More precisely  we show how to derive a new proof of 
  the Fomin-Shapiro Conjecture  from  
  contractibility of fibers, so in other words would follow from Conjecture C. 
  
  In Lemma  ~\ref{intersect-with-simplex},
  we proved that $f_{(i_1,\dots ,i_d)}^{-1}(p)\cap \RR_{\ge 0}^d$ consists entirely of points 
  $(t_1,\dots ,t_d)$ where $\sum_{i=1}^d t_i = K$ for some fixed positive  real constant $K$ 
  that is determined by $p$.    This will be useful in what follows
  since   $\Delta^{d-1}_K$ is compact  for each fixed  $K>0$ whereas $\RR_{\ge 0}^d$ is not.  
    Thus, our analysis of fibers that was carried out in  Sections \ref{sub:whole-fiber}  and \ref{s:CW}  is  pertinent to  the upcoming 
     application in this section to the 
    Fomin-Shapiro Conjecture, as we explain in  more depth  at the end of this section in Remark ~\ref{tying-together}.

 \subsection{Topological background}

\commenth{I  edited the next few sentences as a prelude to this section.}

 First  let us review the key topological result  that we will use to prove how contractibility of fibers 
would combine with our other results to yield a new proof of the Fomin-Shapiro Conjecture.  Included in this review are  a couple of proofs 
given in full  detail in our paper  of results that are folklore or  discussed elsewhere with few details.
 We will use these results  to deduce  
 the homeomorphism type for the image of $f_{(i_1,\dots ,i_d)}$
from our conjecture that the fibers  are  contractible. 

The following beautiful theorem and corollary were stated (in 
somewhat different language)  as Proposition A.1 and Corollary A.2 in \cite{GLMS}.  

\begin{thm}\label{extension-theorem}
A map $f : S^{n-1} \to S^{n-1}$  so that $f^{-1}(y)$ is contractible for all $y \in S^{n-1}$ can be extended to
a map $F: B^n \to B^n$   inducing a homeomorphism $\interior B^n \to \interior B^n$.
\end{thm}

\begin{corollary}\label{tough-topology-theorem}  
Let $\sim$ be an equivalence relation on the closed ball $B^n$ so that 
\begin{itemize}
 
 \item  all equivalence classes are contractible,
 
 \item $S^{n-1}/{\sim}$ is homeomorphic to $S^{n-1}$,
 
 \item if $x \sim y$ with $x \in S^{n-1}$, then $y \in S^{n-1}$,
 
 \item if $x \sim y$ with $x \not \in S^{n-1}$, then $y = x$.
 
\end{itemize}

 Then $B$ is homeomorphic to $B/\sim $.  
\end{corollary}

\begin{proof}[Proof of Corollary \ref{tough-topology-theorem}]
 Let $\sim$ be such an equivalence relation on $S^{n-1}$.    Let $f : S^{n-1} \to S^{n-1}$ be the composite of the quotient map $S^{n-1} \to S^{n-1}/{\sim}$ with a homeomorphism to $S^{n-1}$.   Let $F : B^n \to B^n$ be produced by Theorem \ref{extension-theorem}.   Let $\sim_F$ be the equivalence relation $x \sim_F y$ if and only if $F(x) = F(y)$.   By hypothesis $\sim$ and $\sim_F$ are identical.   By the universal property of the quotient topology, there is a continuous bijection $B/{\sim_F} \to B^n$.   Since the domain is compact and the target is Hausdorff, this map is a homeomorphism.
\end{proof}

We include a proof of Theorem \ref{extension-theorem} since its elements are not familiar to combinatorialists.  The strategy is to argue that the map $S^{n-1}\to S^{n-1}/{\sim}$ is cell-like and then to apply the cell-like (= CE) approximation theorem as well as the local contractibility of the homeomorphism group of a manifold.  

The following definition is taken from the survey of Dydak \cite{Dydak}.

\begin{definition}
 A topological space is {\em cell-like} if any map  to a CW complex is null-homotopic.  A map $f : X \to Y$ is {\em cell-like} if $f$ is proper (the inverse image of any compact set is compact) and $f^{-1}(y)$ is cell-like for all $y \in Y$.
\end{definition}

The key result in this area is Siebenmann's CE-approximation theorem.

\begin{theorem}
Let $f : X \to Y$ be a cell-like map between topological manifolds of the same dimension.   Then $X$ and $Y$ are homeomorphic.
\end{theorem}

\begin{remark}
In fact, if $Y$ is, in addition, a metric space, then for any continuous $\varepsilon: X \to (0,\infty)$, there is a homeomorphism $g: X \to Y$ so that for all $x \in X$, $d(f(x),g(x)) < \varepsilon(x)$.
\end{remark}

\begin{remark}
 The above theorem was proven by Siebenmann \cite{Si} in dimensions greater than four, by Armentrout for dimensions less than four, and by Quinn \cite{Qu2} for dimension four.  
\end{remark}

\begin{proof}[Proof of Theorem \ref{extension-theorem}]
We will define a one-parameter family 
$$\Phi : [0,1] \to \map(S^{n-1},S^{n-1})$$   of self-maps of
$S^{n-1}$ so that $\Phi_0 = f$ and $\Phi_r$ is a
homeomorphism for $r \in (0,1]$.  Given such a $\Phi$,  define $F(rx) =
r \Phi_{1-r} (x)$ for $r \in [0,1] $ and $x \in S^{n-1}$.

The two key ingredients in producing $\Phi$ are the CE-Approximation Theorem and local contractibility of the homeomorphism group of a compact manifold, due independently to \v{C}ernavski{\u\i} \cite{Cern} and Edwards-Kirby \cite{EK}.  

The topology on $\map(S^{n-1},S^{n-1})$ and its subspace $\homeo(S^{n-1})$ is given by the uniform metric $d(g,h) = \sup_{x \in S^{n-1}} \| g(x) - h(x) \|$.  Local contractibility of $\homeo(S^{n-1})$ implies that for every $\epsilon > 0$, there is a $\delta > 0$ so that if $g,h \in B_{\delta}(\Id) \subset \homeo(S^{n-1})$, there is a path from $g$ to $h$ whose image lies in $B_\epsilon(\Id)$.  For $k \in \homeo(S^{n-1})$, right translation $R_k : \homeo(S^{n-1}) \to \homeo(S^{n-1}); R_k(g) = g \circ k$ is an isometry; hence for every $\epsilon > 0$, there is a $\delta > 0$ so that if $g,h \in B_{\delta}(k)$,  there is a path from $g$ to $h$ whose image lies in $B_\epsilon(k)$.  

Now for every $i \in \ZZ_{>0}$, choose $\delta_i > 0$ so that $g,h \in B_{\delta_i}(k)$ implies there is a path from $g$ to $h$ which lies in $B_{1/2^i}(k)$.  We also make the choices so that $\delta_i > \delta_{i+1}$ for all $i$.  To define the map $\Phi: [0,1] \to \homeo(S^{n-1}); r \mapsto \Phi_r$ we set $\Phi_0 = f$, define $\Phi_{1/2^i}$ using the CE-Approximation Theorem, and then connect the dots using local contractibility.  Using the CE-Approximation Theorem, choose homeomorphisms $\Phi_{2^i}$ so that $d(f,\Phi_{1/2^i}) < \delta_i/2$.  Then by the triangle inequality, $d(\Phi_{2^{i+1}},\Phi_{2^i}) < \delta_i$.  By the choice of $\delta_i$ there is a path $\Phi : [1/2^{i+1},1/2^i] \to \homeo(S^{n-1},S^{n-1})$ from $\Phi_{1/2^{i+1}}$ to $\Phi_{1/2^i}$ which lies in a ball of radius $1/2^i$.   Concatenation gives our desired path $\Phi: [0,1] \to \map(S^{n-1},S^{n-1})$.
\end{proof}

\subsection{Contractibility of fibers implies  Fomin-Shapiro Conjecture}

\begin{thm}\label{contract-implies-fs}
Suppose $W$ is finite.  Given any $w\in W$,   
 let 
$Y_w^o$  
be the  cell indexed by $w$   in 
the Bruhat decomposition of the link of the identity  in the
totally nonnegative, real  part of the unipotent radical of a Borel subgroup
in a reductive, connected algebraic  group over $\CC$ defined and split
over $\RR$.    
 Let $(i_1,\dots ,i_d)$ be any reduced word for any element of $W$. 
Contractibility of $f_{(i_1,\dots ,i_d)}^{-1}(p)$  
for all   $p\in Y_w^o$
 for   all  $w\in W$  satisfying 
$w \le_{Bruhat} \delta (i_1,\dots ,i_d) $ 
 would  imply that  
$\overline{Y_{\delta(i_1,\dots ,i_d)}^o} $ 
 is a regular CW complex homeomorphic to a closed ball. 
 In other words, Conjecture C  
 would imply that $Y_{ \delta (i_1,\dots ,i_d)} := \overline{Y_{\delta (i_1,\dots ,i_d)}^o} $ is a regular CW complex homeomorphic to a closed ball.  
\end{thm}

\begin{proof}
We interpret $Y_w:= \overline{Y_w^o}$ as the image of the closed simplex $\Delta^{d-1}$ (also sometimes denoted
$\Delta^{d-1}_1$)  under the 
map $f_{(i_1,\dots ,i_d)}$ given by any  reduced word $(i_1,\dots ,i_d)$ for $w$. 
We will  prove that under our contractibility hypothesis  
all of  the conditions needed to apply  Corollary~\ref{tough-topology-theorem} are met.    

Consider  the 
image of $f_{(i_1,\dots ,i_d)}$  endowed with our given stratification (namely with cells being the images of the cells of the simplex where two cells  of the simplex 
$\Delta^{d-1}$ 
have the  same image if and only if their associated words have the same Demazure product).  Recall from \cite{FS} that this stratification of $im(f_{(i_1,\dots ,i_d)})$ 
has closure poset the Bruhat order.  Because we assume $W$ is a finite group,
this is known to be a CW poset (by virtue of the shelling
of Bj\"orner and Wachs from \cite{BW} or of Dyer from \cite{Dy} together  
the fact that it is thin, which is 
clear by virtue of its definition).  We assume by induction the desired result for all reduced words 
strictly shorter than length $d$.  This 
inductive hypothesis ensures that each closed cell in the boundary of 
the image is a ball, and that its stratification resulting from a reduced  subword
of $(i_1,\dots ,i_d)$ of strictly shorter length is a regular CW decomposition.

In particular, this implies that  our stratification restricted to 
the boundary of  $im(f_{(i_1,\dots ,i_d)})$ is a regular CW complex.  This  in turn 
implies that the boundary of the image  is homeomorphic to the order complex of
 its closure poset (after removal of the element $\hat{0}$ representing the empty face).  But  this  is a sphere whose dimension equals the dimension of the complex, by virtue of  Bruhat order being thin and shellable which implies that each open interval of Bruhat order has   order
complex homeomorphic to  such a sphere.  Thus, the restriction of 
$f_{(i_1,\dots ,i_d)}$ to the boundary of the simplex has image a sphere of appropriate dimension.

By a result of Lusztig which is   recalled in Theorem ~\ref{Lu-theorem}, 
the restriction of $f_{(i_1,\dots ,i_d)}$ to the 
interior of the simplex $\Delta^{d-1}$  is a homeomorphism.  
By Lemma ~\ref{open-cell-image-description},  the image of the 
interior of the simplex is nonintersecting with the image of the boundary 
of the simplex.   Finally, note that the preimage of  $f_{(i_1,\dots ,i_d)}$ is a ball by virtue of being a  closed simplex.   Combining with our contractibility hypothesis for fibers, we have  all of the hypotheses needed to apply Corollary ~\ref{tough-topology-theorem}.  Thus we may 
conclude that the image of $f_{(i_1,\dots ,i_d)}$  is a closed  ball.  Since this 
argument works for any reduced word $(i_1,\dots ,i_d)$ for any $w\in W$, this completes the  proof that contractibility of fibers would yield a new proof of the Fomin-Shapiro Conjecture.
%
\end{proof}

\commenth{I added the next remark.   I also added a reference to this remark in the introduction.}

\begin{remark}\label{tying-together} 
By Lemma ~\ref{intersect-with-simplex}, our cell stratification  of $f^{-1}_{(i_1,\dots ,i_d)}(p) \cap \RR_{\ge 0}^d$ given in Theorem ~\ref{decomp-is-strat}  is 
 a cell stratification of $f^{-1}_{(i_1,\dots ,i_d)}(p) \cap \Delta^{d-1}_K$ for some
$K>0$.    If $p\in Y_w^o$, then by definition of $Y_w^o$  this implies $K=1$ and hence $\Delta^{d-1}_K = \Delta^{d-1}$.  
  Conjecture D, a conjecture we put forward in the introduction,  would imply that  this cell stratification is a regular CW decomposition (as Conjecture D  is a slight generalization of this statement).  
 If this  were a regular CW decomposition,  then our combinatorial results in Propositions  ~\ref{p:fiberNerve} and
~\ref{c:dualBlock}   
would combine to  imply the contractibility hypothesis of Theorem ~\ref{contract-implies-fs}.  By Theorem ~\ref{contract-implies-fs}, we could then deduce that
$Y_w$ would be a regular CW complex homeomorphic to a closed ball.    Thus,  a proof of Conjecture D 
would  
yield  a new proof of the Fomin-Shapiro Conjecture. 

  In Section  ~\ref{s:CW}, we proved both that our decomposition of $f^{-1}_{(i_1,\dots ,i_d)}(p) \cap \RR^d_{\ge 0}$  induced by the 
 standard  cell decomposition of $\RR^d_{\ge 0}$  is a 
cell stratification and also 
that many of the cell incidences in this cell stratification are regular incidences, significant steps towards a proof of Conjecture D.
\end{remark}

\section*{Acknowledgments}
The authors thank  Matthew Dyer,  Alex Engstr\"om, Sergey Fomin, Michael Gekhtman,  Steven Karp, Richard Kenyon, Allen Knutson, Bernd Sturmfels, Shmuel Weinberger and Lauren Williams for helpful comments and discussions.

\raggedbottom

\end{document}